\documentclass[10pt,english,reqno]{amsart} 

\usepackage[hidelinks]{hyperref}
\usepackage{graphicx}
\usepackage{amssymb}
\usepackage{epstopdf}
\usepackage{enumerate}
\usepackage{mathabx}
\usepackage{tikz-cd}
\usepackage{MnSymbol}
\usepackage{mathrsfs}
\usepackage{epigraph}

\DeclareMathSymbol{A}{\mathalpha}{operators}{`A}
\DeclareMathSymbol{B}{\mathalpha}{operators}{`B}
\DeclareMathSymbol{C}{\mathalpha}{operators}{`C}
\DeclareMathSymbol{D}{\mathalpha}{operators}{`D}
\DeclareMathSymbol{E}{\mathalpha}{operators}{`E}
\DeclareMathSymbol{F}{\mathalpha}{operators}{`F}
\DeclareMathSymbol{G}{\mathalpha}{operators}{`G}
\DeclareMathSymbol{H}{\mathalpha}{operators}{`H}
\DeclareMathSymbol{I}{\mathalpha}{operators}{`I}
\DeclareMathSymbol{J}{\mathalpha}{operators}{`J}
\DeclareMathSymbol{K}{\mathalpha}{operators}{`K}
\DeclareMathSymbol{L}{\mathalpha}{operators}{`L}
\DeclareMathSymbol{M}{\mathalpha}{operators}{`M}
\DeclareMathSymbol{N}{\mathalpha}{operators}{`N}
\DeclareMathSymbol{O}{\mathalpha}{operators}{`O}
\DeclareMathSymbol{P}{\mathalpha}{operators}{`P}
\DeclareMathSymbol{Q}{\mathalpha}{operators}{`Q}
\DeclareMathSymbol{R}{\mathalpha}{operators}{`R}
\DeclareMathSymbol{S}{\mathalpha}{operators}{`S}
\DeclareMathSymbol{T}{\mathalpha}{operators}{`T}
\DeclareMathSymbol{U}{\mathalpha}{operators}{`U}
\DeclareMathSymbol{V}{\mathalpha}{operators}{`V}
\DeclareMathSymbol{W}{\mathalpha}{operators}{`W}
\DeclareMathSymbol{X}{\mathalpha}{operators}{`X}
\DeclareMathSymbol{Y}{\mathalpha}{operators}{`Y}
\DeclareMathSymbol{Z}{\mathalpha}{operators}{`Z}

\newcommand{\fr}{\mathfrak}
\newcommand{\cal}{\mathscr}

\newcommand{\op}{\operatorname}
\newcommand{\tn}{\textnormal}

\newcommand{\Spc}{\mathrm{Spc}}
\newcommand{\Sptr}{\mathrm{Sptr}}

\newcommand{\Maps}{\mathrm{Maps}}
\newcommand{\Vect}{\mathrm{Vect}}

\newcommand{\Perf}{\mathrm{Perf}}

\DeclareMathOperator*\colim{colim}

\newcommand{\Hom}{\op{Hom}}

\newcommand{\Ext}{\op{Ext}}

\newcommand{\Spec}{\op{Spec}}

\newcommand{\aff}{\mathrm{aff}}

\newcommand{\Frob}{\mathrm{Fr}}
\newcommand{\Mon}{\mathrm{Mon}}
\newcommand{\Grp}{\mathrm{Grp}}

\newcommand{\Quad}{\mathrm{Quad}}
\newcommand{\der}{\mathrm{der}}

\newcommand{\mot}{\mathrm{mot}}

\newcommand{\loo}[1]{(\!(#1)\!)}
\newcommand{\arc}[1]{[\![#1]\!]}

\newcommand{\GL}{\mathrm{GL}}
\newcommand{\SL}{\mathrm{SL}}

\newcommand{\Rad}{\mathrm{Rad}}

\newcommand{\Gr}{\mathrm{Gr}}
\newcommand{\Bun}{\mathrm{Bun}}

\newcommand{\Ran}{\mathrm{Ran}}

\newcommand{\ad}{\mathrm{ad}}

\newcommand{\pr}{\mathrm{pr}}
\newcommand{\id}{\mathrm{id}}

\newcommand{\act}{\mathrm{act}}

\newcommand{\exch}{\mathrm{exch}}

\newcommand{\Hec}{\mathrm{Hec}}

\newcommand{\super}{\mathrm{super}}
\newcommand{\Sq}{\mathrm{Sq}}
\newcommand{\Tate}{\mathrm{Tate}}
\newcommand{\rank}{\mathrm{rank}}
\newcommand{\tame}{\mathrm{tame}}

\newcommand{\fact}{\mathrm{fact}}
\newcommand{\Fl}{\mathrm{Fl}}
\newcommand{\inner}{\mathrm{int}}
\newcommand{\spin}{\mathrm{spin}}
\newcommand{\Pf}{\mathrm{Pf}}
\newcommand{\Ad}{\mathrm{Ad}}
\newcommand{\deloop}{\textnormal{\textsf{B}}}
\newcommand{\Ktheory}{\underline K}
\newcommand{\Pic}{\textnormal{\textsf{Pic}}}

\newtheorem{thm}[subsubsection]{Theorem}
\newtheorem*{thm*}{Theorem}
\newtheorem{thmx}{Theorem}

\newtheorem{corx}[thmx]{Corollary}
\newtheorem{prop}[subsubsection]{Proposition}
\newtheorem{lem}[subsubsection]{Lemma}

\newtheorem{cor}[subsubsection]{Corollary}

\theoremstyle{definition}
\newtheorem*{rem*}{Remark}

\newtheorem{eg}[subsubsection]{Example}
\newtheorem{rem}[subsubsection]{Remark}

\numberwithin{equation}{section}

\newtheorem{untitledsubsubsection}[subsubsection]{}

\makeatletter
\renewcommand{\@secnumfont}{\bfseries}
\makeatother

\newenvironment{void}
{\begin{untitledsubsubsection}}
{\end{untitledsubsubsection}}

\title{Half-integral levels}

\dedicatory{To James Tao}
\author{Yifei Zhao}
\date{\today}

\begin{document}

\thanks{
The project was funded by the Deutsche Forschungsgemeinschaft (DFG, German Research Foundation) Project-ID 427320536 -- SFB 1442, as well as under Germany's Excellence Strategy EXC 2044 390685587, Mathematics Münster: Dynamics--Geometry--Structure.
}

\begin{abstract}
We construct equivalences among four notions associated to a reductive group scheme $G$: factorization super central extensions of the loop group of $G$ by $\mathbb G_m$ subject to a condition on the commutator, factorization super line bundles on the affine Grassmannian of $G$, rigidified sections of a quotient of $2$-truncated $\textnormal K$-theory over the Zariski classifying stack of $G$, and combinatorial data defined by Brylinski and Deligne in a conjectural extension of their classification theorem.
\end{abstract}

\maketitle


\setcounter{tocdepth}{2}
\tableofcontents

\setlength\epigraphwidth{.5\textwidth}
\setlength\epigraphrule{0pt}
\epigraph{
\itshape ``Is this what it feels like to end?''

``I do not know, for this is not our end.''
}{---Kindred}




\section*{Introduction}

Let $k\loo{t}$ denote the field of formal Laurent series with coefficients in a field $k$. Tate \cite{MR227171} discovered a remarkable central extension of its group of units $k\loo{t}^{\times}$ by $k^{\times}$:
\begin{equation}
\label{eq-tate-central-extension-first}
1 \rightarrow k^{\times} \rightarrow \cal G_{\Tate} \rightarrow k\loo{t}^{\times} \rightarrow 1.
\end{equation}
The preimage of $a \in k\loo{t}^{\times}$ in $\cal G_{\Tate}$ consists of nonzero elements of the relative determinant line $\det(ak\arc{t} \mid k\arc{t})$, where $k\arc{t} \subset k\loo{t}$ is the lattice of formal Taylor series.

If we think of $k\loo{t}^{\times}$ as the $k\loo{t}$-points of the algebraic group $\mathbb G_m$, then the following question arises: for a reductive group $G$, what are the ``natural'' central extensions of $G(k\loo{t})$ by $k^{\times}$?

Brylinski and Deligne \cite{MR1896177} parametrized a large class of central extensions of $G(k\loo{t})$ by $k^{\times}$ using $\tn K$-theory, as follows. Denote by $\Ktheory{}_2$ the Zariski sheafification of the second algebraic $\tn K$-group. Starting with a central extension on the big Zariski site of $\Spec(k)$:
\begin{equation}
\label{eq-brylinski-deligne-extension}
1 \rightarrow \Ktheory{}_2 \rightarrow E \rightarrow G \rightarrow 1,
\end{equation}
evaluating at $\Spec(k\loo{t})$ and pushing out along the tame symbol $\Ktheory{}_2(k\loo{t}) \rightarrow k^{\times}$, we find a central extension of $G(k\loo{t})$ by $k^{\times}$. They went on to give a complete classification of central extensions \eqref{eq-brylinski-deligne-extension}, valid over any regular base scheme of finite type over a field \cite[Theorem 7.2]{MR1896177}. However, no central extension of $\mathbb G_m$ by $\Ktheory{}_2$ produces Tate's central extension \eqref{eq-tate-central-extension-first}. This led Brylinski and Deligne to pose \cite[Questions 12.13(iii)]{MR1896177}:

\medskip

\emph{``For $V = k\arc{t}$, not all natural central extensions by $k^{\times}$ are captured by 12.8. \tn{[...]} We expect that `natural' central extensions of $G(K)$ by $k^{\times}$ are attached to data as follows: a Weyl group and Galois group invariant integer-valued symmetric bilinear form \tn{[...]}''}

\medskip

Our first goal is to find an enlargement of the groupoid of central extensions of $G$ by $\Ktheory{}_2$ and prove that it meets the expectation of Brylinski and Deligne. To this end, we introduce a Zariski sheaf of connective spectra $\Ktheory{}_{[1, 2]}^{\super}$ and establish the following result.

\begin{thmx}[Theorem \ref{thm-classification-super}]
\label{thmx-brylinski-deligne}
Let $X$ be a regular scheme of finite type over a field and $G$ be a reductive group $X$-scheme equipped with a maximal torus $T$. The following Picard groupoids are canonically equivalent:
\begin{enumerate}
	\item rigidified sections of $\Ktheory{}_{[1, 2]}^{\super}$ over the Zariski classifying stack $BG$;
	\item triples $(b, \widetilde{\Lambda}, \varphi)$ defined in \cite[Questions 12.13(iii)]{MR1896177}.
\end{enumerate}
\end{thmx}

Given a $k\loo{t}$-point of $X$, it is straightforward to produce from a rigidified section of $\Ktheory{}_{[1, 2]}^{\super}$ over $BG$ a central extension of $G(k\loo{t})$ by $k^{\times}$. In fact, the result will carry a canonical $\mathbb Z/2$-grading, hence a ``super central extension''. This includes Tate's central extension in the special case $G = \mathbb G_m$.

The second goal of this article is to prove that this passage from $\tn K$-theory to super central extensions of the loop group is \emph{reversible} if one remembers an additional piece of structure called ``factorization''.

As observed by Beilinson and Drinfeld \cite{MR2058353, MR2181808}, the construction of $\cal G_{\Tate}$ globalizes, over any smooth curve $X$, to a \emph{factorization super central extension} of the formal loop group $\cal L\mathbb G_m$. Intuitively speaking, this additional structure describes the behavior of $\cal G_{\Tate}$ as one formal loop on $X$ ``factorizes'' into two. The following result shows that factorization super central extensions of the loop group $\cal LG$, subject to a ``tame commutator'' condition which is automatic in characteristic zero, admit a parametrization parallel to Theorem \ref{thmx-brylinski-deligne}.

\begin{thmx}[Theorem \ref{thm-factorization-super-central-extension-classification}]
\label{thmx-loop-group}
Let $X$ be a smooth curve over a field and $G$ be a reductive group $X$-scheme. The following Picard groupoids are canonically equivalent:
\begin{enumerate}
	\item factorization super central extensions of $\cal LG$ by $\mathbb G_m$ with tame commutator;
	\item[(1')] factorization super line bundles over the affine Grassmannian $\Gr_G$;
	\item triples $(b, \widetilde{\Lambda}_+, \varphi)$ defined in \cite[Questions 12.13(iii)]{MR1896177} up to a ``twist''---if $G$ is equipped with a maximal torus.
\end{enumerate}
\end{thmx}

Upon choosing a $\vartheta$-characteristic, \emph{i.e.}~a square root $\omega^{1/2}$ of the canonical line bundle of the smooth curve $X$, the ``twist'' mentioned in (2) disappears. The Picard groupoids in Theorem \ref{thmx-brylinski-deligne} then become canonically equivalent to those in Theorem \ref{thmx-loop-group}, forming a commutative diagram:
\begin{equation}
\label{eq-intro-all-equivalences}
\begin{tikzcd}
	A(1) \ar[r, "\cong"]\ar[dd, "\cong"] & B(1) \ar[d, "\cong"] \\
	& B(1') \ar[d, "\cong"] \\
	A(2) \ar[r, "\cong"] & B(2)
\end{tikzcd}
\end{equation}

In fact, the equivalences in \eqref{eq-intro-all-equivalences} are the ``half-integral'' generalizations of a family of equivalences which are valid without the choice of a $\vartheta$-characteristic.

\begin{corx}[Corollary \ref{cor-factorization-central-extension-classification}]
\label{corx-no-spin}
Let $X$ be a smooth curve over a field and $G$ be a reductive group $X$-scheme. The following Picard groupoids are canonically equivalent:
\begin{enumerate}
	\item central extensions of $G$ by $\Ktheory{}_2$ on the big Zariski site of $X$;
	\item factorization central extensions of $\cal LG$ by $\mathbb G_m$ with tame commutator;
	\item factorization line bundles over $\Gr_G$;
	\item triples $(Q, \widetilde{\Lambda}, \varphi)$ in \cite[Theorem 7.2]{MR1896177}---if $G$ is equipped with a maximal torus.
\end{enumerate}
\end{corx}

This corollary already improves the current state of knowledge. Indeed, an equivalence between the Picard groupoids (1) and (3) was conjectured in Gaitsgory--Lysenko \cite{MR3769731} and established in \cite{MR4117995, MR4322626} under the additional assumptions that $G$ is split and a certain integer $N_G$ is invertible in the ground field. The equivalence supplied by Corollary \ref{corx-no-spin} is valid for any reductive group $X$-scheme.

In the literature on covering groups in the equicharacteristic setting, the existence of factorization (super) line bundles over the affine Grassmannian $\Gr_G$ with favorable properties is sometimes stated as an assumption, see \emph{e.g.}~\cite{lysenko2016twisted, MR3626614} and \cite[\S14]{MR3787407}. The combination of Theorems \ref{thmx-brylinski-deligne} and \ref{thmx-loop-group} produces them unconditionally.

It is worth mentioning that Theorem \ref{thmx-loop-group} is nontrivial already for $G = \mathbb G_m$. Indeed, fibers of the affine Grassmannian $\Gr_{\mathbb G_m}$ over geometric points of $X$ are highly nonreduced formal schemes. The groupoid of (super) line bundles over $\Gr_{\mathbb G_m}$ does not appear to have a clean description, but the equivalence (1') $\cong$ (2) of Theorem \ref{thmx-loop-group} shows that the factorization ones do. Moreover, the equivalence (1) $\cong$ (1') shows that factorization (super) line bundles over $\Gr_G$ have canonical multiplicative structures over $\cal LG$. Unless $G$ is simply connected, this assertion is not an obvious consequence of existing results.

From a differential geometric perspective, one could trace the conceptual origin of Corollary \ref{corx-no-spin} to works on Chern--Simons theory. Indeed, Dijkgraaf and Witten \cite{MR1048699} first recognized that the quantization parameter, or integral ``level'', of Chern--Simons theory for a compact Lie group $G$ is best understood as an element of the reduced cohomology group $H^4_e(BG, \mathbb Z)$. Suitably categorified, such an element transgresses to a central extension of the loop group of $G$ by $U(1)$. A recent theorem of Waldorf \cite{MR3648501} showed that this transgression procedure is reversible if one remembers the ``fusion'' structure of the target.

In the algebraic context, $H^4_e(BG, \mathbb Z)$ should be replaced by the reduced weight-$2$ motivic cohomology group of $BG$, which classifies central extensions of $G$ by $\Ktheory{}_2$ via the isomorphism of \cite{MR1460391} (see also \cite[Theorem 6.3.5]{MR4117995}):
\begin{equation}
\label{eq-eklv}
H^4_e(BG, \mathbb Z_{\mot}(2)) \xrightarrow{\simeq} H^2_e(BG, \Ktheory{}_2).
\end{equation}
Hence, a central extension of $G$ by $\Ktheory{}_2$ can be thought of as the algebraic notion of an \emph{integral level} and Corollary \ref{corx-no-spin} provides four equivalent descriptions of it.\footnote{This algebraic notion is naturally associated to the \emph{chiral} Wess--Zumino--Witten (WZW) model. We also mention that Henriques \cite{MR3709708} proposed a definition of integral levels for the chiral WZW model via vertex algebras, while our notion is more directly related to chiral algebras, see \cite{MR2058353, rozenblyum2021connections}.} The equivalence (1) $\cong$ (2) of Corollary \ref{corx-no-spin} is a direct analogue of Waldorf's theorem.

With this understanding, we propose to encode the algebraic notion of a \emph{half-integral level} by the Picard groupoids in \eqref{eq-intro-all-equivalences}. In fact, Dijkgraaf and Witten \cite[\S5]{MR1048699} already observed that on \emph{spin} manifolds, formally dividing a class in $H^4_e(BG, \mathbb Z)$ by $2$ sometimes leads to physically meaningful quantities. To interpret these ``half-integral characteristic classes'' as rigidified sections of $\Ktheory{}_{[1, 2]}^{\super}$ over $BG$, we note that the natural inclusion of abelian groups below has a $2$-torsion cokernel:
\begin{equation}
\label{eq-short-exact-sequence-levels}
H^2_e(BG, \Ktheory{}_2) \subset \pi_0\Gamma_e(BG, \Ktheory{}_{[1, 2]}^{\super}).
\end{equation}

In the example of Tate's central extension, we have the equality $2 \cdot [\Tate] = [c_1]^2$, where $[c_1]$ denotes the first Chern class of the universal line bundle over $B\mathbb G_m$, so $[c_1]^2$ generates the abelian group $H^2_e(B\mathbb G_m, \Ktheory{}_2)$, while $[\Tate]$ is half-integral.

Another example is the ``critical level'', \emph{i.e.}~Beilinson and Drinfeld's Pfaffian \cite[\S4]{beilinson1991quantization}, representing half of $[c_2]$ of the adjoint bundle over $BG$. It is half-integral precisely when the half sum of positive roots $\check{\rho}$ is not an integral weight. Indeed, one of the applications of Theorem \ref{thmx-brylinski-deligne} is that it gives a new construction of the Pfaffian line bundle on the moduli stack of $G$-bundles on a spin curve (see \S\ref{sec-examples}).

Half-integral levels in our sense give rise to super conformal blocks on spin curves, as predicted by \cite{MR1048699}, although we do not attempt to fully develop this notion here.

Let us now explain the structure of this article and comment on the proofs.

\subsection*{Structure of the article}

This article is divided into two parts which can be read independently. The first part proves Theorem \ref{thmx-brylinski-deligne} and the second part proves Theorem \ref{thmx-loop-group}.

In \S\ref{sec-k-theory-super}, we define $\Ktheory{}_{[1, 2]}^{\super}$ using a small but essential amount of homotopy theory. Namely, it is set to be the cofiber of a morphism of Zariski sheaves of connective spectra:
\begin{equation}
\label{eq-squaring-map-intro}
\Sq : B\Ktheory{}_1 \rightarrow \Ktheory{}_{[1, 2]}.
\end{equation}
Here, $\Ktheory{}_1$ and $\Ktheory{}_{[1, 2]}$ are the Zariski sheafified truncations of the $\tn K$-theory spectrum. We also explain how to integrate sections of $\Ktheory{}_{[1, 2]}^{\super}$ over a global spin curve.

In \S\ref{sec-brylinski-deligne}, we prove Theorem \ref{thmx-brylinski-deligne}. The proof combines \cite[Theorem 7.2]{MR1896177} with our description of $\Ktheory{}_{[1, 2]}$ obtained in \S\ref{sec-k-theory-super}.

In \S\ref{sect-factorization}, we formulate Theorem \ref{thmx-loop-group}. To define the notion of ``tame commutator'', we make essential use of the Contou-Carr\`ere symbol over the Ran space, as constructed in Campbell--Hayash \cite{campbell2021geometric}. One of the phenomena we observe here is that the condition of having ``tame commutator'' is automatic in characteristic zero. This fact turns out to be equivalent to a new universality statement for the Contou-Carr\`ere symbol.

\begin{corx}[Corollary \ref{cor-contou-carrere-universal}]
\label{corx-tame-symbol}
Let $X$ be a smooth curve over a field $k$ with $\mathrm{char}(k) = 0$. Then any pairing $\cal L\mathbb G_m \otimes \cal L\mathbb G_m \rightarrow \mathbb G_m$ compatible with factorization is an integral power of the Contou-Carr\`ere symbol.
\end{corx}

We deduce this corollary from a surprising theorem of Tao \cite{tao2021mathrmgrg}, which asserts that the presheaf $\Gr_{\mathbb G_m}$ over the Ran space is \emph{reduced} in a suitable sense, provided $\mathrm{char}(k) = 0$. The assertion of Corollary \ref{corx-tame-symbol} is false if $\mathrm{char}(k) > 0$. We do not use it in the proof of Theorem \ref{thmx-loop-group}, which is valid in arbitrary characteristics.

In \S\ref{sect-factorization-proofs}, we prove Theorem \ref{thmx-loop-group}. Our strategy is to first construct functors among the Picard groupoids in Theorem \ref{thmx-loop-group}:
\begin{equation}
\label{eq-factorization-strategy}
(1) \rightarrow (1') \rightarrow (2).
\end{equation}
Our previous work \cite{MR4322626} shows that the second functor is fully faithful. Here, we prove that the composition \eqref{eq-factorization-strategy} is an equivalence by exploiting the group structure inherent in (1). In our approach, each of the equivalences (1) $\cong$ (2), (1') $\cong$ (2) is established using special cases of the other in iteration, so we do not obtain one without the other.

Finally, we mention a shortcoming of this article: the top horizontal functor appearing in \eqref{eq-intro-all-equivalences} is defined \emph{ad hoc} as the composition of the other functors. It should have a conceptually transparent description as a ``transgression'' along the space of formal loops:
\begin{equation}
\label{eq-transgression-intro}
\int_{(\mathring D, \omega^{1/2})} : \Gamma_e(BG, \Ktheory{}_{[1, 2]}^{\super}) \rightarrow \Hom(\cal LG, \Pic^{\super}),
\end{equation}
as in the differential geometric context, but we are unable to find such a description. One difficulty seems to be that we do not understand the behavior of Zariski-sheafified $\tn K$-groups over singular spaces such as $\cal LG$. An attempt at defining \eqref{eq-transgression-intro} as a ``transgression'' in the integral case, \emph{i.e.}~for sections of $\Ktheory{}_2[2]$ over $BG$, was made in Kapranov--Vasserot \cite{MR2332353}, but it relies on \cite[Proposition 4.2.1]{MR2332353} which is false as stated. A different strategy was carried out in Gaitsgory \cite{MR4117995}, but it requires the hypothesis that $N_G$ be invertible in the ground field, which we wish to avoid.

\subsection*{Acknowledgements}

James Tao has made the most important contibution to all problems considered in this article. This includes not only his published works on this topic \cite{tao2021mathrmgrg, MR4198527, MR4322626}, but also numerous ideas communicated to me during our collaboration. It is with profound gratitude and humility that I dedicate this article to him.

I thank Michael Finkelberg for fruitful e-mail exchanges and for his interest in Theorem \ref{thmx-loop-group} in relation to \cite{braverman2022coulomb}. I thank Jo\~{a}o Louren\c{c}o for the proof of Lemma \ref{lem-schubert-variety-properties}.

Finally, I thank Dennis Gaitsgory for initiating me into factorization structures and for teaching me many things along the way.

\part{$\tn K$-theory}

\section{$\Ktheory{}_{[1,2]}^{\super}$}
\label{sec-k-theory-super}

The main goal of this section is to introduce the Zariski sheaf of connective spectra $\Ktheory{}_{[1, 2]}^{\super}$. The first section \S\ref{sec-connective-k-theory} reviews necessary notions concerning algebraic $K$-theory. In \S\ref{sec-truncated-k-theory}, we give a ``hands-on'' description of the truncation $\Ktheory{}_{[1, 2]}$. Using this description, we are able to define $\Ktheory{}_{[1, 2]}^{\super}$ in \S\ref{sec-super-k-theory}. The material of \S\ref{sec-integration-on-curves} is not needed in the sequal: its goal is to show that sections of $\Ktheory{}_{[1, 2]}^{\super}$ can be integrated over a global spin curve relative to a regular base scheme $S$ to yield a super line bundle over $S$.

\subsection{Connective $\tn K$-theory}
\label{sec-connective-k-theory}

\begin{void}
Let $\Spc$ denote the $\infty$-category of spaces. It is a symmetric monoidal $\infty$-category under the Cartesian product.

Write $\Mon_{\mathbb E_{\infty}}(\Spc)$ for the $\infty$-category of $\mathbb E_{\infty}$-monoids in $\Spc$. It contains a full subcategory $\Grp_{\mathbb E_{\infty}}(\Spc)$ consisting of grouplike $\mathbb E_{\infty}$-monoids. The forgetful functor $\Grp_{\mathbb E_{\infty}}(\Spc) \rightarrow \Mon_{\mathbb E_{\infty}}(\Spc)$ admits a left adjoint, called \emph{group completion}:
\begin{equation}
\label{eq-group-completion}
	\Omega \deloop : \Mon_{\mathbb E_{\infty}}(\Spc) \rightarrow \Grp_{\mathbb E_{\infty}}(\Spc).
\end{equation}

Let $\Sptr$ denote the $\infty$-category of spectra. We use homotopical grading and denote by $\Sptr_{\ge 0}$ the full subcategory of connective spectra.

There is a canonical equivalence of $\infty$-categories \cite[Remark 5.2.6.26]{lurie2017higher}:
\begin{equation}
\label{eq-symmetric-monoid-to-spectra}
\Grp_{\mathbb E_{\infty}}(\Spc) \cong \Sptr_{\ge 0}.
\end{equation}

We shall also use without explicit mention the equivalence between Picard groupoids and $1$-truncated connective spectra.
\end{void}

\begin{void}
Let $R$ be a commutative ring. Denote by $\Vect(R)$ the category of finitely generated projective $R$-modules and $\Vect(R)^{\cong}$ its maximal subgroupoid. The operation of direct sum equips $\Vect(R)^{\cong}$ with a symmetric monoidal structure. Its image under \eqref{eq-group-completion} is by definition the \emph{connective $\tn K$-theory} $K(R)$ of $R$.

We shall view $K(R)$ either as a grouplike $\mathbb E_{\infty}$-monoid or as a connective spectrum, using the canonical equivalence \eqref{eq-symmetric-monoid-to-spectra}.

Note that the unit of the adjunction between \eqref{eq-group-completion} and the forgetful functor supplies a morphism of $\mathbb E_{\infty}$-monoids:
\begin{equation}
\label{eq-class-of-perfect-complexes}
\Vect(R)^{\cong} \rightarrow K(R),\quad \cal E \mapsto [\cal E].
\end{equation}
\end{void}

\begin{void}
We equip $\Mon_{\mathbb E_{\infty}}(\Spc)$ and $\Grp_{\mathbb E_{\infty}}(\Spc)$ with the canonical symmetric monoidal structure of \cite[Theorem 5.1]{MR3450758}. With respect to these symmetric monoidal structures, \eqref{eq-group-completion} is symmetric monoidal. Hence it lifts to a functor of $\mathbb E_{\infty}$-monoids:
\begin{equation}
\label{eq-group-completion-algebra}
\Omega \deloop : \Mon_{\mathbb E_{\infty}}(\Mon_{\mathbb E_{\infty}}(\Spc)) \rightarrow \Mon_{\mathbb E_{\infty}}(\Grp_{\mathbb E_{\infty}}(\Spc)).
\end{equation}

The right adjoint of \eqref{eq-group-completion}, being lax symmetric monoidal, also lifts to a functor of $\mathbb E_{\infty}$-monoids and supplies the right adjoint of \eqref{eq-group-completion-algebra}, see \cite[Lemma 3.6]{MR3450758}.

The operation of tensor product upgrades $\Vect(R)^{\cong}$ into an $\mathbb E_{\infty}$-monoid in $\Mon_{\mathbb E_{\infty}}(\Spc)$. Thus $K(R)$ acquires an $\mathbb E_{\infty}$-monoid structure in $\Grp_{\mathbb E_{\infty}}(\Spc)$ (\emph{i.e.}~$K(R)$ is a connective $\mathbb E_{\infty}$-spectrum) such that the unit \eqref{eq-class-of-perfect-complexes} is symmetric monoidal.
\end{void}

\begin{rem}
Informally, the symmetric monoidal structure on \eqref{eq-class-of-perfect-complexes} says that for each pair of objects $\cal E_1, \cal E_2\in\Vect(R)^{\cong}$, $[\cal E_1\otimes \cal E_2]$ is canonically equivalent to $[\cal E_1]\cdot [\cal E_2]$, together with the homotopy coherence data.
\end{rem}

\begin{void}
For any integer $a$, we write $K_{\ge a}(R)$ (resp.~$K_{\le a}(R)$) for the truncation $\tau_{\ge a}K(R)$ (resp.~$\tau_{\le a}K(R)$). For a pair of integers $a\le b$, we write $K_{[a, b]} := \tau_{\ge a}\tau_{\le b}K(R)$. We also use $K_a(R)$ to denote $\Omega^aK_{[a, a]}(R) \cong \pi_aK(R)$.

The association $S = \Spec(R)\mapsto K(R)$ defines a presheaf $K$ of connective $\mathbb E_{\infty}$-spectra on the category of affine schemes. Let $\Ktheory$ denote its sheafification in the Zariski topology.

Zariski sheafification of the truncated presheaves above define $\Ktheory{}_{\ge a}$, $\Ktheory{}_{\le a}$, $\Ktheory{}_{[a, b]}$, and $\Ktheory{}_a$. Since sheafification is $t$-exact, the forgetful functor from presheaves of spectra to sheaves of spectra is left $t$-exact. Hence $\Ktheory{}_{\le a}$ is $a$-truncated as a presheaf of spectra, \emph{i.e.}~its value at any $R$ has vanishing homotopy groups above degree $a$.
\end{void}

\begin{rem}
\label{rem-determinant-via-k-theory}
For example, the map sending $\cal E\in\Vect(R)^{\cong}$ to its determinant line bundle $\det(\cal E)$ induces an isomorphism of sheaves of Picard groupoids:
$$
\Ktheory{}_{[0, 1]} \xrightarrow{\simeq} \Pic^{\mathbb Z},
$$
where $\Pic^{\mathbb Z}$ sends $R$ to the Picard groupoid of $\mathbb Z$-graded line bundles on $\Spec(R)$, see \cite[Proposition 12.18]{MR3674218}.
\end{rem}

\subsection{The sheaf $\Ktheory{}_{[1, 2]}$}
\label{sec-truncated-k-theory}

\begin{void}
The goal of this subsection is to give an explicit description of $\Ktheory{}_{[1, 2]}$.

More precisely, we consider the fiber sequence defined by truncation:
$$
\deloop^2K_2(R) \rightarrow K_{[1, 2]}(R) \rightarrow \deloop K_1(R)
$$
for each ring $R$, which induces a fiber sequence of Zariski sheaves of connective spectra:
\begin{equation}
\label{eq-truncated-k-theory-fiber-sequence}
	\deloop^2\Ktheory{}_2 \rightarrow \Ktheory{}_{[1,2]} \rightarrow \deloop \Ktheory{}_1.
\end{equation}

Our description will be that of the fiber sequence \eqref{eq-truncated-k-theory-fiber-sequence}.
\end{void}

\begin{void}
\label{void-truncated-k-theory-section}
Denote by $\Pic(R)\subset \Vect(R)^{\cong}$ the full subcategory of line bundles over $\Spec(R)$.

The map \eqref{eq-class-of-perfect-complexes} induces a morphism of the underlying \emph{pointed spaces} $\Pic(R) \rightarrow K(R)$ sending $\cal L$ to $[\cal L] - [\cal O]$. Thus, we obtain a morphism of Zariski sheaves of pointed spaces, without changing the same notation:
\begin{equation}
\label{eq-section-of-k-theory-from-pic}
\Pic \rightarrow \Ktheory,\quad \cal L \mapsto [\cal L] - [\cal O].
\end{equation}

Since the class of $[\cal L] - [\cal O]$ in $K_0(R)$ vanishes Zariski locally on $\Spec(R)$, the morphism \eqref{eq-section-of-k-theory-from-pic} factors through $\Ktheory{}_{\ge 1}$ and we may compose it with the truncation map $\Ktheory{}_{\ge 1} \rightarrow \Ktheory{}_{[1,2]}$ to obtain a map of Zariski sheaves of pointed spaces:
\begin{equation}
\label{eq-section-of-truncated-k-theory-from-pic}
	s : \Pic \rightarrow \Ktheory{}_{[1, 2]}.
\end{equation}

The description of $\Ktheory{}_{[0, 1]}$ via determinant (Remark \ref{rem-determinant-via-k-theory}) shows that \eqref{eq-section-of-truncated-k-theory-from-pic} is a section of \eqref{eq-truncated-k-theory-fiber-sequence} on the underlying sheaves of pointed spaces, \emph{i.e.}~the composition of \eqref{eq-section-of-truncated-k-theory-from-pic} with the truncation map $\Ktheory{}_{[1,2]} \rightarrow \deloop\Ktheory{}_1$ is the canonical isomorphism $\Pic \xrightarrow{\simeq} \deloop\Ktheory{}_1$.
\end{void}

\begin{void}
\label{void-fiber-sequence-classification}
Let $\cal C$ be a site and $A_1$, $A_2$ be sheaves of abelian groups over $\cal C$.

Consider the following two groupoids:
\begin{enumerate}
	\item the groupoid of extensions $A$ of $\deloop(A_1)$ by $\deloop^2(A_2)$ as sheaves of connective spectra, equipped with a section $s : \deloop(A_1) \rightarrow A$ of the underlying sheaves of pointed spaces;
	\item the (discrete) groupoid of anti-symmetric pairings $A_1 \otimes A_1 \rightarrow A_2$.
\end{enumerate}
Let us construct a functor from (1) to (2):
\begin{equation}
\label{eq-fiber-sequence-classification}
	\begin{Bmatrix}
	\begin{tikzcd}[column sep = 1.5em]
	& & \deloop (A_1) \ar[d, "\id"]\ar[dl, swap, "\textnormal{pointed}"] \\
	\deloop^2(A_2) \ar[r] & A \ar[r] & \deloop(A_1)
	\end{tikzcd}
	\end{Bmatrix}
	\rightarrow
	\begin{Bmatrix}
	\textnormal{anti-symmetric}\\
	A_1\otimes A_1 \rightarrow A_2
	\end{Bmatrix}.
\end{equation}

Indeed, the section $s$ defines a ``cocycle'' morphism of sheaves of spaces:
\begin{equation}
\label{eq-cocycle-formula}
\deloop(A_1) \times \deloop(A_1) \rightarrow \deloop^2(A_2),\quad (x, y) \mapsto s(x + y) - s(x) - s(y).
\end{equation}
The morphism \eqref{eq-cocycle-formula} is \emph{bi-rigidified} in the following sense: it is equipped with trivializations along $\deloop(A_1) \times e$ and $e\times \deloop(A_1)$, which are isomorphic over $e\times e$.

To such a morphism, we may apply the loop space functor in the first, then the second factor, to obtain a map:
\begin{equation}
\label{eq-loop-space-pairing}
A_1 \times A_1 \rightarrow A_2.
\end{equation}
\end{void}

\begin{lem}
\label{lem-fiber-sequence-classification}
The map \eqref{eq-loop-space-pairing} is bilinear and anti-symmetric. The resulting functor \eqref{eq-fiber-sequence-classification} is an equivalence of groupoids.
\end{lem}

\begin{void}
\label{void-fiber-sequence-classification-alternative}
The proof of Lemma \ref{lem-fiber-sequence-classification} proceeds by first giving an alternative definition of the functor \eqref{eq-fiber-sequence-classification} which is evidently an equivalence, and then showing that it coincides with the construction above.

First, we observe that groupoid (1) is equivalent to the groupoid (1') of extensions $A'$ of $A_1$ by $\deloop(A_2)$ as sheaves of connective spectra, equipped with an $\mathbb E_1$-monoidal section $A_1 \rightarrow A'$. The equivalence is given by the functors $\Omega$ and $\deloop$.

Put differently, the groupoid (1') consists of symmetric monoidal structures on the sheaf of associative monoids $\deloop(A_2) \times A_1$ such that the inclusion and projection functors:
$$
\deloop(A_2) \subset \deloop(A_2) \times A_1 \rightarrow A_1
$$
are symmetric monoidal.

Such symmetric monoidal structures are in turn given by commutativity constraints on $\deloop(A_2)\times A_1$ vanishing on $\deloop(A_2)$, which are anti-symmetric bilinear pairings:
\begin{equation}
\label{eq-fiber-sequence-classification-braiding}
A_1 \otimes A_1 \rightarrow A_2.
\end{equation}
Indeed, the commutativity constraint for two objects $x, y \in \deloop(A_2)\times A_1$ is an isomorphism $x\otimes y\xrightarrow{\simeq} y\otimes x$ which depends only on the classes of $x, y$ in $A_1$. Such an isomorphism is the multiplication by a unique element of $A_2$. The hexagon and inverse axioms translate to the bilinearity and anti-symmetry of the resulting map $A_1\times A_1 \rightarrow A_2$.

The procedure above establishes an equivalence between the groupoid (1) and such pairings. Hence, it remains to prove that the pairing \eqref{eq-fiber-sequence-classification-braiding} extracted from any object of the groupoid (1) concides with \eqref{eq-loop-space-pairing}. This follows from the observation in topology below.
\end{void}

\begin{void}
\label{void-monoid-loop-spaces}
Let $M$, $N$ be $\mathbb E_1$-monoids in $\Spc$ and let $s : M \rightarrow N$ be a morphism of pointed spaces. Then we may construct two maps of spaces:
\begin{equation}
\label{eq-monoid-loop-spaces}
\begin{tikzcd}
\Omega(M) \times \Omega(M)\ar[r, shift left = 0.5ex, "s_1"]\ar[r, shift left = -0.5ex, swap, "s_2"] & \Omega^2(N).
\end{tikzcd}
\end{equation}

The map $s_1$ is given by applying $\Omega$ to the first, then the second factor, to the bi-rigidifed map of spaces:
\begin{equation}
\label{eq-pointed-morphism-cocycle}
M\times M \rightarrow N,\quad (x, y) \mapsto s(x + y) - s(x) - s(y).
\end{equation}

The map $s_2$ uses the $\mathbb E_1$-monoidal morphism $\Omega_s : \Omega(M) \rightarrow \Omega(N)$ induced from $s$ and sends two loops $a, b \in \Omega(M)$ to the following loop in $\Omega(N)$:
\begin{align}
\notag
\mathbf 1 &\xrightarrow{\simeq} \Omega_s(a + b) - (\Omega_s(a) + \Omega_s(b)) \\
\label{eq-loop-by-braiding}
&\xrightarrow{\simeq} \Omega_s(b + a) - (\Omega_s(b) + \Omega_s(a)) \xrightarrow{\simeq} \mathbf 1.
\end{align}
where the first and last isomorphisms are defined by the $\mathbb E_1$-monoid structure on $\Omega_s$ and the middle one is defined by the braidings\footnote{We use the term ``braiding'' to refer to the isomorphism $a \otimes b\xrightarrow{\simeq} b \otimes a$ in an $\mathbb E_2$-monoid and the term ``commutativity constraint'' to refer to the same isomorphism in an $\mathbb E_{\infty}$-monoid.} on $\Omega(M)$ and $\Omega(N)$.

\emph{Claim}: there is a homotopy equivalence $s_1 \xrightarrow{\simeq} s_2$.

To prove the claim, we choose as our model for ``spaces'' topological spaces having the homotopy type of a CW complex, see \cite[012Z]{kerodon} for its equivalence with the standard model using Kan complexes.

The monoidal operation on $\Omega(M)$ can be viewed as concatenation of loops and the braiding is given as follows. For two loops $a, b\in\Omega(M)$ we find a morphism:
\begin{equation}
\label{eq-braiding-loop-space}
[0, 1]\times [0, 1] \rightarrow M,\quad (t_1, t_2) \mapsto a(t_1) + b(t_2),
\end{equation}
where the sum is the $\mathbb E_1$-monoid product on $M$. Then \eqref{eq-braiding-loop-space} can be viewed as a homotopy from its restriction to $([0, 1]\times\{0\}) \cup (\{1\}\times [0, 1])$ to its restriction to $(\{0\}\times[0, 1])\cup ([0,1]\times\{1\})$, exhibiting the braiding $a + b \xrightarrow{\simeq} b + a$ in $\Omega(M)$.

The same description holds for the braiding in $\Omega(N)$.

Now we come to the morphism $s_1$. It carries $a, b$ to the element of $\Omega^2(N)$ which is represented by a map $S^2 \rightarrow N$ fitting into the following diagram:
\begin{equation}
\label{eq-double-loop-homotopy}
\begin{tikzcd}[column sep = 1.5em]
	{[0, 1]}\times {[0,1]} \ar[r, "{(a, b)}"]\ar[d] & M \times M \ar[d, "\eqref{eq-pointed-morphism-cocycle}"] \\
	S^2 \ar[r] & N
\end{tikzcd}
\end{equation}
where the left vertical map collapses the outer edges of the square.

Reading \eqref{eq-double-loop-homotopy} as a homotopy from its restriction to $([0, 1]\times\{0\}) \cup (\{1\}\times [0, 1])$ to its restriction to $(\{0\}\times[0, 1])\cup ([0,1]\times\{1\})$ (both equivalent to the trivial loop in $N$), we see that it is precisely the loop \eqref{eq-loop-by-braiding} defined by the braidings on $\Omega(M)$ and $\Omega(N)$.

The proof of the claim, and thus Lemma \ref{lem-fiber-sequence-classification}, is concluded.
\end{void}

\begin{void}
We shall now use the equivalence of Lemma \ref{lem-fiber-sequence-classification} to classify the fiber sequence \eqref{eq-truncated-k-theory-fiber-sequence} together with the distinguished section $s$ \eqref{eq-section-of-truncated-k-theory-from-pic}.

Namely, its image under the functor \eqref{eq-fiber-sequence-classification} is an anti-symmetric form:
\begin{equation}
\label{eq-truncated-k-theory-classification}
\Ktheory{}_1 \otimes \Ktheory{}_1 \rightarrow \Ktheory{}_2.
\end{equation}
\end{void}

\begin{prop}
\label{prop-truncated-k-theory-classification}
The map \eqref{eq-truncated-k-theory-classification} equals the product pairing $x, y\mapsto \{x, y\}$.
\end{prop}
\begin{proof}
Let $\cal L_1$, $\cal L_2$ be sections of $\deloop(\Ktheory{}_1)\cong \deloop\mathbb G_m$. The image of $(\cal L_1, \cal L_2)$ under the ``cocycle'' morphism, \emph{i.e.}~the special case of \eqref{eq-cocycle-formula} for $A_1 = \Ktheory{}_1$, $A_2 = \Ktheory{}_2$:
\begin{equation}
\label{eq-k-theory-cocycle-morphism}
\deloop(\Ktheory{}_1) \times \deloop(\Ktheory{}_1) \rightarrow \deloop^2(\Ktheory{}_2)
\end{equation}
is the section in $\deloop^2(\Ktheory{}_2)$ obtained from:
\begin{equation}
\label{eq-cocycle-K-theory}
([\cal L_1\otimes\cal L_2] - [\cal O]) - ([\cal L_1] - [\cal O]) - ([\cal L_2] - [\cal O])
\end{equation}
under the truncation map $\Ktheory{}_{\ge 2} \rightarrow \deloop^2(\Ktheory{}_2)$.

Using the fact that \eqref{eq-class-of-perfect-complexes} is symmetric monoidal, the section \eqref{eq-cocycle-K-theory} is equivalent to $([\cal L_1] - [\cal O])\cdot ([\cal L_2] - [\cal O])$, where $\cdot$ denotes multiplication on $\Ktheory$.

The map $\Pic \rightarrow \Ktheory{}_{\ge 1}$ induced from $\cal L \mapsto [\cal L] - [\cal O]$ coincides with the map $\deloop(\Ktheory{}_1) \rightarrow \Ktheory{}_{\ge 1}$ defined by truncation. Thus, \eqref{eq-k-theory-cocycle-morphism} renders the following diagram commutative:
$$
\begin{tikzcd}[column sep = 1.5em]
	\deloop(\Ktheory{}_1) \times \deloop(\Ktheory{}_1) \ar[r, "\eqref{eq-k-theory-cocycle-morphism}"]\ar[d] & \deloop^2(\Ktheory{}_2) \ar[d] \\
	\Ktheory{}_{\ge 1} \otimes \Ktheory{}_{\ge 1} \ar[r, "\cdot"] & \Ktheory{}_{\ge 2}
\end{tikzcd}
$$
Here, the bottom horizontal arrow is the multiplicative structure on $\Ktheory$ and the vertical maps are defined by truncations. However, the only bi-rigidified morphism $\deloop(\Ktheory{}_1) \times \deloop(\Ktheory{}_1) \rightarrow \deloop^2(\Ktheory{}_2)$ rendering this diagram commutative is the product pairing.
\end{proof}

\begin{rem}
\label{rem-k-theory-pairing}
Let us view the map defined by the product pairing $x, y\mapsto \{x, y\}$ as a morphism of Zariski sheaves of connective spectra:
\begin{equation}
\label{eq-k-theory-pairing}
\Pic \otimes \Pic \rightarrow \deloop^2(\Ktheory{}_2),\quad (\cal L_1,\cal L_2) \mapsto \{\cal L_1, \cal L_2\}.
\end{equation}

Then Proposition \ref{prop-truncated-k-theory-classification} exihibits a canonical isomorphism of sections of $\Ktheory{}_{[1, 2]}$:
\begin{equation}
\label{eq-section-quadratic-relation}
s(\cal L_1\otimes \cal L_2) - s(\cal L_1) - s(\cal L_2) \cong \{\cal L_1, \cal L_2\},
\end{equation}
where $s : \Pic \rightarrow \Ktheory{}_{[1, 2]}$ is the morphism \eqref{eq-section-of-truncated-k-theory-from-pic} of sheaves of pointed spaces. One may therefore view $s$ as a ``quadratic refinement'' of the product pairing \eqref{eq-k-theory-pairing}.
\end{rem}

\begin{void}
\label{void-truncated-k-theory-loop-space}
Combining Proposition \ref{prop-truncated-k-theory-classification} with the alternative construction of the pairing given in \S\ref{void-fiber-sequence-classification-alternative}, we obtain an explicit description of the loop space of \eqref{eq-truncated-k-theory-fiber-sequence}:
\begin{equation}
\label{eq-truncated-k-theory-loop-space}
	\deloop\Ktheory{}_2 \rightarrow \Omega(\Ktheory{}_{[1,2]}) \rightarrow \Ktheory{}_1.
\end{equation}

Namely, it splits as sheaves of $\mathbb E_1$-monoids, given by \eqref{eq-section-of-truncated-k-theory-from-pic}. The commutativity constraint on $\Omega(\Ktheory{}_{[1, 2]})$ is described by the anti-symmetric pairing:
$$
\Ktheory{}_1\otimes\Ktheory{}_1 \rightarrow \Ktheory{}_2,\quad x, y\mapsto \{x, y\}.
$$
\end{void}

\begin{void}
\label{void-multiplication-by-line-bundle}
Let $R$ be a ring and $\cal M$ be a line bundle over $S := \Spec(R)$. We shall temporarily work over the big Zariski site of $\Spec(R)$.

Multiplication by the object $[\cal M] \in K(R)$ induces a morphism of sheaves of connective spectra $\cdot [\cal M] : \Ktheory \rightarrow \Ktheory$, hence a morphism:
\begin{equation}
\label{eq-multiplication-by-line-bundle}
	\cdot [\cal M] : \Ktheory{}_{[1, 2]} \rightarrow \Ktheory{}_{[1, 2]}.
\end{equation}

Since the image of $[\cal M]$ in $\Gamma(\Spec(R), \Ktheory{}_0)$ is the multiplicative unit, multiplication by $[\cal M]$ induces the identity map on the homotopy sheaves $\Ktheory{}_n$ for each $n\ge 0$. In particular, we obtain an automorphism of the triangles \eqref{eq-truncated-k-theory-fiber-sequence}:
\begin{equation}
\label{eq-multiplication-by-line-bundle-shearing}
\begin{tikzcd}
	\deloop^2\Ktheory{}_2 \ar[r]\ar[d, "\id"] & \Ktheory{}_{[1, 2]} \ar[r]\ar[d, "{\cdot [\cal M]}"] & \deloop\Ktheory{}_1 \ar[d, "\id"] \\
	\deloop^2\Ktheory{}_2 \ar[r] & \Ktheory{}_{[1, 2]} \ar[r] & \deloop\Ktheory{}_1
\end{tikzcd}
\end{equation}

\emph{Claim}: \eqref{eq-multiplication-by-line-bundle} is the sum of the identity on $\Ktheory{}_{[1, 2]}$ with the shearing map $\deloop\Ktheory{}_1 \cong \Pic \rightarrow \deloop^2\Ktheory{}_2$ defined by $\cal L\mapsto \{\cal L, \cal M\}$ (in the notation \eqref{eq-k-theory-pairing}).

To see this, we may lift a section $\cal L$ of $\deloop\Ktheory{}_1$ to $s(\cal L) = [\cal L] - [\cal O]$ of $\Ktheory{}_{[1, 2]}$. The section $s(\cal L)\cdot[\cal M] \cong [\cal L \otimes \cal M] - [\cal M]$ is then the sum of $s(\cal L)$ with $\{\cal L, \cal M\}$ by \eqref{eq-section-quadratic-relation}.
\end{void}

\subsection{The sheaf $\Ktheory{}_{[1,2]}^{\super}$}
\label{sec-super-k-theory}

\begin{void}
Let $\Pic^{\super}$ denote the Zariski sheaf of super (\emph{i.e.}~$\mathbb Z/2$-graded) line bundles. As a sheaf of connective spectra, it coincides with the cofiber of the map:
\begin{equation}
\label{eq-evenly-graded-trivial-lines}
\underline{\mathbb Z} \rightarrow \Pic^{\mathbb Z},\quad n \mapsto (\cal O, 2n).
\end{equation}

We could suggestively denote $\Pic^{\super}$ by $\Ktheory{}_{[0, 1]}^{\super}$, viewing \eqref{eq-evenly-graded-trivial-lines} as a morphism $\Ktheory{}_0 \rightarrow \Ktheory{}_{[0, 1]}$ lifting the squaring map on $\Ktheory{}_0$.

The goal of this subsection is to introduce the Zariski sheaf $\Ktheory{}_{[1, 2]}^{\super}$ of connective spectra, defined as the cofiber of a map lifting the squaring map on $\deloop\Ktheory{}_1$:
\begin{equation}
\label{eq-evenly-graded-trivial-k-classes}
\Sq : \deloop\Ktheory{}_1 \rightarrow \Ktheory{}_{[1, 2]}.
\end{equation}
\end{void}

\begin{void}
Let us first define \eqref{eq-evenly-graded-trivial-k-classes} as a morphism of sheaves of \emph{pointed spaces}. To do so, we interpret $\deloop\Ktheory{}_1$ as $\Pic$ and define \eqref{eq-evenly-graded-trivial-k-classes} by the formula:
\begin{equation}
\label{eq-duplication-section}
\Pic \rightarrow \Ktheory{}_{[1,2]},\quad \cal L \mapsto [\cal L] - [\cal L^{-1}].
\end{equation}
More precisely, the formula $\cal L\mapsto [\cal L] - [\cal L^{-1}]$ defines a map $\Pic(R) \rightarrow K(R)$, which induces \eqref{eq-duplication-section} upon Zariski sheafification and truncation as in \S\ref{void-truncated-k-theory-section}.

We argue that the structure of a morphism of connective spectra on \eqref{eq-duplication-section} is unique, if it exists. Indeed, since \eqref{eq-duplication-section} has $1$-connective source and target, it is equivalent to a morphism of sheaves of $\mathbb E_1$-monoids:
\begin{equation}
\label{eq-duplication-section-loop}
	f : \mathbb G_m \rightarrow \Omega (\Ktheory{}_{[1,2]}).
\end{equation}
Since $\Omega(\Ktheory{}_{[1,2]})$ is $1$-truncated, an $\mathbb E_{\infty}$-monoid structure on \eqref{eq-duplication-section-loop} is equivalent to the \emph{condition} that it preserves the commutativity constraint.

In particular, the following assertion involves no additional structure.
\end{void}

\begin{prop}
\label{prop-duplication-section-spectra}
The morphism of sheaves of pointed spaces \eqref{eq-duplication-section} lifts to a morphism of sheaves of connective spectra.
\end{prop}

\begin{void}
The proof of Proposition \ref{prop-duplication-section-spectra} proceeds by explicitly identifying the morphism \eqref{eq-duplication-section-loop} using the description of $\Omega(\Ktheory{}_{[1, 2]})$ in \S\ref{void-truncated-k-theory-loop-space}.

Namely, under the $\mathbb E_1$-monoidal splitting $\Omega(\Ktheory{}_{[1, 2]}) \cong \deloop(\Ktheory{}_2) \times \Ktheory{}_1$, \eqref{eq-duplication-section-loop} corresponds to two $\mathbb E_1$-monoidal morphisms:
\begin{align*}
	f_1 &: \mathbb G_m \rightarrow \Ktheory{}_1,\\
	f_2 &: \mathbb G_m \rightarrow \deloop(\Ktheory{}_2).
\end{align*}
\end{void}

\begin{lem}
\label{lem-duplication-section-loop}
The following statements hold:
\begin{enumerate}
	\item $f_1$ is the squaring map $x\mapsto x^2$;
	\item $f_2$ is trivial as a morphism of sheaves of \emph{pointed spaces}, and its $\mathbb E_1$-monoid structure is defined by the automorphism of the trivial $\Ktheory{}_2$-torsor:
	$$
	f_2(x) \otimes f_2(y) \xrightarrow{\simeq} f_2(xy),\quad 1\mapsto 2\cdot \{x, y\},
	$$
	for each $x, y \in \mathbb G_m$.
\end{enumerate}
\end{lem}
\begin{proof}
As the section \eqref{eq-section-of-truncated-k-theory-from-pic} lifts the identity map on $\Pic\cong \deloop(\Ktheory{}_1)$, statement (1) follows from the isomorphism $\cal L \otimes (\cal L^{-1})^{-1} \cong \cal L^{\otimes 2}$.

For statement (2), we first apply the isomorphism \eqref{eq-section-quadratic-relation} to the pairs of sections $\cal L, \cal L^{-1} \in \Pic$ and $\cal L, \cal L \in\Pic$ to obtain isomorphisms in $\Ktheory{}_{[1,2]}$:
\begin{align*}
s(\cal L) + s(\cal L^{-1}) &\cong -\{\cal L, \cal L^{-1}\}, \\
2\cdot s(\cal L) &\cong s(\cal L^2) - \{\cal L, \cal L\},
\end{align*}
where $s$ denotes the section \eqref{eq-section-of-truncated-k-theory-from-pic}.

Their difference yields an isomorphism in $\Ktheory{}_{[1, 2]}$:
\begin{equation}
\label{eq-section-via-pairing}
[\cal L] - [\cal L^{-1}] \cong s(\cal L^2) - 2\cdot \{\cal L, \cal L\}.
\end{equation}
In particular, this shows that $f_2$ is given by $(-2)$ times the loop space of the self-pairing:
\begin{equation}
\label{eq-self-cup-product-pairing}
\Pic \rightarrow \deloop^2(\Ktheory{}_2),\quad \cal L\mapsto \{\cal L, \cal L\}.
\end{equation}

According to \cite[Theorem 2.5]{MR2915483}, the loop space of \eqref{eq-self-cup-product-pairing} is the map $\mathbb G_m \rightarrow \deloop(\Ktheory{}_2)$ which is trivial as a morphism of sheaves of pointed spaces, with $\mathbb E_1$-monoid structure given, for any $x, y\in\mathbb G_m$, by the automorphism $1\mapsto -\{x, y\}$ of the trivial $\Ktheory{}_2$-torsor. The desired conclusion follows.
\end{proof}

\begin{proof}[Proof of Proposition \ref{prop-duplication-section-spectra}]
It suffices to prove that \eqref{eq-duplication-section-loop} preserves the commutativity constraint. In other words, given $x, y \in \mathbb G_m$, we must show that the following diagram of sections of $\Omega(\Ktheory{}_{[1,2]})$ commutes:
\begin{equation}
\label{eq-duplication-section-braiding}
\begin{tikzcd}[column sep = 1.5em]
	f(x)\otimes f(y) \ar[r, "\simeq"]\ar[d, "c_{f(x), f(y)}"] & f(xy) \ar[d, "f(c_{x, y})"] \\
	f(y)\otimes f(x) \ar[r, "\simeq"] & f(yx)
\end{tikzcd}
\end{equation}
Here, the horizontal morphisms are given by the $\mathbb E_1$-monoid structure of $f$ and the vertical morphisms are the commutativity constraints of $\Omega(\Ktheory{}_{[1, 2]})$, respectively $\mathbb G_m$ (identity).

By Lemma \ref{lem-duplication-section-loop}, the commutativity of \eqref{eq-duplication-section-braiding} is equivalent to the following equality of sections of $\Ktheory{}_2$:
$$
2\cdot \{x, y\} = \{x^2, y^2\} + 2\cdot \{y, x\}.
$$
This follows at once from the bilinearity and anti-symmetry of the pairing.
\end{proof}

\begin{rem}
From the proof of Proposition \ref{prop-duplication-section-spectra}, we see that if we replace \eqref{eq-duplication-section} by the ``obvious'' lift of the squaring map $\cal L\mapsto 2\cdot s(\cal L)$, it would \emph{not} define a morphism of sheaves of connective spectra.
\end{rem}

\begin{void}
\label{void-super-k-spectrum-definition}
Having constructed \eqref{eq-duplication-section}, thus $\Sq$ \eqref{eq-evenly-graded-trivial-k-classes}, as a morphism of sheaves of connective spectra, we define $\Ktheory{}_{[1, 2]}^{\super}$ to be the cofiber of $\Sq$.

The following diagram summarizes four cofiber sequences of Zariski sheaves of connective spectra relevant for us:
\begin{equation}
\label{eq-super-k-spectrum-fiber-sequences}
\begin{tikzcd}[column sep = 1.5em]
	& \deloop(\Ktheory{}_1) \ar[r, "\id"]\ar[d, "\Sq"] & \deloop(\Ktheory{}_1) \ar[d, "\cdot 2"] \\
	\deloop^2(\Ktheory{}_2) \ar[r]\ar[d, "\id"] & \Ktheory{}_{[1, 2]} \ar[r]\ar[d] & \deloop(\Ktheory{}_1) \ar[d] \\
	\deloop^2(\Ktheory{}_2) \ar[r] & \Ktheory{}_{[1,2]}^{\super} \ar[r] & \deloop(\Ktheory{}_1)/2
\end{tikzcd}
\end{equation}
\end{void}

\subsection{Integration on curves}
\label{sec-integration-on-curves}

\begin{void}
Given a quasi-compact and quasi-separated scheme $S$, we write $\mathbf K(S)$ for the \emph{non-connective} $\tn K$-theory spectrum of the stable $\infty$-category $\Perf(S)$ \cite[\S9]{MR3070515}. The association $S\mapsto \mathbf K(S)$ is a Zariski sheaf of spectra \cite[Theorem 8.1]{MR1106918}.

If $S$ is regular, then the restriction of $\mathbf K$ to the small Zariski site of $S$ takes values in connective spectra, so $\mathbf K(S)$ coincides with $\Gamma(S, \Ktheory)$ \cite[Proposition 6.8]{MR1106918}.
\end{void}

\begin{void}
\label{void-relative-curve}
Let $S$ be a regular affine scheme of finite type over a field. This assumption guarantees that each Zariski sheaf $\Ktheory{}_n$ over $S$ has cohomological amplitude $\le n$ by the Gersten resolution.

Let $p : X_S \rightarrow S$ be a smooth, proper morphism of relative dimension $1$ with connected geometric fibers.

The functor $\Perf(X_S) \rightarrow \Perf(S)$, $\cal E \mapsto Rp_*\cal E$ induces a morphism of spectra $\mathbf K(X_S) \rightarrow \mathbf K(S)$. By regularity, this amounts to a morphism of connective spectra:
\begin{equation}
\label{eq-integration-curves-full-spectra}
	\Gamma(X_S, \Ktheory) \rightarrow \Gamma(S, \Ktheory).
\end{equation}
\end{void}

\begin{lem}
\label{lem-integration-curves-truncation}
For $n= 0, 1$, the morphism \eqref{eq-integration-curves-full-spectra} fits into a commutative diagram:
\begin{equation}
\label{eq-integration-curves-truncation-spectra}
\begin{tikzcd}[column sep = 1em]
	\Gamma(X_S, \Ktheory) \ar[r, "\eqref{eq-integration-curves-full-spectra}"]\ar[d] & \Gamma(S, \Ktheory) \ar[d] \\
	\Gamma(X_S, \Ktheory{}_{\le n+1}) \ar[r] & \Gamma(S, \Ktheory{}_{\le n})
\end{tikzcd}
\end{equation}
where the vertical arrows are defined by truncation on $\Ktheory$.
\end{lem}
\begin{proof}
We treat the case $n = 1$, as the case $n = 0$ is similar but simpler.

For $n = 1$, it suffices to trivialize the composition of maps of connective spectra:
\begin{equation}
\label{eq-integration-curves-full-spectra-composition}
\begin{tikzcd}[column sep = 1.5em]
\Gamma(X_S, \Ktheory{}_{\ge 3}) \ar[r] & \Gamma(X_S, \Ktheory) \ar[r, "\eqref{eq-integration-curves-full-spectra}"] & \Gamma(S, \Ktheory) \ar[r] & \Gamma(S, \Ktheory{}_{\le 1}),
\end{tikzcd}
\end{equation}
where the last arrow is defined by truncation on $\Ktheory$.

Since $\Ktheory{}_{\le 1} \cong \Pic^{\mathbb Z}$ (Remark \ref{rem-determinant-via-k-theory}) and $S$ is regular, a section of $\Gamma(S, \Ktheory{}_{\le 1})$ is trivialized once it is trivialized away from codimension $\ge 2$. Hence we may replace $S$ by the spectrum of a discrete valuation ring $R$ with field of fractions $F$.

In this case, $X_S$ is Noetherian of Krull dimension 2, so for each $i\ge 0$, the complex $\Gamma(X_S, \Ktheory{}_i[i])$ is concentrated in cohomological degrees $\le -i+2$. Triviality of \eqref{eq-integration-curves-full-spectra-composition} thus amounts to the \emph{condition} that its induced map on $H^{-1}$ below vanishes:
\begin{equation}
\label{eq-integration-curves-full-spectra-composition-cohomology}
H^2(X_S, \Ktheory{}_3) \rightarrow H^0(S, \Ktheory{}_1) \cong R^{\times}.
\end{equation}

However, the formation of \eqref{eq-integration-curves-full-spectra-composition} is of Zariski local nature on $S$, so \eqref{eq-integration-curves-full-spectra-composition-cohomology} fits into the commutative diagram:
$$
\begin{tikzcd}[column sep = 1em]
	H^2(X_S, \Ktheory{}_3) \ar[r]\ar[d] & R^{\times} \ar[d] \\
	H^2(X_F, \Ktheory{}_3) \ar[r] & F^{\times}
\end{tikzcd}
$$
Here, $X_F := X\times_S\Spec(F)$ has Krull dimension $1$, so $H^2(X_F, \Ktheory{}_3) = 0$ and \eqref{eq-integration-curves-full-spectra-composition-cohomology} vanishes.
\end{proof}

\begin{void}
We may now define a morphism:
\begin{equation}
\label{eq-integration-along-curves}
\int_{X_S} : \Gamma(X_S, \Ktheory{}_{[1, 2]}) \rightarrow \Gamma(S, \Pic^{\mathbb Z}),
\end{equation}
to be the composition:
$$
\Gamma(X_S, \Ktheory{}_{[1, 2]}) \rightarrow \Gamma(X_S, \Ktheory{}_{\le 2}) \rightarrow \Gamma(S, \Ktheory{}_{\le 1}) \cong \Gamma(S, \Pic^{\mathbb Z}).
$$
where the first morphism comes from the inclusion $\Ktheory{}_{[1, 2]} \rightarrow \Ktheory{}_{\le 2}$, and the second morphism is the bottom arrow of \eqref{eq-integration-curves-truncation-spectra} for $n = 1$.

Comparing the cases $n=1$ and $n = 0$ in \eqref{eq-integration-curves-truncation-spectra} shows that \eqref{eq-integration-along-curves} induces a morphism of fiber sequences:
\begin{equation}
\label{eq-integration-along-curves-degree}
\begin{tikzcd}[column sep = 1em]
	\Gamma(X_S, \deloop^2\Ktheory{}_2) \ar[r]\ar[d, "\int_{X_S}"] & \Gamma(X_S, \Ktheory{}_{[1, 2]}) \ar[r]\ar[d, "\int_{X_S}"] & \Gamma(X_S, \deloop\Ktheory{}_1) \ar[d, "\int_{X_S}"] \\
	\Gamma(S, \Pic) \ar[r] & \Gamma(S, \Pic^{\mathbb Z}) \ar[r] & \Gamma(S, \underline{\mathbb Z})
\end{tikzcd}
\end{equation}
Here, the rightmost vertical arrow has the following explict description: it associates to a line bundle over $X_S$ its degree, viewed as a locally constant function over $S$.
\end{void}

\begin{rem}
The first vertical functor in \eqref{eq-integration-along-curves-degree} is constructed by Gaitsgory in \cite[\S2.4]{MR4117995} using the Gersten resolution of $\Ktheory{}_2$.
\end{rem}

\begin{void}
Let us now define the ``super'' variant of \eqref{eq-integration-along-curves}, which requires a \emph{spin structure} over $X_S$. From now on, we fix a square root $\omega^{1/2}$ of the relative canonical bundle $\omega_{X_S/S}$.

Define the morphism:
\begin{equation}
\label{eq-integration-along-curves-omega-twist}
\int_{(X_S, \omega^{1/2})} : \Gamma(X_S, \Ktheory{}_{[1, 2]}) \rightarrow \Gamma(S, \Pic^{\mathbb Z})
\end{equation}
to be the composition of \eqref{eq-integration-along-curves} with the multiplication $\cdot[\omega^{1/2}] : \Ktheory{}_{[1, 2]} \rightarrow \Ktheory{}_{[1, 2]}$ (\emph{i.e.}~the morphism \eqref{eq-multiplication-by-line-bundle} for $\cal M := \omega^{1/2}$).

The following observation shows that the $\omega^{1/2}$-twisted integration morphism \eqref{eq-integration-along-curves-omega-twist} intertwines the squaring maps \eqref{eq-evenly-graded-trivial-lines} and \eqref{eq-evenly-graded-trivial-k-classes}.
\end{void}

\begin{lem}
The following diagram is canonically commutative:
\begin{equation}
\label{eq-integration-along-curves-omega-twist-square}
\begin{tikzcd}[column sep =3em]
	\Gamma(X_S, \deloop\Ktheory{}_1) \ar[d, "\Sq"] \ar[r, "\int_{X_S}"] & \Gamma(S, \underline{\mathbb Z}) \ar[d, "{n\mapsto(\cal O, 2n)}"] \\
	\Gamma(X_S, \Ktheory{}_{[1, 2]}) \ar[r, "\int_{(X_S, \omega^{1/2})}"] & \Gamma(S, \Pic^{\mathbb Z})
\end{tikzcd}
\end{equation}
\end{lem}
\begin{proof}
By construction, the lower circuit of \eqref{eq-integration-along-curves-omega-twist-square} sends a line bundle $\cal L$ over $X_S$ to the $\mathbb Z$-graded line bundle:
$$
\det(Rp_*(\cal L \otimes \omega^{1/2})) \otimes \det(Rp_*(\cal L^{-1}\otimes\omega^{1/2}))^{-1},
$$
This is the trivial line bundle by Grothendieck--Serre duality. It is placed in degree $2\deg(\cal L)$ by the Riemann--Roch formula.
\end{proof}

\begin{void}
Since $X_S$ is regular, the Zariski cohomology group $H^2(X_S, \mathbb G_m)$ vanishes. Thus, taking cofibers of the vertical arrows in \eqref{eq-integration-along-curves-omega-twist-square} yields a morphism:
\begin{equation}
\label{eq-integration-along-curves-super}
\int_{(X_S,\omega^{1/2})} : \Gamma(X_S, \Ktheory{}_{[1, 2]}^{\super}) \rightarrow \Gamma(S, \Pic^{\super}).
\end{equation}

The morphisms of fiber sequences \eqref{eq-multiplication-by-line-bundle-shearing} and \eqref{eq-integration-along-curves-degree} induce a morphism of fiber sequences:
\begin{equation}
\label{eq-integration-along-curves-degree-super}
\begin{tikzcd}[column sep = 1em]
	\Gamma(X_S, \deloop^2\Ktheory{}_2) \ar[r]\ar[d, "\int_{X_S}"] & \Gamma(X_S, \Ktheory{}_{[1, 2]}^{\super}) \ar[r]\ar[d, "\int_{(X_S, \omega^{1/2})}"] & \Gamma(X_S, \deloop(\Ktheory{}_1/2)) \ar[d, "\int_{X_S}"] \\
	\Gamma(S, \Pic) \ar[r] & \Gamma(S, \Pic^{\super}) \ar[r] & \Gamma(S, \underline {\mathbb Z}/2)
\end{tikzcd}
\end{equation}
Here, the term $\Gamma(X_S, \deloop(\Ktheory{}_1/2))$ is identified with the cofiber of the multiplication by $2$ map on $\Gamma(X_S, \deloop\Ktheory{}_1)$, and the rightmost vertical arrow has the following description: it sends a line bundle over $X_S$ to its degree mod $2$.
\end{void}

\begin{void}
\label{void-super-conformal-blocks}
We now explain how sections of $\Ktheory{}_{[1, 2]}^{\super}$ over the Zariski classifying stack of a split reductive group scheme define super conformal blocks, at least in the vacuum case.

Let $\cal M^{\spin}$ denote the moduli stack of spin curves. Namely, an $S$-point of $\cal M^{\spin}$ consists of a morphism $p : X_S\rightarrow S$ of smooth, proper morphism of relative dimension $1$ with connected geometric fibers together with a square root $\omega^{1/2}$ of the relative canonical bundle.

Given an affine group scheme $G$, denote by $\deloop G$ the stack classifying \emph{Zariski} locally trivial $G$-torsors.\footnote{This is in accordance with the notation $B$ used elsewhere in this section, but our $BG$ is different from the usual stack classifying \'etale or fppf locally trivial $G$-torsors.}

Denote by $\Bun_G$ the stack over $\cal M^{\spin}$ whose $S$-points are triples $(X_S, \omega^{1/2}, P)$, where $(X_S, \omega^{1/2})$ is an $S$-point of $\cal M^{\spin}$ and $P$ is a $G$-bundle over $X_S$.

If $G$ is \emph{split} reductive, we shall define a functor:
\begin{equation}
\label{eq-construction-super-line-global}
\Gamma(\deloop G, \Ktheory{}_{[1, 2]}^{\super}) \rightarrow \Gamma(\Bun_G, \Pic^{\super}).
\end{equation}
Given a section $\kappa$ of $\Ktheory{}_{[1, 2]}^{\super}$ over $\deloop G$, we define the $\mathbb Z/2$-graded quasi-coherent sheaf $\mathbb V_{\kappa}$ over $\cal M^{\spin}$ to be the pushforward along $\Bun_G \rightarrow \cal M^{\spin}$ of the image of $\kappa$ along \eqref{eq-construction-super-line-global}.
\end{void}

\begin{proof}[Construction of \eqref{eq-construction-super-line-global}]
Consider an $S$-point $(X_S, \omega^{1/2}, P)$ of $\Bun_G$ where $S$ is regular and affine. \'Etale locally over $S$, we may assume that $P$ is Zariski locally trivial \cite[Theorem 2]{MR1362973}, so it defines a morphism $P : X_S \rightarrow \deloop G$.

Pulling back along $P$ and applying \eqref{eq-integration-along-curves-super} yields a functor:
\begin{equation}
\label{eq-construction-super-line-global-test}
\Gamma(\deloop G, \Ktheory{}_{[1, 2]}^{\super}) \rightarrow \Gamma(S, \Pic^{\super}).
\end{equation}
Since $\Bun_G\rightarrow\Spec(\mathbb Z)$ is smooth, a super line bundle over $\Bun_G$ is equivalent to a compatible system of super line bundles over $S$, for all regular affine schemes $S$ over $\Bun_G$. Using functoriality of \eqref{eq-construction-super-line-global-test} in $S$, we obtain \eqref{eq-construction-super-line-global}.
\end{proof}

\begin{rem}
Let us work over an algebraically closed field $k$.

If $G$ is split simple and simply connected with a split maximal torus $T\subset G$, the Picard groupoid of sections of $\deloop^2\Ktheory{}_2$ over $\deloop G$ rigidified along the base point $e : \Spec(k) \rightarrow \deloop G$ is discrete and isomorphic to the abelian group of Weyl-invariant quadratic forms on the cocharacter lattice $\Lambda$ of $T$ \cite[Theorem 4.7]{MR1896177}.

This abelian group has a canonical generator: the Weyl-invariant quadratic form $Q$ with $Q(\alpha) = 1$ at any short coroot $\alpha$. To each integer $\kappa$, the quadratic form $\kappa\cdot Q$ thus defines a section of $\deloop^2\Ktheory{}_2$ over $\deloop G$.

The quasi-coherent sheaf $\mathbb V_{\kappa}$ defined by this section, via the construction of \S\ref{void-super-conformal-blocks}, is identified with the space of (vacuum) conformal blocks at level $\kappa$ in the usual sense, see \cite{MR1289330}. They are known to be finite locally free if $\tn{char}(k) = 0$ \cite{MR1048605}.

We have not undertaken a serious investigation of $\mathbb V_{\kappa}$ in the generality of \S\ref{void-super-conformal-blocks}.
\end{rem}

\begin{rem}
Let us note an analogue of \eqref{eq-integration-along-curves-degree-super} for surfaces. Suppose that $X$ is a proper smooth surface over an algebraically closed field $k$. Consider the composition:
\begin{equation}
\label{eq-surface-integration}
\int_X : \Gamma(X, \deloop^2\Ktheory{}_2) \rightarrow H^2(X, \Ktheory{}_2) \cong \tn{CH}^2(X) \xrightarrow{\deg} \mathbb Z.
\end{equation}

Suppose that the dualizing sheaf $\omega_{X/k}$ admits a square root $\omega^{1/2}$. \emph{Claim}: the morphism \eqref{eq-surface-integration} canonically extends to a morphism $\Gamma(X, \Ktheory{}_{[1, 2]}^{\super}) \rightarrow \mathbb Z$.

This extension will be defined as an analogue of the morphism \eqref{eq-integration-along-curves-super} for surfaces. Namely, we note that \eqref{eq-integration-along-curves} has an analogue for surfaces: the map $\Gamma(X, \Ktheory{}_{[1, 2]}) \rightarrow \mathbb Z$ induced from $[\cal E] \mapsto \chi(R\Gamma(X, \cal E))$. To see that it factors through $\Gamma(X, \Ktheory{}_{[1, 2]}^{\super})$ when $\omega^{1/2}$ exists, we appeal to the equality:
\begin{equation}
\label{eq-euler-characteristic-vanishing}
\chi(R\Gamma(X, \cal L \otimes \omega^{1/2})) - \chi(R\Gamma(X, \cal L^{-1}\otimes\omega^{1/2})) = 0,
\end{equation}
for every line bundle $\cal L$ over $X$, which follows from the Riemann--Roch formula:
$$
\chi(R\Gamma(X, \cal L\otimes\omega^{1/2})) = \chi(R\Gamma(X, \cal O)) + \frac{1}{2}(\cal L\cdot\cal L - \omega^{1/2}\cdot\omega^{1/2}),
$$
where $\cdot$ denotes the intersection pairing, \emph{i.e.}~the composition of \eqref{eq-k-theory-pairing} with \eqref{eq-surface-integration}.
\end{rem}

\section{Brylinski--Deligne classification}
\label{sec-brylinski-deligne}

In this section, we classify rigidified sections of $\Ktheory{}_{[1, 2]}^{\super}$ over the Zariski classifying stack $\deloop G$ of a reductive group scheme $G$ over a base scheme $S$, assumed regular and of finite type over a field. The main result is Theorem \ref{thm-classification-super}.

We begin in \S\ref{sec-classification-integer-grading} with a classification of rigidified sections of $\Ktheory{}_{[1, 2]}$ over $\deloop G$ (Proposition \ref{prop-classification-integral}). Using tools developed in \S\ref{sec-truncated-k-theory}, we reduce this result to the Brylinski--Deligne theorem \cite[Theorem 7.2]{MR1896177}. In \S\ref{sec-classification-super-grading}, we state the main result. The next subsection \S\ref{sec-central-extension-k1} is a technical interlude classifying central extensions of $G$ by $\mathbb G_m$ over an arbitrary base scheme. The results of \S\ref{sec-classification-integer-grading} and \S\ref{sec-central-extension-k1} are combined in \S\ref{sec-proof-classifcation-super} to prove Theorem \ref{thm-classification-super}.

\subsection{Classification: $\Ktheory{}_{[1, 2]}$}
\label{sec-classification-integer-grading}

\begin{void}
\label{void-classification-setup}
Let $S$ be a regular scheme of finite type over a field.

Let $G \rightarrow S$ be a reductive group scheme equipped with a maximal torus $T\subset G$. Cocharacters of $T$ form an \'etale sheaf of abelian groups $\Lambda$ over $S$.

Denote by $G_{\mathrm{sc}}$ the simply connected form of $G$. The preimage of $T$ in $G_{\mathrm{sc}}$ is a maximal torus $T_{\mathrm{sc}}\subset G_{\mathrm{sc}}$, whose sheaf of cocharacters is denoted by $\Lambda_{\mathrm{sc}}$. The algebraic fundamental group $\pi_1G$ may then be realized as $\Lambda/\Lambda_{\mathrm{sc}}$.
\end{void}

\begin{void}
Denote by $\deloop G$ the stack of \emph{Zariski} locally trivial $G$-torsor. Denote by $e : S \rightarrow \deloop G$ the unit section. For a Zariski sheaf $\cal F$ of connective spectra, we write $\Gamma_e(\deloop G, \cal F)$ for the fiber of the morphism:
$$
e^* : \Gamma(\deloop G, \cal F) \rightarrow \Gamma(S, \cal F).
$$

We also denote by $\underline{\Gamma}{}_e(\deloop G, \cal F)$ the presheaf over $S$ whose section over an $S_1$-scheme is $\Gamma_e(\deloop G\times_SS_1, \cal F)$. It is a sheaf in the Zariski topology.

In this subsection, we describe $\underline{\Gamma}{}_e(\deloop G, \Ktheory{}_{[1, 2]})$ in terms of the combinatorics of $G$.
\end{void}

\begin{void}
\label{void-brylinski-deligne}
We first recall Brylinski and Deligne's description of $\underline{\Gamma}{}_e(\deloop G, \deloop^2\Ktheory{}_2)$. Indeed, \cite[Theorem 7.2]{MR1896177} constructs a canonical equivalence of sheaves of Picard groupoids:
\begin{equation}
\label{eq-brylinski-deligne}
	\underline{\Gamma}{}_e(\deloop G, \deloop^2\Ktheory{}_2) \xrightarrow{\simeq} \vartheta_G(\Lambda),
\end{equation}
where sections of $\vartheta_G(\Lambda)$ are triples $(Q, \tilde{\Lambda}, \varphi)$ defined below:
\begin{enumerate}
	\item $Q$ is a Weyl-invariant integral quadratic form on $\Lambda$;
	\item $\tilde{\Lambda}$ is a central extension of $\Lambda$ by $\mathbb G_m$, whose commutator pairing $\Lambda\otimes\Lambda\rightarrow \mathbb G_m$ equals $\lambda_1,\lambda_2\mapsto(-1)^{b(\lambda_1,\lambda_2)}$ for $b(\lambda_1,\lambda_2):= Q(\lambda_1 + \lambda_2) - Q(\lambda_1) - Q(\lambda_2)$;
	\item $\varphi$ is an isomorphism between the restriction of $\tilde{\Lambda}$ to $\Lambda_{\mathrm{sc}}$ and the central extension induced from $Q_{\mathrm{sc}}$ (in the sense of Remark \ref{rem-brylinski-deligne-functor}), the restriction of $Q$ to $\Lambda_{\mathrm{sc}}$.
\end{enumerate}
The Picard groupoid structure on $\vartheta_G(\Lambda)$ is defined by sum in $Q$ and Baer sum in $\tilde{\Lambda}$.
\end{void}

\begin{rem}
\label{rem-brylinski-deligne-functor}
To be more explicit, \cite[\S3]{MR1896177} first shows that $\underline{\Gamma}{}_e(\deloop T, \deloop^2\Ktheory{}_2)$ is canonically equivalent to the sheaf of Picard groupoids $\vartheta(\Lambda)$ whose sections are pairs $(Q, \tilde{\Lambda})$, where $Q$ is an integral quadratic form on $\Lambda$, and $\tilde{\Lambda}$ is as in (2).

Then \cite[\S4]{MR1896177} shows that $\underline{\Gamma}{}_e(\deloop G_{\mathrm{sc}}, \deloop^2\Ktheory{}_2)$ is the sheaf of discrete groupoids whose sections are Weyl-invariant integral quadratic forms on $\Lambda_{\mathrm{sc}}$. Restriction along $T_{\mathrm{sc}} \subset G_{\mathrm{sc}}$ and applying the description of $\underline{\Gamma}{}_e(\deloop T_{\mathrm{sc}}, \deloop^2\Ktheory{}_2)$, we obtain a functor from Weyl-invariant quadratic forms on $\Lambda_{\mathrm{sc}}$ to $\vartheta(\Lambda_{\mathrm{sc}})$:
\begin{equation}
\label{eq-canonical-theta-data-brylinski-deligne}
\Quad(\Lambda_{\mathrm{sc}})^W \rightarrow \vartheta(\Lambda_{\mathrm{sc}}).
\end{equation}
In particular, any $Q_{\mathrm{sc}} \in \Quad(\Lambda_{\mathrm{sc}})^W$ induces a central extension of $\Lambda_{\mathrm{sc}}$ by $\mathbb G_m$.

The functor from $\underline{\Gamma}{}_e(\deloop G, \deloop^2\Ktheory{}_2)$ to $\vartheta_G(\Lambda)$ is given as follows. The pair $(Q, \tilde{\Lambda})$ is defined by its restriction along $T\subset G$, and the isomorphism $\varphi$ comes from functoriality with respect to the commutative diagram:
$$
\begin{tikzcd}
	T_{\mathrm{sc}} \ar[r, phantom, "\subset"]\ar[d] & G_{\mathrm{sc}} \ar[d] \\
	T \ar[r, phantom, "\subset"] & G
\end{tikzcd}
$$
\end{rem}

\begin{rem}
\label{rem-brylinski-deligne-descent}
Both sides of \eqref{eq-brylinski-deligne} are \'etale sheaves over $S$. Indeed, it is clear that $\vartheta_G(\Lambda)$ satisfies \'etale descent. The \'etale descent of $\underline{\Gamma}{}_e(\deloop G, \deloop^2\Ktheory{}_2)$ is established in \cite[\S2]{MR1896177}, logically prior to proving that \eqref{eq-brylinski-deligne} is an equivalence.
\end{rem}

\begin{void}
Let us define an enlargement $\vartheta^{\mathbb Z}_G(\Lambda)$ of $\vartheta_G(\Lambda)$. Namely, a section of $\vartheta^{\mathbb Z}_G(\Lambda)$ is a quadruple $(Q, \tilde{\Lambda}, \varphi, x)$ where $(Q, \tilde{\Lambda}, \varphi)$ is a section of $\vartheta_G(\Lambda)$ and:
\begin{enumerate}
	\item[(4)] $x : \Lambda \rightarrow \mathbb Z$ is a character vanishing on $\Lambda_{\mathrm{sc}}$ (\emph{i.e.}~a character of $\pi_1G$).
\end{enumerate}

Therefore, as a sheaf of \emph{pointed spaces}, $\vartheta_G^{\mathbb Z}(\Lambda)$ is the product $\vartheta_G(\Lambda) \times \underline{\Hom}(\pi_1G, \mathbb Z)$. As a sheaf of \emph{Picard groupoids}, we demand that it fits into a fiber sequence:
\begin{equation}
\label{eq-classification-integral-fiber-sequence}
	\vartheta_G(\Lambda) \rightarrow \vartheta_G^{\mathbb Z}(\Lambda) \rightarrow \underline{\Hom}(\pi_1G, \mathbb Z).
\end{equation}

Specifying the Picard groupoid structure on $\vartheta_G^{\mathbb Z}(\Lambda)$ thus amounts to specifying a \emph{symmetric cocycle}, \emph{i.e.}~a morphism of Picard groupoids:
\begin{equation}
\label{eq-classification-integral-cocycle}
\underline{\Hom}(\pi_1G, \mathbb Z) \otimes \underline{\Hom}(\pi_1G, \mathbb Z) \rightarrow \vartheta_G(\Lambda),
\end{equation}
together with a null-homotopy of its precomposition with the anti-symmetrizer.
\end{void}

\begin{proof}[Construction of \eqref{eq-classification-integral-cocycle}]
Given sections $x_1, x_2$ of $\underline{\Hom}(\pi_1G, \mathbb Z)$, the morphism \eqref{eq-classification-integral-cocycle} assigns to $x_1\otimes x_2$ the triple $(Q, \tilde{\Lambda}, \varphi)$, where $Q(\lambda) := x_1(\lambda)x_2(\lambda)$, $\tilde{\Lambda}$ is the central extension defined by the cocycle $\lambda_1, \lambda_2 \mapsto (-1)^{x_1(\lambda_1)x_2(\lambda_2)}$, and $\varphi$ is the identity automorphism of the trivial central extension of $\Lambda_{\mathrm{sc}}$ by $\mathbb G_m$.

In order to construct a null-homotopy of the image of $x_1\otimes x_2 - x_2 \otimes x_1$, we need to trivialize the central extension of $\Lambda$ by $\mathbb G_m$ defined by the cocycle:
\begin{equation}
\label{eq-antisymmetrized-cocycle}
\lambda_1, \lambda_2 \mapsto (-1)^{x_1(\lambda_1)x_2(\lambda_2)-x_2(\lambda_1)x_1(\lambda_2)}.
\end{equation}
In other words, we need to find a map $q : \Lambda \rightarrow \mathbb G_m$ such that $q(\lambda_1 + \lambda_2)q(\lambda_1)^{-1}q(\lambda_2)^{-1}$ coincides with \eqref{eq-antisymmetrized-cocycle}. The desired map is set to be $q(\lambda):= (-1)^{x_1(\lambda)x_2(\lambda)}$.
\end{proof}

\begin{rem}
\label{rem-central-extension-as-monoidal-morphism}
By associating to each $\lambda\in\Lambda$ its fiber $\cal L^{\lambda}\subset\tilde{\Lambda}$ viewed as a $\mathbb G_m$-torsor, the central extension $\tilde{\Lambda}$ in a section $(Q, \tilde{\Lambda}, \varphi)$ of $\vartheta_G(\Lambda)$ can be viewed as a monoidal morphism $\Lambda \rightarrow \Pic$ which preserves the commutativity constraint up to the factor $(-1)^b$.

Likewise, an object $(Q, \tilde{\Lambda}, \varphi, x)$ of $\vartheta_G^{\mathbb Z}(\Lambda)$ defines a monoidal morphism:
$$
\Lambda \rightarrow \Pic^{\mathbb Z},\quad \lambda\mapsto (\cal L^{\lambda}, x(\lambda)),
$$
which preserves the commutativity constraint up to the factor $(-1)^{\tilde b}$ (viewed as a $\mathbb G_m$-valued bilinear form on $\Lambda$), where $\tilde b$ is defined by:
\begin{equation}
\label{eq-adjusted-symmetric-form}
\lambda_1, \lambda_2\mapsto b(\lambda_1, \lambda_2) + x(\lambda_1)x(\lambda_2).
\end{equation}
\end{rem}

\begin{prop}
\label{prop-classification-integral}
There is a canonical equivalence of sheaves of Picard groupoids:
\begin{equation}
\label{eq-classification-integral}
	\underline{\Gamma}{}_e(\deloop G, \Ktheory{}_{[1, 2]}) \xrightarrow{\simeq} \vartheta_G^{\mathbb Z}(\Lambda).
\end{equation}
It is related to the Brylinski--Deligne equivalence by a commutative diagram:
$$
\begin{tikzcd}[column sep = 1em]
	\underline{\Gamma}{}_e(\deloop G, \deloop^2\Ktheory{}_2) \ar[d, "\eqref{eq-brylinski-deligne}"] \ar[r, phantom, "\subset"] & \underline{\Gamma}{}_e(\deloop G, \Ktheory{}_{[1, 2]}) \ar[d, "\eqref{eq-classification-integral}"] \\
	\vartheta_G(\Lambda) \ar[r, phantom, "\subset"] & \vartheta_G^{\mathbb Z}(\Lambda)
\end{tikzcd}
$$
\end{prop}
\begin{proof}
As a presheaf of pointed spaces, $\underline{\Gamma}{}_e(\deloop G, \Ktheory{}_{[1, 2]})$ is the direct product $\underline{\Gamma}{}_e(\deloop G, \deloop^2\Ktheory{}_2)\times\underline{\Gamma}{}_e(\deloop G, \deloop\Ktheory{}_1)$ thanks to the section \eqref{eq-section-of-truncated-k-theory-from-pic}. In particular, $\underline{\Gamma}{}_e(\deloop G, \Ktheory{}_{[1, 2]})$ satisfies \'etale descent, see Remark \ref{rem-brylinski-deligne-descent}.

The desired functor \eqref{eq-classification-integral} is defined to be the Brylinski--Deligne equivalence \eqref{eq-brylinski-deligne} on the first factor and the canonical isomorphism:
\begin{equation}
\label{eq-character-classification}
\underline{\Gamma}{}_e(\deloop G, \deloop\Ktheory{}_1)\xrightarrow{\simeq} \underline{\Hom}(G, \Ktheory{}_1)\xrightarrow{\simeq} \underline{\Hom}(\pi_1G, \mathbb Z)
\end{equation}
on the second factor.

It thus remains to lift this functor to one between Picard groupoids. We appeal to the description of $\Ktheory{}_{[1,2]}$ using the symmetric cocycle $\deloop\Ktheory{}_1 \otimes \deloop\Ktheory{}_1 \rightarrow \deloop^2\Ktheory{}_2$ associated to the anti-symmetric pairing $\{\cdot, \cdot\} : \Ktheory{}_1\otimes\Ktheory{}_1 \rightarrow \Ktheory{}_2$ (Proposition \ref{prop-truncated-k-theory-classification}). Indeed, it suffices to construct an isomorphism between the $\underline{\Gamma}{}_e(\deloop G, \deloop^2\Ktheory{}_2)$-valued pairing it induces on $\underline{\Gamma}{}_e(\deloop G, \deloop\Ktheory{}_1)$ and the pairing \eqref{eq-classification-integral-cocycle}:
\begin{equation}
\label{eq-k-theory-pairing-over-group}
\begin{tikzcd}[column sep = 1em]
	\underline{\Gamma}{}_e(\deloop G, \deloop \Ktheory{}_1) \otimes \underline{\Gamma}{}_e(\deloop G, \deloop\Ktheory{}_1) \ar[r, "{\{\cdot,\cdot\}}"]\ar[d, "\eqref{eq-character-classification}"] & \underline{\Gamma}{}_e(\deloop G, \deloop^2\Ktheory{}_2) \ar[d, "\eqref{eq-brylinski-deligne}"] \\
	\underline{\Hom}(\pi_1G, \mathbb Z) \otimes \underline{\Hom}(\pi_1G, \mathbb Z) \ar[r, "\eqref{eq-classification-integral-cocycle}"] & \vartheta_G(\Lambda)
\end{tikzcd}
\end{equation}
compatibly with null homotopies of their pre-composition with the anti-symmetrizer.

By definition of $\vartheta_G(\Lambda)$, it suffices to treat the case $G = T$ as long as the isomorphism we construct is functorial in $T$.

In this case, any pair of characters $x_1, x_2$ of $T$ defines under the top horizontal arrow of \eqref{eq-k-theory-pairing-over-group} the central extension:
$$
1 \rightarrow \Ktheory{}_2 \rightarrow E \rightarrow T \rightarrow 1
$$
corresponding to the cocycle $T\otimes T \rightarrow \Ktheory{}_2$, $(t_1, t_2)\mapsto \{x_1(t_1), x_2(t_2)\}$. The null-homotopy the central extension defined by the cocycle $(t_1, t_2)\mapsto \{x_1(t_1), x_2(t_2)\} - \{x_2(t_1), x_1(t_2)\}$ is exhibited by the map $T \rightarrow \Ktheory{}_2$, $t\mapsto\{x_1(t), x_2(t)\}$. These data correspond to the description of \eqref{eq-classification-integral-cocycle} (for $G = T$) under the equivalence of \cite[Theorem 3.16]{MR1896177}.
\end{proof}

\subsection{Classification: $\Ktheory{}_{[1,2]}^{\super}$}
\label{sec-classification-super-grading}

\begin{void}
\label{void-theta-data-super}
Let us define another enlargement $\vartheta_G^{\super}(\Lambda)$ of $\vartheta_G(\Lambda)$ which fits into a fiber sequence of sheaves of Picard groupoids over $S$:
\begin{equation}
\label{eq-classification-super-fiber-sequence}
\vartheta_G(\Lambda) \rightarrow \vartheta_G^{\super}(\Lambda) \rightarrow \underline{\Hom}(\pi_1G, \mathbb Z/2).
\end{equation}

Namely, an object of $\vartheta_G^{\super}(\Lambda)$ is a triple $(b, \tilde{\Lambda}, \varphi)$ where:
\begin{enumerate}
	\item $b$ is a Weyl-invariant integral symmetric bilinear form on $\Lambda$, such that $b(\lambda, \lambda)\in 2\mathbb Z$ for any $\lambda \in \Lambda_{\mathrm{sc}}$;
	\item $\tilde{\Lambda}$ is a central extension of $\Lambda$ by $\mathbb G_m$, whose commutator pairing equals $\lambda_1, \lambda_2\mapsto (-1)^{b(\lambda_1, \lambda_2) + \epsilon(\lambda_1)\epsilon(\lambda_2)}$, where $\epsilon(\lambda) := b(\lambda, \lambda) \text{ mod }2$;
	\item $\varphi$ is an isomorphism between the restriction of $\tilde{\Lambda}$ to $\Lambda_{\mathrm{sc}}$ and the central extension induced by $Q_{\mathrm{sc}}$ as in \S\ref{void-brylinski-deligne}.
\end{enumerate}
To define the Picard groupoid structure on $\vartheta_G^{\super}(\Lambda)$, it is more natural to interpret $\tilde{\Lambda}$ as a monoidal morphism (see Remark \ref{rem-central-extension-as-monoidal-morphism}):
\begin{equation}
\label{eq-central-extension-as-monoidal-morphism-super}
\Lambda \rightarrow \Pic^{\super},\quad \lambda \mapsto (\cal L^{\lambda}, \epsilon(\lambda)),
\end{equation}
which preserves the commutativity constraint up to the bilinear form $(-1)^b$. The Picard groupoid structure on $\vartheta_G^{\super}(\Lambda)$ is induced from sum in $b$ and the Picard groupoid structure of $\Pic^{\super}$. In particular, it is \emph{not} strictly commutative in general.

Let us construct the fiber sequence \eqref{eq-classification-super-fiber-sequence}. The inclusion of $\vartheta_G(\Lambda)$ in $\vartheta_G^{\super}(\Lambda)$ sends $(Q, \tilde{\Lambda}, \varphi)$ to the triple $(b, \tilde{\Lambda}, \varphi)$, where $b(\lambda_1, \lambda_2) := Q(\lambda_1 + \lambda_2) - Q(\lambda_1) - Q(\lambda_2)$. The second map $\vartheta_G^{\super}(\Lambda) \rightarrow \Hom(\pi_1G, \mathbb Z/2)$ assigns to $(b, \tilde{\Lambda}, \varphi)$ the homomorphism $\epsilon$ as in (2). Note that $\epsilon$ vanishes if and only if $b$ comes from a quadratic form.
\end{void}

\begin{rem}
The sheaf of Picard groupoids $\vartheta_G^{\super}(\Lambda)$ is introduced in \cite[Questions 12.13(iii)]{MR1896177}. For $G = T$ a torus, a variant of it has also appeared in \cite[\S3.10]{MR2058353}, where its sections are called \emph{$\vartheta$-data}.
\end{rem}

\begin{thm}
\label{thm-classification-super}
There is a canonical equivalence of sheaves of Picard groupoids:
\begin{equation}
\label{eq-classification-super}
\underline{\Gamma}{}_e(\deloop G, \Ktheory{}_{[1,2]}^{\super}) \xrightarrow{\simeq} \vartheta_G^{\super}(\Lambda).
\end{equation}
It is related to the Brylinski--Deligne equivalence by a commutative diagram:
\begin{equation}
\label{eq-classification-super-comparison-with-brylinski-deligne}
\begin{tikzcd}[column sep = 1em]
	\underline{\Gamma}{}_e(\deloop G, \deloop^2\Ktheory{}_2) \ar[d, "\eqref{eq-brylinski-deligne}"] \ar[r, phantom, "\subset"] & \underline{\Gamma}{}_e(\deloop G, \Ktheory{}^{\super}_{[1, 2]}) \ar[d, "\eqref{eq-classification-super}"] \\
	\vartheta_G(\Lambda) \ar[r, phantom, "\subset"] & \vartheta_G^{\super}(\Lambda)
\end{tikzcd}
\end{equation}
\end{thm}

\begin{void}
The proof of Theorem \ref{thm-classification-super} will occupy the remainder of this section. For now, we shall formulate a compatibility statement between the isomorphisms \eqref{eq-classification-integral} and \eqref{eq-classification-super} (which will in fact be used to define \eqref{eq-classification-super}.)

To do so, we need to construct two morphisms of sheaves of Picard groupoids:
\begin{align}
\label{eq-classification-squaring-map}
	\Hom(\pi_1G, \mathbb Z) &\rightarrow \vartheta_G^{\mathbb Z}(\Lambda) \\
\label{eq-classification-reduction-mod-2}
	\vartheta_G^{\mathbb Z}(\Lambda) &\rightarrow \vartheta_G^{\super}(\Lambda).
\end{align}

The morphism \eqref{eq-classification-squaring-map} sends a character $x : \pi_1G \rightarrow \mathbb Z$ to the quadruple $(Q, \tilde{\Lambda}, \varphi, 2x)$ where $Q(\lambda) := -2x(\lambda)^2$, $\tilde{\Lambda}$ is the trivial central extension, and $\varphi$ is the identity automorphism of the trivial central extension.

The morphism \eqref{eq-classification-reduction-mod-2} is the identity on the subgroupoid $\vartheta_G(\Lambda)$. To any character $x$ in the additional factor $\Hom(\pi_1G, \mathbb Z)$, it assigns the triple $(b, \tilde{\Lambda}, \varphi)$ where $b(\lambda_1, \lambda_2) := x(\lambda_1)x(\lambda_2)$, $\tilde{\Lambda}$ is the trivial central extension (\emph{i.e.}~the morphism \eqref{eq-central-extension-as-monoidal-morphism-super} is given by $\lambda \mapsto (\cal O, \epsilon(\lambda))$), and $\varphi$ is the identity automorphism of the trivial central extension.
\end{void}

\begin{lem}
The maps \eqref{eq-classification-squaring-map}, \eqref{eq-classification-reduction-mod-2} thus defined are morphisms of Picard groupoids, and fit into a fiber sequence of such:
\begin{equation}
\label{eq-classification-defining-sequence}
\Hom(\pi_1G, \mathbb Z) \rightarrow \vartheta_G^{\mathbb Z}(\Lambda) \rightarrow \vartheta_G^{\super}(\Lambda).
\end{equation}
\end{lem}
\begin{proof}
We only verify that \eqref{eq-classification-defining-sequence} is indeed a fiber sequence. Let $(Q, \tilde{\Lambda}, \varphi, x)$ be an object in the fiber of \eqref{eq-classification-reduction-mod-2}. Thus the induced symmetric form:
$$
\lambda_1, \lambda_2\mapsto Q(\lambda_1 + \lambda_2) - Q(\lambda_1) - Q(\lambda_2) + x(\lambda_1)x(\lambda_2)
$$
must vanish. Setting $\lambda_1 = \lambda_2$, this implies that $x(\lambda) \in 2\mathbb Z$ for all $\lambda\in\Lambda$, so we may write $x = 2y$ for a character $y : \pi_1G \rightarrow \mathbb Z$ and there holds $Q(\lambda) = -2y(\lambda)^2$. The fact that $(Q, \tilde{\Lambda}, \varphi, x)$ lies in the fiber also supplies us with a trivialization of $\widetilde{\Lambda}$ compatible with $\varphi$. This yields an isomorphism between $(Q, \tilde{\Lambda}, \varphi, x)$ and the image of $y$ under \eqref{eq-classification-squaring-map}.
\end{proof}

\begin{void}
The compatibility statement asserts that \eqref{eq-classification-defining-sequence} coincides with the cofiber sequence defining $\Ktheory{}_{[1, 2]}^{\super}$ (see \S\ref{void-super-k-spectrum-definition}) evaluated at $\deloop G$:
\begin{equation}
\label{eq-classification-defining-sequence-matchup}
\begin{tikzcd}[column sep = 1em]
	\underline{\Gamma}{}_e(\deloop G, \deloop \Ktheory{}_1) \ar[r, "\Sq"]\ar[d, "\cong"] & \underline{\Gamma}{}_e(\deloop G, \Ktheory{}_{[1, 2]}) \ar[r]\ar[d, "\eqref{eq-classification-integral}"] & \underline{\Gamma}{}_e(\deloop G, \Ktheory{}_{[1, 2]}^{\super}) \ar[d, "\eqref{eq-classification-super}"] \\
	\Hom(\pi_1G, \mathbb Z) \ar[r, "\eqref{eq-classification-squaring-map}"] & \vartheta_G^{\mathbb Z}(\Lambda) \ar[r, "\eqref{eq-classification-reduction-mod-2}"] & \vartheta_G^{\super}(\Lambda)
\end{tikzcd}
\end{equation}

The following statement can be verified without any knowledge of \eqref{eq-classification-super}.
\end{void}

\begin{lem}
\label{lem-classification-squaring-map}
The left square in \eqref{eq-classification-defining-sequence-matchup} is canonically commutative.
\end{lem}
\begin{proof}
We use the expression \eqref{eq-section-via-pairing} of the map $\Sq$ as the difference:
\begin{equation}
\label{eq-section-via-pairing-classification}
\deloop\Ktheory{}_1 \rightarrow \Ktheory{}_{[1, 2]},\quad \cal L \mapsto s(\cal L^2) - 2\{\cal L, \cal L\}.
\end{equation}

Under the equivalence \eqref{eq-classification-integral}, $s$ corresponds to the natural inclusion of $\Hom(\pi_1G, \mathbb Z)$ in $\vartheta_G^{\mathbb Z}(\Lambda)$, while the map $\cal L\mapsto\{\cal L, \cal L\}$ corresponds to the restriction of \eqref{eq-classification-integral-cocycle} along the diagonal copy of $\Hom(\pi_1G, \mathbb Z)$ (established in the proof of Proposition \ref{prop-classification-integral}). The map induced from \eqref{eq-section-via-pairing-classification} upon taking $\underline{\Gamma}{}_e(\deloop G, \cdot)$ is thus readily computed to be \eqref{eq-classification-squaring-map}.
\end{proof}

\subsection{Central extensions by $\Ktheory{}_1$}
\label{sec-central-extension-k1}

\begin{void}
In this subsection, we let $S$ be an arbitrary base scheme and $G \rightarrow S$ be a reductive group scheme. Our goal is to classify central extensions of $G$ by $\Ktheory{}_1 \cong \mathbb G_m$.

When $S$ is the spectrum of a field, this classification is obtained by Weissman \cite[Theorem 1.11]{MR2798431}. We give a self-contained proof valid over any base scheme.
\end{void}

\begin{void}
For a reductive group scheme $H\rightarrow S$, we write $\Rad(H)$ for the radical of $H$ as defined in \cite[XXII, D\'efinition 4.3.6]{SGA3}. Namely, it is the maximal torus of the center of $H$. The formation of $\Rad(H)$ is stable under base change and recovers the classical notion (maximal connected normal solvable subgroup) over a geometric point of $S$.

Given a central extension of a reductive group scheme $H$ by $\mathbb G_m$ (or any torus):
\begin{equation}
\label{eq-central-extension-reductive-group-scheme-k1}
1 \rightarrow \mathbb G_m \rightarrow \widetilde H \rightarrow H\rightarrow 1,
\end{equation}
we first observe that $\widetilde H$ is representable by a reductive group scheme. Indeed, one checks directly that $\widetilde H\rightarrow S$ is smooth and its geometric fibers have vanishing unipotent radicals.
\end{void}

\begin{void}
By functoriality of the algebraic fundamental group, we obtain a functor from the Picard groupoid of central extension of $G$ by $\mathbb G_m$ to that of extensions of $\pi_1G$ by $\mathbb Z$ as sheaves of abelian groups:
\begin{equation}
\label{eq-fundamental-group-applied-to-central-extensions}
	\Hom_{\mathbb E_1}(G, \deloop\mathbb G_m) \rightarrow \Hom_{\mathbb Z}(\pi_1G, \deloop\mathbb Z).
\end{equation}
\end{void}

\begin{prop}
\label{prop-central-extension-multiplicative}
The functor \eqref{eq-fundamental-group-applied-to-central-extensions} is an equivalence.
\end{prop}

\begin{void}
The Picard groupoids in \eqref{eq-fundamental-group-applied-to-central-extensions} are of \'etale local nature on $S$, so we may assume the existence of a maximal torus $T\subset G$ in the proof of Proposition \ref{prop-central-extension-multiplicative}.

Since the $G$-conjugation extends along the map $G_{\mathrm{sc}} \rightarrow G$, the quotient stack $G/G_{\mathrm{sc}}$ has a monoidal structure. As such, we have isomorphisms of monoidal stacks:
\begin{equation}
\label{eq-quotient-stack-equivalences}
G/G_{\mathrm{sc}} \xleftarrow{\simeq} T/T_{\mathrm{sc}} \xrightarrow{\simeq} \pi_1(G)\otimes\mathbb G_m.
\end{equation}
Here, the tensor product is understood in the derived sense and sheafified in the fppf, or equivalently the \'etale topology.
\end{void}

\begin{proof}[Proof of Proposition \ref{prop-central-extension-multiplicative}]
In view of the isomorphisms \eqref{eq-quotient-stack-equivalences}, it suffices to prove that the following two forgetful functors are equivalences:
\begin{align}
\label{eq-central-extension-fundamental-group-forgetful}
	\Hom_{\mathbb Z}(T/T_{\mathrm{sc}}, \deloop\mathbb G_m) &\rightarrow \Hom_{\mathbb E_1}(T/T_{\mathrm{sc}}, \deloop\mathbb G_m),\\
\label{eq-central-extension-simply-connected-forgetful}
	\Hom_{\mathbb E_1}(G/G_{\mathrm{sc}}, \deloop\mathbb G_m) &\rightarrow \Hom_{\mathbb E_1}(G, \deloop\mathbb G_m).
\end{align}

Indeed, the left-hand-side of \eqref{eq-central-extension-fundamental-group-forgetful} is identified with $\Hom_{\mathbb Z}(\pi_1G, \deloop\mathbb Z)$ by the vanishing of $\underline{\Ext}^1(-, \mathbb G_m)$ on the category of fppf sheaves of abelian groups.

Given a central extension of a reductive group scheme $H$ by $\mathbb G_m$ as in \eqref{eq-central-extension-reductive-group-scheme-k1}, we have a short exact sequence of tori:
\begin{equation}
\label{eq-central-extension-reductive-group-scheme-k1-radical}
1 \rightarrow \mathbb G_m \rightarrow \Rad(\widetilde H) \rightarrow\Rad(H) \rightarrow 1.
\end{equation}
Indeed, the fact that $\Rad(\widetilde H)\rightarrow \Rad(H)$ is surjective can be checked on geometric fibers. Moreover, $\Rad(\widetilde H)$ contains $\mathbb G_m$ since the latter is central, so the inclusion of the kernel $\Rad(\widetilde H)\cap \mathbb G_m$ inside $\mathbb G_m$ is an isomorphism.

We make two observations:
\begin{enumerate}
	\item If $H$ is a torus, then so is $\widetilde H$. This is because the map $\Rad(\widetilde H) \rightarrow\widetilde H$ is an isomorphism by comparing \eqref{eq-central-extension-reductive-group-scheme-k1} with \eqref{eq-central-extension-reductive-group-scheme-k1-radical}.
	\item If $H$ is semisimple, then we find an isomorphism $\mathbb G_m\xrightarrow{\simeq}\Rad(\widetilde H)$.
\end{enumerate}

To prove that \eqref{eq-central-extension-fundamental-group-forgetful} is an equivalence, it suffices to show that any central extension of a torus by $\mathbb G_m$ is commutative. This follows from observation (1).

To prove that \eqref{eq-central-extension-simply-connected-forgetful} is an equivalence, we first write the left-hand-side as the groupoid of central extensions:
\begin{equation}
\label{eq-central-extension-reductive-group-by-multiplicative}
1 \rightarrow \mathbb G_m \rightarrow \widetilde G \rightarrow G \rightarrow 1,
\end{equation}
equipped with a $\widetilde G$-equivariant splitting over $G_{\mathrm{sc}}$ for the adjoint action. Our task is to show that such a splitting exists uniquely.

To construct such a splitting, we may assume that $G$ is simply connected in \eqref{eq-central-extension-reductive-group-by-multiplicative}. Let $\widetilde G_{\der} \subset \widetilde G$ denote its derived subgroup. We claim that the composition:
\begin{equation}
\label{eq-derived-subgroup-lifting}
\widetilde G_{\der}\subset\widetilde G \rightarrow G
\end{equation}
is an isomorphism.

It suffices to prove that \eqref{eq-derived-subgroup-lifting} is a central isogeny, \emph{i.e.}~it is finite, flat, and surjective, with kernel contained in the center of $\widetilde G_{\der}$. The statement on the kernel is clear. The fact that \eqref{eq-derived-subgroup-lifting} is finite, finite, and surjective may be established smooth locally, so we base change along $\widetilde G \rightarrow G$, where \eqref{eq-derived-subgroup-lifting} becomes the multiplication map:
$$
\widetilde G_{\der} \times \mathbb G_m \rightarrow \widetilde G.
$$
However, by observation (2), this morphism is identified with the isogeny $\widetilde G_{\der}\times \Rad(\widetilde G) \rightarrow\widetilde G$ of \cite[XXII, 6.2.3]{SGA3}.

The isomorphism \eqref{eq-derived-subgroup-lifting} for $G$ simply connected equips \eqref{eq-central-extension-reductive-group-by-multiplicative} with a section over $G_{\mathrm{sc}}$. It is unique since any two sections differ by a character $G_{\mathrm{sc}} \rightarrow\mathbb G_m$ which is necessarily trivial. To see that this section is $\widetilde G$-equivariant, it suffices to observe that the diagram:
$$
\begin{tikzcd}[column sep = 1em]
	\widetilde G\times_GG_{\mathrm{sc}} \ar[r]\ar[d] & G_{\mathrm{sc}} \ar[d] \\
	\widetilde G \ar[r] & G
\end{tikzcd}
$$
is $\widetilde G$-equivariant, and any automorphism of $\widetilde G \times_G G_{\mathrm{sc}}$ preserves its derived subgroup.
\end{proof}

\subsection{Proof of Theorem \ref{thm-classification-super}}
\label{sec-proof-classifcation-super}

\begin{void}
We return to the set-up of \S\ref{void-classification-setup}. In particular, the base scheme $S$ is assumed to be regular and of finite type over a field. In this subsection, we construct the equivalence \eqref{eq-classification-super} and thereby prove Theorem \ref{thm-classification-super}.

We shall construct this equivalence in two stages: we first do it when $\pi_1G$ is torsion-free and satisfies a Galois cohomological condition. This step uses Proposition \ref{prop-classification-integral} and Proposition \ref{prop-central-extension-multiplicative}. We then bootstrap the general case from this one, using the flasque resolution over general base due to Gonz\'alez-Avil\'es \cite{MR3085137}.

The fact that we have to play with Galois cohomology is because we do not know \emph{a priori} that $\underline{\Gamma}{}_e(\deloop G, \Ktheory{}_{[1, 2]}^{\super})$ satisfies \'etale descent.
\end{void}

\begin{void}
Note that our hypothesis on $S$ guarantees that every $S$-tori is isotrivial, \emph{i.e.}~split by a finite \'etale cover \cite[X, Th\'eor\`eme 5.16]{SGA3}. In particular, it makes sense for an $S$-torus to be quasi-trivial, see \cite[Definition 1.2]{MR878473}.
\end{void}

\begin{lem}
\label{lem-torsion-free}
If $\pi_1G$ is the sheaf of cocharacters of a quasi-trivial torus, then both rows in \eqref{eq-classification-defining-sequence-matchup} are cofiber sequences of Zariski sheaves.
\end{lem}
\begin{proof}
This assertion amounts to the Zariski local surjectivity of the two horizontal morphisms appearing in \eqref{eq-classification-defining-sequence-matchup}:
\begin{align*}
f_1 &: \underline{\Gamma}{}_e(\deloop G, \Ktheory{}_{[1, 2]}) \rightarrow \underline{\Gamma}{}_e(\deloop G, \Ktheory{}_{[1,2]}^{\super}), \\
f_2 &: \vartheta_G^{\mathbb Z}(\Lambda) \rightarrow \vartheta_G^{\super}(\Lambda).
\end{align*}

For $f_2$, we note that comparing \eqref{eq-classification-integral-fiber-sequence} with \eqref{eq-classification-super-fiber-sequence} leads to a Cartesian square:
\begin{equation}
\label{eq-mod-2-morphism-classification}
\begin{tikzcd}[column sep = 1em]
	\vartheta_G^{\mathbb Z}(\Lambda) \ar[r, twoheadrightarrow]\ar[d, "f_2"] & \underline{\Hom}(\pi_1G, \mathbb Z) \ar[d, "\text{mod }2"] \\
	\vartheta_G^{\super}(\Lambda) \ar[r] & \underline{\Hom}(\pi_1G, \mathbb Z/2)
\end{tikzcd}
\end{equation}
\emph{Claim}: the ``mod $2$'' morphism is surjective in the Zariski topology.

Indeed, Zariski locally on $S$, we may find a finite Galois cover $S_1 \rightarrow S$ which splits $\pi_1G$. Denote by $\Gamma$ the Galois group of $S_1/S$ and $M$ the $\mathbb Z$-linear dual of the $\Gamma$-module associated to $\pi_1G$ at a geometric point of $S$. Then the problem amounts to the surjectivity of $M^{\Gamma} \rightarrow (M/2)^{\Gamma}$, which follows from $H^1(\Gamma, M) = 0$ by quasi-triviality.

It follows that $f_2$ is also surjective in the Zariski topology.

For $f_1$, the canonical maps in \eqref{eq-super-k-spectrum-fiber-sequences} induce a Cartesian square:
\begin{equation}
\label{eq-mod-2-compatibility-with-grading}
\begin{tikzcd}[column sep = 1em]
	\underline{\Gamma}{}_e(\deloop G, \Ktheory{}_{[1, 2]}) \ar[r, twoheadrightarrow]\ar[d, "f_1"] & \underline{\Gamma}{}_e(\deloop G, B\Ktheory{}_1) \ar[d, "\text{mod }2"] \\
	\underline{\Gamma}{}_e(\deloop G, \Ktheory{}^{\super}_{[1, 2]}) \ar[r] & \underline{\Gamma}{}_e(\deloop G, \deloop\Ktheory{}_1/2)
\end{tikzcd}
\end{equation}
Proposition \ref{prop-central-extension-multiplicative} implies that the ``mod $2$'' morphism is identified with the one appearing in \eqref{eq-mod-2-morphism-classification}. In particular, it is also surjective in the Zariski topology given the hypothesis on $\pi_1G$. The same thus holds for $f_1$.
\end{proof}

\begin{void}
\label{void-classification-super-torsion-free}
Suppose that $\pi_1G$ is the sheaf of cocharacters of a quasi-trivial torus. By Lemma \ref{lem-classification-squaring-map} and Lemma \ref{lem-torsion-free}, we may define a morphism fitting into \eqref{eq-classification-defining-sequence-matchup}:
\begin{equation}
\label{eq-classification-super-torsion-free}
\underline{\Gamma}{}_e(\deloop G, \Ktheory_{[1,2]}^{\super}) \rightarrow \vartheta_G^{\super}(\Lambda).
\end{equation}
It is an equivalence by Proposition \ref{prop-classification-integral}.
\end{void}

\begin{void}
\label{void-auxiliary-theta-data}
For general $G$, we first introduce an auxiliary sheaf of Picard groupoids $\widetilde{\vartheta}_G^{\super}(\Lambda)$, defined to be the fiber product:
\begin{equation}
\label{eq-auxiliary-theta-data}
\begin{tikzcd}[column sep = 1em]
	\widetilde{\vartheta}_G^{\super}(\Lambda) \ar[r]\ar[d] & \Quad(\Lambda_{\mathrm{sc}})^W \ar[d, "\eqref{eq-canonical-theta-data-brylinski-deligne}"] \\
	\vartheta^{\super}(\Lambda) \ar[r] & \vartheta^{\super}(\Lambda_{\mathrm{sc}})
\end{tikzcd}
\end{equation}
where the bottom horizontal map is defined by functoriality with respect to $\Lambda_{\mathrm{sc}} \rightarrow \Lambda$. Concretely, a section of $\widetilde{\vartheta}_G^{\super}(\Lambda)$ is a triple $(b, \widetilde{\Lambda}, \varphi)$ as in $\vartheta_G^{\super}(\Lambda)$, but the Weyl-invariance on $b$ is relaxed: it is only required to be Weyl-invariant over $\Lambda_{\mathrm{sc}}$.

Restrictions along $T\subset G$, $G_{\mathrm{sc}} \rightarrow G$ and applying the functor \eqref{eq-classification-super-torsion-free} to $T$, $G_{\mathrm{sc}}$, and $T_{\mathrm{sc}}$ produces a functor:
\begin{equation}
\label{eq-classification-super-relaxed}
\underline{\Gamma}{}_e(\deloop G, \Ktheory{}_{[1, 2]}^{\super}) \rightarrow \widetilde{\vartheta}_G^{\super}(\Lambda).
\end{equation}
\end{void}

\begin{void}
\label{void-z-extension}
Suppose that we have a central extension of reductive group $S$-schemes:
\begin{equation}
\label{eq-coflasque-extension}
1 \rightarrow T_1 \rightarrow \widetilde G \rightarrow G \rightarrow 1,
\end{equation}
where $T_1$ is a torus with sheaf of cocharacters $\Lambda_1$. Denote by $\widetilde T$ the preimage of $T$ in $\widetilde G$. It is a maximal torus with sheaf of cocharacters $\widetilde{\Lambda}$.

By functoriality, we find two morphisms of presheaves of Picard groupoids:
\begin{align}
	\label{eq-z-extension-super}
	\underline{\Gamma}{}_e(\deloop G, \Ktheory{}_{[1, 2]}^{\super}) & \rightarrow \lim_{[n]} \underline{\Gamma}{}_e(\deloop(\widetilde G\times T_1^{\times n}), \Ktheory{}_{[1, 2]}^{\super}), \\
	\label{eq-z-extension-classification}
	\vartheta_G^{\super}(\Lambda) & \rightarrow \lim_{[n]}\vartheta_{\widetilde G\times T_1^{\times n}}^{\super}(\widetilde{\Lambda} \oplus \Lambda_1^{\oplus n}),
\end{align}
where the limits are taken over the simplicial category.
\end{void}

\begin{lem}
\label{lem-z-extension-totalization}
The following statements hold:
\begin{enumerate}
	\item the functor \eqref{eq-z-extension-super} is an equivalence;
	\item the functor \eqref{eq-z-extension-classification} is fully faithful.
\end{enumerate}
\end{lem}
\begin{proof}
Statement (1) follows from the fact that the induced map on Zariski classifying stacks $\deloop \widetilde G \rightarrow \deloop G$ is surjective in the Zariski topology.

To prove statement (2), we fit $\vartheta_G^{\super}(\Lambda)$ into a fiber sequence of \'etale sheaves of Picard groupoids over $S$:
\begin{equation}
\label{eq-super-theta-data-canonical-sequence}
\underline{\Hom}{}_{\mathbb Z}(\Lambda, \deloop\mathbb G_m) \rightarrow \vartheta_G^{\super}(\Lambda) \rightarrow \Gamma^2(\check{\Lambda})^W.
\end{equation}
Here, $\check{\Lambda}$ is the dual of $\Lambda$, so $\Gamma^2(\check{\Lambda})^W$ is the abelian group of Weyl-invariant symmetric bilinear forms on $\Lambda$. The second map in \eqref{eq-super-theta-data-canonical-sequence} sends a triple $(b, \widetilde{\Lambda}, \varphi)$ to $b$, so its fiber is precisely the Picard groupoid of symmetric monoidal morphisms $\Lambda \rightarrow \deloop\mathbb G_m$.

The fully faithfulness will follow, if we know that the two outer terms in \eqref{eq-super-theta-data-canonical-sequence} satisfy descent along $\widetilde{\Lambda} \rightarrow \Lambda$. For $\underline{\Hom}{}_{\mathbb Z}(\Lambda, \deloop\mathbb G_m)$, this is because $\Lambda$ is identified with $\colim_{[n]}(\widetilde{\Lambda} \oplus \Lambda_1^{\oplus n})$. For $\Gamma^2(\check{\Lambda})^W$, this is the elementary observation that a symmetric bilinear form on $\widetilde{\Lambda}$ descends to $\Lambda$ if its restrictions to $\widetilde{\Lambda}\oplus\Lambda_1$ along the action and projection maps coincide.
\end{proof}

\begin{rem}
In fact, the functor \eqref{eq-z-extension-classification} is also an equivalence. This will be established in the course of the proof of Theorem \ref{thm-classification-super} below.
\end{rem}

\begin{proof}[Proof of Theorem \ref{thm-classification-super}]
The case where $\pi_1G$ is the sheaf of cocharacters of a quasi-trivial torus is already treated in \S\ref{void-classification-super-torsion-free}.

For general $G$, it remains to prove that the functor \eqref{eq-classification-super-relaxed} factors through an equivalence onto the full subgroupoid $\vartheta_G^{\super}(\Lambda)$.

To do so, we choose a central extension \eqref{eq-coflasque-extension} with the additional property that $\pi_1\widetilde G$ is the sheaf of cocharacters of a quasi-trivial torus. Such central extensions exist, thanks to \cite[Proposition 3.2]{MR3085137}.

Combining the equivalences for $\widetilde G\times T_1^{\times n}$ and Lemma \ref{lem-z-extension-totalization}, we obtain the following (solid) functors among Zariski sheaves of Picard groupoids:
\begin{equation}
\label{eq-classification-super-descent}
\begin{tikzcd}[column sep = 1em]
	\underline{\Gamma}{}_e(\deloop G, \Ktheory{}_{[1, 2]}^{\super}) \ar[r, "\simeq"]\ar[d, dashrightarrow] & \lim_{[n]}\underline{\Gamma}{}_e(\deloop (\widetilde G\times T_1^{\times n}), \Ktheory{}_{[1, 2]}^{\super}) \ar[d, "\cong"] \\
	\vartheta_G^{\super}(\Lambda) \ar[r, phantom, "\subset"] & \lim_{[n]}\vartheta_{\widetilde G\times T_1^{\times n}}^{\super}(\widetilde{\Lambda} \oplus \Lambda_1^{\oplus n})
\end{tikzcd}
\end{equation}

Note that a symmetric bilinear form on $\Lambda$ is Weyl-invariant if and only if its restriction to $\widetilde{\Lambda}$ is. Hence, the functor \eqref{eq-classification-super-relaxed} factors through the full subgroupoid $\vartheta_G^{\super}(\Lambda)$, supplying the dashed arrow in \eqref{eq-classification-super-descent}. It follows that all functors in \eqref{eq-classification-super-descent} are equivalences.
\end{proof}

\begin{cor}
Let $G$ be a reductive group $S$-scheme. The Zariski sheaf of Picard groupoids $\underline{\Gamma}{}_e(\deloop G, \Ktheory{}_{[1, 2]}^{\super})$ over $S$ satisfies \'etale descent.
\end{cor}
\begin{proof}
Working Zariski locally over $S$, we may assume that $G$ admits a maximal torus $T$ \cite[XIV, Corollaire 3.20]{SGA3}. Let $\Lambda$ denote its sheaf of cocharacters.

Theorem \ref{thm-classification-super} then implies that $\underline{\Gamma}{}_e(\deloop G, \Ktheory{}_{[1, 2]}^{\super})$ is equivalent to $\vartheta_G^{\super}(\Lambda)$, which clearly satisfies \'etale descent.
\end{proof}

\begin{void}
We finish our study of $\underline{\Gamma}{}_e(\deloop G, \Ktheory{}_{[1, 2]}^{\super})$ by computing its homotopy sheaves. Let us assume that $T\subset G$ is a fixed maximal torus with sheaf of cocharacters $\Lambda$.

From the Cartesian square \eqref{eq-auxiliary-theta-data}, we obtain a long exact sequence of \'etale sheaves of abelian groups:
\begin{align*}
1 \rightarrow \pi_1\widetilde{\vartheta}_G^{\super}(\Lambda) &\rightarrow \pi_1\vartheta^{\super}(\Lambda) \rightarrow \pi_1\vartheta^{\super}(\Lambda_{\mathrm{sc}}) \\
&\rightarrow \pi_0\widetilde{\vartheta}_G^{\super}(\Lambda) \rightarrow \pi_0\vartheta^{\super}(\Lambda) \oplus \Quad(\Lambda_{\mathrm{sc}})^W \rightarrow \pi_0\vartheta^{\super}(\Lambda_{\mathrm{sc}}).
\end{align*}

The homotopy sheaves of $\vartheta^{\super}(\Lambda)$ are easily computed: $\pi_1\vartheta^{\super}(\Lambda)$ is isomorphic to $\underline{\Hom}(\Lambda, \mathbb G_m)$, and $\pi_0\vartheta^{\super}(\Lambda)$ is isomorphic to the sheaf of symmetric bilinear forms on $\Lambda$. Therefore, $\pi_1\widetilde{\vartheta}_G^{\super}(\Lambda)$ is isomorphic to $\underline{\Hom}(\pi_1G, \mathbb G_m)$, and $\pi_0\widetilde{\vartheta}_G^{\super}(\Lambda)$ is isomorphic to the sheaf of symmetric bilinear forms on $\Lambda$ whose restriction to $\Lambda_{\mathrm{sc}}$ comes from a Weyl-invariant quadratic form.

By definition of the full subgroupoid $\vartheta_G^{\super}(\Lambda) \subset \widetilde{\vartheta}_G^{\super}(\Lambda)$, we see that $b \in \pi_0\widetilde{\vartheta}_G^{\super}(\Lambda)$ belongs to $\pi_0\vartheta_G^{\super}(\Lambda)$ if and only if it is Weyl-invariant.

Writing $\Gamma^2(\check{\Lambda})^W_{\mathrm{sc}}$ for the abelian sheaf of Weyl-invariant symmetric bilinear forms on $\Lambda$ whose restriction to $\Lambda_{\mathrm{sc}}$ comes from a quadratic form, we obtain:
$$
\pi_i\underline{\Gamma}{}_e(BG, \Ktheory{}_{[1, 2]}^{\super}) \cong 
\begin{cases}
\Gamma^2(\check{\Lambda})^W_{\mathrm{sc}} & i = 0, \\
\underline{\Hom}(\pi_1G, \mathbb G_m) & i = 1, \\
0 & i\ge 2.
\end{cases}
$$
\end{void}

\subsection{Examples}
\label{sec-examples}

\begin{void}
Recall the canonical identification between $\Ktheory{}_{[0, 1]}$ and $\Pic^{\mathbb Z}$ (Remark \ref{rem-determinant-via-k-theory}). It induces an isomorphism $\deloop\Ktheory{}_1 \cong \deloop\mathbb G_m$, whose inverse can be viewed as a rigidified section of $\deloop\Ktheory{}_1$ over $\deloop\mathbb G_m$ and we denote it by $c_1$. It represents the K-theoretic first Chern class.

Using the cup product and the pairing $\Ktheory{}_1 \otimes \Ktheory{}_1 \rightarrow \Ktheory{}_2$, $x, y\mapsto \{x, y\}$, we obtain a rigidified section $c_1\cup c_1$ of $\deloop^2\Ktheory{}_2$ over $\deloop\mathbb G_m$. Its value at an $R$-point $\mathscr L$ is given by the pairing $\{\mathscr L, \mathscr L\}$ (\emph{cf.}~Remark \ref{rem-k-theory-pairing}). The central extension of $\mathbb G_m$ by $\Ktheory{}_2$ corresponding to $c_1\cup c_1$ is given by $\Ktheory{}_2\times \mathbb G_m$ with $x, y\mapsto \{x, y\}$ as cocycle. The quadratic form associated to $c_1 \cup c_1$ (\emph{cf.}~Remark \ref{rem-brylinski-deligne-functor}) takes value $1$ at the identity cocharacter of $\mathbb G_m$.
\end{void}

\begin{void}\emph{The Tate section.}
We have a rigidified section of $\Ktheory{}_{[1, 2]}$ over $\deloop\mathbb G_m$, sending $\mathscr L$ to $[\mathscr L] - [\mathscr O]$. We shall write $\Tate$ for the induced rigidified section of $\Ktheory{}_{[1, 2]}^{\super}$.

Let us show that as a rigidified section of $\Ktheory{}_{[1, 2]}^{\super}$, $c_1\cup c_1$ is twice the Tate extension.
\end{void}

\begin{lem}
\label{lem-cup-product-twice-tate}
There is a canonical isomorphism of rigidified sections of $\Ktheory{}_{[1, 2]}^{\super}$ over $\deloop\mathbb G_m$:
$$
2\cdot \Tate \cong c_1\cup c_1.
$$
\end{lem}
\begin{proof}
We need to establish an isomorphism in $\Ktheory{}_{[1, 2]}^{\super}$:
$$
2\cdot([\mathscr L] - [\mathscr O]) \cong \{\mathscr L, \mathscr L\}
$$
natural in the line bundle $\mathscr L$. Note that $[\mathscr L] - [\mathscr O]$ is denoted by $s(\mathscr L)$ in \S\ref{void-truncated-k-theory-section} and the proof of Lemma \ref{lem-duplication-section-loop} yields isomorphisms in $\Ktheory{}_{[1, 2]}$:
\begin{align*}
	[\mathscr L] - [\mathscr L^{-1}] &\cong s(\mathscr L^2) - 2\cdot \{\mathscr L, \mathscr L\} \\
	&\cong 2\cdot s(\mathscr L) - \{\mathscr L, \mathscr L\}.
\end{align*}

By definition, the section $[\mathscr L] - [\mathscr L^{-1}]$ vanishes in $\Ktheory{}_{[1, 2]}$, which gives rise to the desired isomorphism $2\cdot s(\mathscr L) \cong \{\mathscr L, \mathscr L\}$.
\end{proof}

\begin{rem}
\label{rem-tate-section-quadratic}
The rigidified section $\Tate : \deloop\mathbb G_m \rightarrow \Ktheory{}_{[1, 2]}^{\super}$ has the following quadratic multiplicative structure:
$$
\Tate(\mathscr L^n) \cong n^2\cdot \Tate(\mathscr L).
$$

Indeed, inducting on $n$ using the relation \eqref{eq-section-quadratic-relation} yields an isomorphism between $\Tate(\mathscr L^n)$ and $n\cdot \Tate(\mathscr L) + \binom{n}{2}\{\mathscr L, \mathscr L\}$, but the latter is $n^2\cdot \Tate(\mathscr L)$ by Lemma \ref{lem-cup-product-twice-tate}.
\end{rem}

\begin{void}
\label{void-SL2-twice-tate}
Consider the group scheme $\SL_2$ equipped with the diagonal maximal torus $\mathbb G_m \subset \SL_2$. Write $V$ for the universal rank-$2$ vector bundle over $\deloop\SL_2$. The section $[V] - [\mathscr O^{\oplus 2}]$ of $\Ktheory$ over $\deloop\SL_2$ factors through $\Ktheory{}_{\ge 2}$, so it induces a section of $\deloop^2\Ktheory{}_2$ over $\deloop \SL_2$.

The pullback of $[V] - [\mathscr O^{\oplus 2}]$ to $\deloop\mathbb G_m$ is the rigidified section sending $\mathscr L$ to $[\mathscr L \oplus \mathscr L^{-1}] - [\mathscr O^{\oplus 2}]$. The proof of Lemma \ref{lem-duplication-section-loop} yields isomorphisms:
\begin{align}
\notag
	[\mathscr L \oplus \mathscr L^{-1}] - [\mathscr O^{\oplus 2}] & \cong ([\mathscr L] - [\mathscr O]) + ([\mathscr L^{-1}] - [\mathscr O]) \\
\label{eq-SL2-twice-tate}
	& \cong \{\mathscr L, \mathscr L\}.
\end{align}

In other words, the pullback of $[V] - [\mathscr O^{\oplus 2}]$ to $\deloop\mathbb G_m$ is canonically isomorphic to $c_1\cup c_1$. In particular, it is isomorphic to twice the Tate section (Lemma \ref{lem-cup-product-twice-tate}).
\end{void}

\begin{rem}
\label{rem-standard-representation-minimal-form}
It follows from \S\ref{void-SL2-twice-tate} that the rigidified section $[V] - [\mathscr O^{\oplus 2}] : \deloop\SL_2 \rightarrow \deloop^2\Ktheory{}_2$ is classified by the Weyl-invariant quadratic form whose value at a coroot is $1$.

More generally, for any integer $n\ge 2$, we may consider the universal rank-$n$ vector bundle $V$ over $\deloop \SL_n$. The section $[V] - [\mathscr O^{\oplus n}]$ of $\deloop^2\Ktheory{}_2$ over $\deloop\SL_n$ is also classified by the Weyl-invariant quadratic form whose value at a coroot is $1$, with respect to the diagonal maximal torus. This follows by choosing a subgroup $\SL_2\subset\SL_n$ corresponding to a simple coroot and reducing to the case of $\SL_2$.
\end{rem}

\begin{void}\emph{Pfaffian and all that.}
\label{void-pfaffian-context}
Let $G$ be a split reductive group with fixed split maximal torus and Borel subgroup $T \subset B \subset G$ and a pinning. Let $\Ad$ denote the adjoint bundle over $\deloop G$, \emph{i.e.}~the vector bundle associated to the adjoint representation $\mathfrak g$.

We shall construct a ``half'' of $[\Ad] - [\mathscr O^{\dim\fr g}]$ as a rigidified section of $\Ktheory{}_{[1, 2]}^{\super}$.
\end{void}

\begin{prop}
\label{prop-pfaffian-construction}
In the context of \S\ref{void-pfaffian-context}, there is a canonical rigidified section $\Pf$ of $\Ktheory{}_{[1, 2]}^{\super}$ over $\deloop G$ equipped with an isomorphism:
\begin{equation}
\label{eq-pfaffian-square-isomorphism}
2\cdot \Pf \cong [\Ad] - [\mathscr O^{\dim \fr g}].
\end{equation}
\end{prop}

\begin{void}
We shall use the classification theorem (Theorem \ref{thm-classification-super}) to construct $\Pf$ and the isomorphism \eqref{eq-pfaffian-square-isomorphism}.

Denote by $\Lambda$ the cocharacter lattice of $T$ and $\Lambda_{\mathrm{sc}}$ that of the induced maximal torus $T_{\mathrm{sc}}$ of the simply connected form $G_{\mathrm{sc}}$ of $G$. The choice of $B$ endows $\Lambda_{\mathrm{sc}}$ with a basis consisting of simple coroots $\alpha\in\Delta$. The choice of a pinning induces a canonical extension of each $\alpha$ to a subgroup of $G_{\mathrm{sc}}$ isomorphic to $\SL_2$:
\begin{equation}
\label{eq-pinning-induced-root-subgroup}
\begin{tikzcd}
	\mathbb G_m \ar[r, phantom, "\subset"] \ar[d, "\alpha"] & \SL_2 \ar[d] \\
	T_{\mathrm{sc}} \ar[r, phantom, "\subset"] & G_{\mathrm{sc}}
\end{tikzcd}
\end{equation}
\end{void}

\begin{void}
\label{void-classification-pinning-formulation}
In the presence of a pinning, Theorem \ref{thm-classification-super} can be reformulated as classifying rigidified sections of $\Ktheory{}_{[1, 2]}^{\super}$ over $\deloop G$ by pairs $(f, \{\varphi_{\alpha}\}_{\alpha\in\Delta})$, where:
\begin{enumerate}
	\item $f$ is a rigidified section of $\Ktheory{}_{[1, 2]}^{\super}$ over $\deloop T$, whose associated symmetric bilinear form $b$ is Weyl-invariant and its restriction to $\Lambda_{\mathrm{sc}}$ comes from a quadratic form $Q_{\mathrm{sc}}$;
	\item for each $\alpha\in\Delta$, $\varphi_{\alpha}$ is an isomorphism between the restriction of $f$ along $\alpha : \deloop\mathbb G_m \rightarrow \deloop T$ and the $Q_{\mathrm{sc}}(\alpha)$-multiple of the $c_1\cup c_1$.
\end{enumerate}

Indeed, $f$ accounts for the data $(b, \widetilde{\Lambda})$ of \S\ref{void-theta-data-super}. To see that $\varphi$ of \emph{loc.cit.}~is equivalent to $\{\varphi_{\alpha}\}_{\alpha\in\Delta}$, we argue as follows: $\varphi$ is an isomorphism of two central extensions of $\Lambda_{\mathrm{sc}}$ by $\mathbb G_m$ with equal commutators, so it is uniquely determined over the basis $\Delta$. On the other hand, the rigidified section $\deloop G_{\mathrm{sc}} \rightarrow \deloop^2\Ktheory{}_2$ classified by $Q_{\mathrm{sc}}$ restricts to the $Q_{\mathrm{sc}}(\alpha)$-multiple of $c_1\cup c_1$ along $\alpha : \deloop\mathbb G_m \rightarrow \deloop T_{\mathrm{sc}} \rightarrow \deloop G_{\mathrm{sc}}$, in view of \S\ref{void-SL2-twice-tate} and \eqref{eq-pinning-induced-root-subgroup}.
\end{void}

\begin{void}
Let us turn to the construction of $\Pf$.
\end{void}

\begin{proof}[Proof of Proposition \ref{prop-pfaffian-construction}]
The symmetric bilinear form associated to the rigidified section $[\Ad] - [\mathscr O^{\dim\fr g}] : \deloop G \rightarrow \deloop^2\Ktheory{}_2$ is the Killing form (\emph{cf.}~Remark \ref{rem-standard-representation-minimal-form}):
$$
\lambda_1, \lambda_2\mapsto \sum_{\check{\beta} \in \Phi} \langle\check{\beta}, \lambda_1\rangle\langle\check{\beta}, \lambda_2\rangle,
$$
where $\Phi$ is the set of roots of $G$. The choice of $B$ expresses $\Phi$ as the union of positive roots $\Phi_+$ and negative roots $\Phi_-$. In particular, the Killing form is twice the symmetric bilinear form $b$ defined by:
$$
b(\lambda_1, \lambda_2) := \sum_{\check{\beta} \in \Phi_+}\langle\check{\beta}, \lambda_1\rangle\langle\check{\beta}, \lambda_2\rangle.
$$

In particular, we have $b(\lambda, \lambda) = \langle 2\check{\rho}, \lambda\rangle \mod 2$, for $2\check{\rho} := \sum_{\check{\beta} \in \Phi_+}\check{\beta}$, so $b(\alpha, \alpha)$ is even for all $\alpha\in\Delta$. This means that the restriction of $b$ to $\Lambda_{\mathrm{sc}}$ comes from a quadratic form $Q_{\mathrm{sc}}$.

Next, the restriction of $[\Ad] - [\mathscr O^{\dim\fr g}]$ to $\deloop T$ is the section:
$$
\sum_{\check{\beta}\in\Phi_+} ([\mathscr L^{\check{\beta}}] - [\mathscr O]) + \sum_{\check{\beta}\in\Phi_+} ([\mathscr L^{-\check{\beta}}] - [\mathscr O]) \cong 2\cdot \sum_{\check{\beta} \in \Phi_+} ([\mathscr L^{\check{\beta}}] - [\mathscr O]),
$$
where $\mathscr L^{\check{\beta}}$ denotes the line bundle over $\deloop T$ defined by the root $\check{\beta}$, and we applied the isomorphism $[\mathscr L^{\check{\beta}}] \cong [\mathscr L^{-\check{\beta}}]$ in $\Ktheory{}_{[1, 2]}^{\super}$. In particular, the restriction of $[\Ad] - [\mathscr O^{\dim\fr g}]$ to $\deloop T$ is twice the rigidified section:
\begin{equation}
\label{eq-pfaffian-over-maximal-torus}
\sum_{\check{\beta} \in \Phi_+}([\mathscr L^{\check{\beta}}] - [\mathscr O]).
\end{equation}

Note that the symmetric bilinear form associated to \eqref{eq-pfaffian-over-maximal-torus} is $b$. We shall argue that the restriction of \eqref{eq-pfaffian-over-maximal-torus} along each $\alpha : \deloop\mathbb G_m \rightarrow \deloop T$ is the $Q_{\mathrm{sc}}(\alpha)$-multiple of $c_1\cup c_1$. This would furnish the construction of $\Pf$ in view of \S\ref{void-classification-pinning-formulation}.

Indeed, the restriction of \eqref{eq-pfaffian-over-maximal-torus} along $\alpha$ is the rigidified section:
\begin{align*}
\sum_{\check{\beta} \in \Phi_+} ([\mathscr L^{\langle\check{\beta}, \alpha\rangle}] - [\mathscr O]) &\cong \sum_{\check{\beta} \in \Phi_+}\langle\check{\beta}, \alpha\rangle^2 \cdot\Tate \\
&\cong 2\cdot Q_{\mathrm{sc}}(\alpha) \cdot \Tate \cong Q_{\mathrm{sc}}(\alpha) \cdot (c_1\cup c_1)
\end{align*}
where the first isomorphism follows from the quadratic structure of the Tate section (Remark \ref{rem-tate-section-quadratic}) and the last isomorphism follows from Lemma \ref{lem-cup-product-twice-tate}.

The isomorphism $2\cdot \Pf\cong[\Ad] - (\mathscr O^{\dim\fr g})$ results directly from the construction of $\Pf$, so we omit the details.
\end{proof}

\begin{rem}
The image of $\Pf$ in $\underline{\Hom}(\pi_1G, \mathbb Z/2)$ under \eqref{eq-classification-super-fiber-sequence} is the character:
\begin{equation}
\label{eq-pfaffian-mod-two-character}
\lambda \mapsto \langle 2\check{\rho}, \lambda\rangle \mod 2.
\end{equation}
Thus, $\Pf$ comes from a rigidified section of $\deloop^2\Ktheory{}_2$ if and only if $\check{\rho}$ is integral.

The same fact also shows that $\Pf$ generally does \emph{not} come from a rigidified section of $\Ktheory_{[1, 2]}$, \emph{i.e.}~it is genuinely ``half-integral''. Indeed, if it did, then \eqref{eq-pfaffian-mod-two-character} would have to lift to character $\pi_1G \rightarrow \mathbb Z$, but this is generally not the case.
\end{rem}

\begin{void}
Let us combine Proposition \ref{prop-pfaffian-construction} and the integration functor of \S\ref{void-super-conformal-blocks} to construct the Pfaffian line bundle on the moduli stack of $G$-bundles over a spin curve.

More precisely, let $p : X_S \rightarrow S$ be a smooth, proper morphism of relative dimension $1$ with connected geometric fibers together with a square root $\omega^{1/2}$ of the relative canonical bundle. Denote by $\Bun_G$ the moduli stack of $G$-bundles over $X_S$. The rigidified relative canonical bundle of $\Bun_G \rightarrow S$ is the line bundle:
\begin{equation}
\label{eq-rigidified-relative-canonical-bundle}
\mathscr L_{\det} := \det(Rp_* \fr g_P) \otimes (\det (Rp_*\fr g_{P^0}))^{-1},
\end{equation}
where $p : X_S\times_S \Bun_G \rightarrow \Bun_G$ is the projection and $\fr g_P$ (resp.~$\fr g_{P^0}$) is the adjoint bundle of the universal $G$-bundle $P$ (resp.~trivial $G$-bundle $P^0$).
\end{void}

\begin{void}
The commutative diagram \eqref{eq-integration-along-curves-degree-super} yields the following commutative diagram via the construction of \S\ref{void-super-conformal-blocks}:
$$
\begin{tikzcd}[column sep = 1em]
	\Gamma(\deloop G, \deloop^2\Ktheory{}_2) \ar[d, "\int_{X_S}"] \ar[r] & \Gamma(\deloop G, \Ktheory{}_{[1, 2]}^{\super}) \ar[d, "\int_{(X_S, \omega^{1/2})}"] \\
	\Gamma(\Bun_G, \Pic) \ar[r] & \Gamma(\Bun_G, \Pic^{\super})
\end{tikzcd}
$$
Here, the horizontal morphisms are the tautological inclusions.

Note that $\mathscr L_{\det}$ is the image of $[\Ad] - [\mathscr O^{\dim\fr g}]$ under the left vertical functor. By Proposition \ref{prop-pfaffian-construction}, the image of $[\Ad] - [\mathscr O^{\dim\fr g}]$ in $\Gamma(\deloop G, \Ktheory{}_{[1, 2]}^{\super})$ is twice the Pfaffian section $\Pf$. In particular, $\mathscr L_{\det}$ admits a square root as a \emph{super} line bundle, given by:
$$
\mathscr L_{\Pf} := \int_{(X_S, \omega^{1/2})} \Pf.
$$
\end{void}

\begin{rem}
A square root of $\mathscr L_{\det}$ has been constructed in \cite[\S4]{MR2058353} using a different method. One feature of our construction is that it yields a purely group-theoretic object $\Pf$, while the spin curve $(X_S, \omega^{1/2})$ only appears in the integration functor.
\end{rem}

\part{Loop groups}

\section{Statements}
\label{sect-factorization}

The goal of this section is to state the classification of factorization super central extensions of $\cal LG$: Theorem \ref{thm-factorization-super-central-extension-classification}. The first two subsections \S\ref{sec-factorization-definition}, \S\ref{sec-loop-groups} review the notions of factorization structure and loop groups. In \S\ref{sec-contou-carrere}, we use the Contou-Carr\`ere symbol to define the notion of ``tame commutator'' and study its basic properties. In \S\ref{sec-factorization-classification}, we state the classification theorem of factorization super central extensions of $\cal LG$ and briefly indicate the structure of its proof.

We work over a ground field $k$. Let $X$ be a smooth curve over $k$.

\subsection{Factorization}
\label{sec-factorization-definition}

\begin{void}
Denote by $\Ran$ the presheaf whose $S$-points are nonempty finite subsets of $\Maps(S, X)$. We shall write an $S$-point of $\Ran$ as $x^I = (x^i)_{i\in I}$, where $I$ is a nonempty finite set.

Given an $S$-point $x^I$ of $\Ran$, we denote by $\Gamma_{x^I}$ the sum of the graphs $\Gamma_{x^i} \subset S\times X$ over $i\in I$ as effective Cartier divisors. Let $D_{x^I}$ be the completion of $S\times X$ along $\Gamma_{x^I}$ and $\mathring D_{x^I}$ be its open subscheme $D_{x^I}\backslash\Gamma_{x^I}$.

Two $S$-points $x^I$, $x^J$ of $\Ran$ are called \emph{disjoint} if $\Gamma_{x^I}\cap\Gamma_{x^J} = \emptyset$. Denote by $x^I\sqcup x^J$ the $S$-point of $\Ran$ given by their union.
\end{void}

\begin{rem}
\label{rem-ran-space-colimit-presentation}
For each nonempty finite set $I$, there is a tautological map $X^I \rightarrow \Ran$, sending an $S$-point $x^I$ of $X^I$ to the associated finite subset of $\Maps(S, X)$. The presheaf $\Ran$ is identified with the colimit of presheaves:
$$
\colim_I (X^I) \xrightarrow{\simeq} \Ran,
$$
indexed by the category of nonempty finite sets with surjections.
\end{rem}

\begin{void}
Let $\cal Y$ be a presheaf over $\Ran$. Given an $S$-point $x^I$ of $\Ran$, we write $\cal Y_{x^I}$ for the base change of $\cal Y$ along $x^I$.

The presheaf $\cal Y$ is called \emph{factorization} when we are supplied with a functorial system (in $S$) of isomorphisms for all disjoint pairs of $S$-points $(x^I, x^J)$ of $\Ran$:
\begin{equation}
\label{eq-factorization-isomorphism}
\varphi_{x^I, x^J} : \cal Y_{x^I\sqcup x^J} \xrightarrow{\simeq} \cal Y_{x^I} \times_S \cal Y_{x^J},
\end{equation}
satisfying the analogues of commutativity and associativity conditions. Namely, the following diagram commutes:
\begin{equation}
\label{eq-factorization-commutativity}
\begin{tikzcd}[column sep = 2em]
\cal Y_{x^I\sqcup x^J} \ar[r, "\varphi_{x^I, x^J}"]\ar[d, "\cong"] & \cal Y_{x^I} \times_S\cal Y_{x^J} \ar[d, "\cong"] \\
\cal Y_{x^J\sqcup x^I} \ar[r, "\varphi_{x^J, x^I}"] & \cal Y_{x^J}\times_S\cal Y_{x^I}
\end{tikzcd}
\end{equation}
where the left vertical arrow comes from the equality $x^I\sqcup x^J = x^J\sqcup x^I$ as $S$-points of $\Ran$ and the right vertical arrow is the map swapping the two factors; the following diagram commutes for pairwise disjoint $S$-points $(x^I, x^J, x^K)$ of $\Ran$:
\begin{equation}
\label{eq-factorization-associativity}
\begin{tikzcd}[column sep  = -4em]
	& \cal Y_{x^I\sqcup x^J\sqcup x^K} \ar[dr, "\varphi_{x^I\sqcup x^J, x^K}"]\ar[dl, swap, "\varphi_{x^I, x^J\sqcup x^K}"] \\
	\cal Y_{x^I} \times_S \cal Y_{x^J\sqcup x^K} \ar[dr, swap, "\id\times\varphi_{x^J, x^K}"] & & \cal Y_{x^I\sqcup x^J} \times_S \cal Y_{x^K} \ar[dl, "\varphi_{x^I, x^J}\times\id"] \\
	 & \cal Y_{x^I}\times_S\cal Y_{x^J} \times_S \cal Y_{x^K}
\end{tikzcd}
\end{equation}
\end{void}

\begin{void}
Let $\cal Y$ be a factorization presheaf such that $\cal Y_{x^I}$ satisfies fppf descent for each $S$-point $x^I$ of $\Ran$. (We do not impose fppf descent on $\cal Y$ because $\Ran$ itself does not satisfy \'etale descent, see \cite[Warning 2.4.4]{MR3887650}.)

A \emph{factorization super line bundle} over $\cal Y$ is a super line bundle $\cal L$ over $\cal Y$ equipped with functorial isomorphisms for all disjoint pairs of $S$-points $(x^I, x^J)$ of $\Ran$ with respect to \eqref{eq-factorization-isomorphism}:
\begin{equation}
\label{eq-line-bundle-factorization-isomorphism}
(\varphi_{x^I, x^J})^*(\cal L_{x^I}\boxtimes\cal L_{x^J}) \xrightarrow{\simeq} \cal L_{x^I\sqcup x^J},
\end{equation}
which are compatible with \eqref{eq-factorization-commutativity} and \eqref{eq-factorization-associativity}.

Let us spell out the compatibility with \eqref{eq-factorization-commutativity}. Denote by $\exch : \cal Y_{x^I} \times_S \cal Y_{x^J} \rightarrow \cal Y_{x^J} \times_S \cal Y_{x^I}$ the map which exchanges the coordinates. The commutativity constraint of the Picard groupoid of super line bundles yields an isomorphism:
\begin{equation}
\label{eq-super-line-bundle-braiding}
\exch^*(\cal L_{x^J} \boxtimes \cal L_{x^I}) \xrightarrow{\simeq} \cal L_{x^I} \boxtimes \cal L_{x^J}.
\end{equation}
The compatbility states that the image of \eqref{eq-super-line-bundle-braiding} under $(\varphi_{x^I, x^J})^*$, viewed as an isomorphism $(\varphi_{x^J, x^I})^*(\cal L_{x^J} \boxtimes \cal L_{x^I}) \xrightarrow{\simeq} (\varphi_{x^I, x^J})^*(\cal L_{x^I}\boxtimes\cal L_{x^J})$ by the commutativity of \eqref{eq-factorization-commutativity}, intertwines the isomorphisms \eqref{eq-line-bundle-factorization-isomorphism} attached to $(x^I, x^J)$, respectively $(x^J, x^I)$.
\end{void}

\begin{void}
Let $\cal H$ be a group factorization presheaf such that $\cal H_{x^I}$ satisfies fppf descent for each $S$-point $x^I$ of $\Ran$.

A multiplicative super line bundle $\cal L$ over $\cal H$ is called \emph{factorization} if it is equipped with a factorization structure which commutes with the multiplicative structure, \emph{i.e.}~\eqref{eq-line-bundle-factorization-isomorphism} is an isomorphism of multiplicative line bundles over $\cal H_{x^I\sqcup x^J} \cong \cal H_{x^I} \times_S \cal H_{x^J}$.

Note that a multiplicative factorization super line bundle over $\cal H$ is equivalent to a super central extension of group presheaves over $\Ran$:
\begin{equation}
\label{eq-super-central-extension}
1 \rightarrow \mathbb G_{m, \Ran} \rightarrow \widetilde{\cal H} \rightarrow \cal H \rightarrow 1,
\end{equation}
equipped with a functorial homomorphism $\widetilde{\varphi}_{x^I, x^J}$ lifting the factorization isomorphism $\varphi_{x^I, x^J}$ of $\cal H$ for each disjoint pair of $S$-points $(x^I, x^J)$ of $\Ran$:
\begin{equation}
\label{eq-central-extension-factorization}
\begin{tikzcd}[column sep = 1em]
	1 \ar[r] & \mathbb G_{m, S}\times_S \mathbb G_{m, S} \ar[r]\ar[d, "{(a, b)\mapsto ab}"] & \widetilde{\cal H}_{x^I} \times_S \widetilde{\cal H}_{x^J} \ar[r]\ar[d, "\widetilde{\varphi}_{x^I, x^J}"] & \cal H_{x^I} \times_S \cal H_{x^J} \ar[r]\ar[d, "\varphi_{x^I, x^J}"] & 1 \\
	1 \ar[r] & \mathbb G_{m, S} \ar[r] & \widetilde{\cal H}_{x^I\sqcup x^J} \ar[r] & \cal H_{x^I\sqcup x^J} \ar[r] & 1
\end{tikzcd}
\end{equation}
which satisfies commutativity and associativity. The data \eqref{eq-super-central-extension}, \eqref{eq-central-extension-factorization} subject to these conditions are called a \emph{factorization super central extension} of $\cal H$ by $\mathbb G_{m, \Ran}$. They form a Picard groupoid to be denoted by:
$$
\Hom_{\fact}(\cal H, \Pic^{\super}).
$$
(We interpret them as homomorphisms $\cal H \rightarrow \Pic^{\super}$ compatible with factorization.)

Let us again be explicit about commutativity: \eqref{eq-super-central-extension} being a super central extension, each $S$-point $(x^I, h^I)$ of $\widetilde{\cal H}$ carries a grading, viewed as a locally constant section of $\underline{\mathbb Z}/2$ over $S$. Commutativity refers to the equality:
$$
\widetilde{\varphi}_{x^I, x^J}(h^I, h^J) = (-1)^{\epsilon^I\epsilon^J}\widetilde{\varphi}_{x^J, x^I}(h^J, h^I),
$$
whenever $h^I$ (resp.~$h^J$) has grading $\epsilon^I$ (resp.~$\epsilon^J$).
\end{void}

\subsection{Loop groups}
\label{sec-loop-groups}

\begin{void}
Let $Y \rightarrow X$ be an affine morphism of finite type.

Denote by $\cal LY$ (resp.~$\cal L^+Y$) the presheaf whose $S$-points are pairs $(x^I, y^I)$ where $x^I$ is an $S$-point of $\Ran$ and $y^I$ is an $X$-morphism $\mathring D_{x^I} \rightarrow Y$ (resp.~$D_{x^I} \rightarrow Y$). Note that $\cal L^+Y$ is a closed subpresheaf of $\cal LY$ and the structural morphism $\cal LY \rightarrow \Ran$ (resp.~$\cal L^+Y\rightarrow \Ran$) is indschematic (resp.~schematic), see \cite[2.4-2.5]{MR2102701}.

Furthermore, $\cal LY$ admits a canonical factorization structure. Indeed, for any disjoint pair of $S$-points $(x^I, x^J)$ of $\Ran$, there is a functorial isomorphism:
$$
\cal L_{x^I\sqcup x^J}Y \xrightarrow{\simeq} \cal L_{x^I} Y \times_S \cal L_{x^J} Y,
$$
induced from $\mathring D_{x^I\sqcup x^J} \cong \mathring D_{x^I} \sqcup\mathring D_{x^J}$, which is clearly commutative and associative. Analogously, $\cal L^+Y$ also admits a canonical factorization structure.

Since the association $Y \mapsto \cal LY$ (resp.~$\cal L^+Y$) preserves limits, it carries an affine group $X$-scheme $G$ of finite type to a factorization group presheaf $\cal LG$ (resp.~$\cal L^+G$) over $\Ran$.
\end{void}

\begin{void}
Let $G$ be a smooth group $X$-scheme with connected geometric fibers. We also introduce the \emph{affine Grassmannian} $\Gr_G$ as the presheaf whose $S$-points are triples $(x^I, P, \alpha)$, where $x^I$ is an $S$-point of $X$, $P$ is a $G$-torsor over $S\times X$, and $\alpha$ is a trivialization of $P$ over $S\times X\backslash\Gamma_{x^I}$. Then $\Gr_G \rightarrow \Ran$ is ind-schematic of ind-finite type; it is ind-proper when $G$ is reductive \cite[Theorem 3.1.3]{MR3752460}.

The factorization structure on $\Gr_G$ is defined by Beauville--Laszlo gluing \cite[Theorem 3.2.1]{MR3752460} and the canonical map $\cal LG \rightarrow \Gr_G$ realizes the latter as the quotient $\cal LG/\cal L^+G$ in the \'etale topology \cite[Proposition 3.1.9]{MR3752460}.
\end{void}

\begin{void}
\label{void-arc-group-structure}
For later purposes, we give a convenient description of $\cal L^+G \rightarrow \Ran$ as an inverse limit of smooth affine group schemes relative to $\Ran$.

Consider an $S$-point $x^I$ of $\Ran$. The morphism $\Gamma_{x^I} \rightarrow S$ is finite locally free. Denote by $R_{\Gamma}G$ the Weil restriction along $\Gamma_{x^I} \rightarrow S$ of $G$ (pulled back along $\Gamma_{x^I}\subset S\times X \twoheadrightarrow X$.) Then $R_{\Gamma}G$ is representable by a smooth affine group $S$-scheme \cite[\S7.6]{MR1045822}. The evaluation map defines a short exact sequence:
\begin{equation}
\label{eq-weil-restriction-sequence}
1 \rightarrow \cal L_{x^I}^{\ge 1}G \rightarrow \cal L^+_{x^I}G \rightarrow R_{\Gamma}G \rightarrow 1.
\end{equation}

More generally, we let $\Gamma_{x^I}^{(n)}$ (for $n\ge 0$) denote the $n$th order infinitesimal neighborhood of the closed immersion $\Gamma_{x^I} \subset S\times X$. Then $\Gamma_{x^I}^{(n)} \rightarrow S$ is finite locally free: writing $\cal I$ for the ideal sheaf defining $\Gamma_{x^I}$, we see that each $\cal I^n/\cal I^{n+1}$ is locally isomorphic to $\cal O_{S\times X}/\cal I$ as an $\cal O_S$-module. Let $R_{\Gamma^{(n)}}G$ be the Weil restrction along $\Gamma_{x^I}^{(n)} \rightarrow S$, which is again representable by a smooth affine group $S$-scheme. This gives us a limit presentation:
$$
\cal L^+_{x^I}G \xrightarrow{\simeq} \lim_n R_{\Gamma^{(n)}}G.
$$

Under the Tannakian formalism, the formula $\xi \mapsto 1 + \xi$ defines an isormorphism between the vector group $S$-scheme $\fr g\otimes(\cal I^{n+1}/\cal I^{n+2})$ and the kernel of the evaluation map $R_{\Gamma^{(n+1)}}G \rightarrow R_{\Gamma^{(n)}}G$. In particular, the group scheme $\cal L_{x^I}^{\ge 1}G$ in \eqref{eq-weil-restriction-sequence} is an (infinite) iterated extension of vector group $S$-schemes.
\end{void}

\subsection{Contou-Carr\`ere}
\label{sec-contou-carrere}

\begin{void}
For each integer $n\ge 1$, we shall define \emph{Tate central extension} as a factorization super central extension:
\begin{equation}
\label{eq-tate-central-extension}
1 \rightarrow \mathbb G_{m, \Ran} \rightarrow \widetilde{\GL}_n \rightarrow \cal L\GL_n \rightarrow 1.
\end{equation}

Viewing $\widetilde{\GL}_n$ as a super line bundle over $\cal L\GL_n$, its fiber at an $S$-point $(x^I, a^I)$ of $\cal L\GL_n$ is the super $\cal O_S$-module:
$$
\det(a^I\cal O_{D_{x^I}}^{\oplus n} \mid \cal O_{D_{x^I}}^{\oplus n}) \text{ with grading }\rank(a^I\cal O_{D_{x^I}}^{\oplus n} \mid \cal O_{D_{x^I}}^{\oplus n})\text{ mod }2,
$$
where $\det(L_1 \mid L_2)$ denotes the relative determinant of two lattices $L_1$, $L_2$ in the Tate $\cal O_S$-module $\cal O_{\mathring D_{x^I}}^{\oplus n}$ and $\rank(L_1 \mid L_2)$ denotes their relative rank. (See \cite{MR2181808} or \cite[\S3]{campbell2021geometric} for the definition of these notions).

The multiplicative structure of \eqref{eq-tate-central-extension} is defined by the canonical isomorphism:
$$
\det(a^Ib^I\cal O_{D_{x^I}}^{\oplus n} \mid \cal O_{D_{x^I}}^{\oplus n}) \cong \det(a^I\cal O_{D_{x^I}}^{\oplus n} \mid \cal O_{D_{x^I}}^{\oplus n}) \otimes \det(b^I\cal O_{D_{x^I}}^{\oplus n} \mid \cal O_{D_{x^I}}^{\oplus n}),
$$
for any $S$-points $(x^I, a^I)$ and $(x^I, b^I)$ of $\cal L\GL_n$. The factorization isomorphism arises from the $\mathbb Z/2$-graded multiplicativity of determinants with respect to direct sums:
$$
\det(a^I\cal O_{D_{x^I}}^{\oplus n} \oplus b^J\cal O_{D_{x^J}}^{\oplus n} \mid \cal O_{D_{x^I}}^{\oplus n} \oplus \cal O_{D_{x^J}}^{\oplus n}) \cong \det(a^I\cal O_{D_{x^I}}^{\oplus n} \mid \cal O_{D_{x^I}}^{\oplus n}) \otimes \det(b^J\cal O_{D_{x^J}}^{\oplus n} \mid \cal O_{D_{x^J}}^{\oplus n}),
$$
for $S$-points $(x^I, a^I)$, $(x^J, b^J)$ of $\cal L\GL_n$ with $x^I$, $x^J$ disjoint.
\end{void}

\begin{void}
Following \cite[\S4]{campbell2021geometric}, we define the \emph{Contou-Carr\`ere symbol} (or \emph{tame symbol}) to be the commutator pairing of \eqref{eq-tate-central-extension} for $n = 1$:
\begin{equation}
\label{eq-contou-carrere-pairing}
\langle \cdot, \cdot \rangle : 
\cal L\mathbb G_m \otimes \cal L\mathbb G_m \rightarrow \mathbb G_{m, \Ran},
\end{equation}
Namely, $\langle\cdot,\cdot\rangle$ carries $S$-points $(x^I, a^I)$, $(x^I, b^I)$ of $\cal L\mathbb G_m$ to the element $(x^I, \tilde a^I\tilde b^I (\tilde a^I)^{-1}(\tilde b^I)^{-1})$ of $\mathbb G_{m,\Ran}$, where $\tilde a^I$ (resp.~$\tilde b^I$) is a lift of $a^I$ (resp.~$b^I$) to $\widetilde{\mathbb G}_m$ which exists locally on $S$.

The pairing \eqref{eq-contou-carrere-pairing} is \emph{factorization} in the following sense: given disjoint $S$-points $x^I$, $x^J$ of $\Ran$ and lifts $a^I$, $b^I$ (resp.~$a^J$, $b^J$) of $x^I$ (resp.~$x^J$) to $\cal L\mathbb G_m$, there holds:
$$
\langle a^I\sqcup a^J, b^I\sqcup b^J\rangle = \langle a^I, b^I\rangle\langle a^J, b^J\rangle.
$$

Furthermore, \eqref{eq-contou-carrere-pairing} is \emph{perfect} in the sense that its adjoint:
\begin{equation}
\label{eq-contou-carrere-adjoint}
\cal L\mathbb G_m \rightarrow \underline{\Hom}(\cal L\mathbb G_m, \mathbb G_{m, \Ran})
\end{equation}
is an isomorphism of factorization group presheaves \cite[Corollary 5.4.1.1]{campbell2021geometric}. This pairing exhibits $\cal L^+\mathbb G_m$ as the Cartier dual of $\Gr_{\mathbb G_m}$ \cite[Theorem 5.2.1]{campbell2021geometric}.
\end{void}

\begin{void}
\label{void-definition-tame}
More generally, let $T$ be an $X$-torus with dual $X$-torus $\check T$, \eqref{eq-contou-carrere-adjoint} induces an isomorphism between $\cal L\check T$ and $\underline{\Hom}(\cal LT, \mathbb G_{m, \Ran})$.

In particular, for a pair of $X$-tori $T_1$, $T_2$ with sheaves of cocharacters $\Lambda_1$, $\Lambda_2$, any bilinear form $b : \Lambda_1 \otimes \Lambda_2 \rightarrow \mathbb Z$ defines a factorization pairing:
\begin{equation}
\label{eq-contou-carrere-torus}
\langle\cdot, \cdot\rangle_b : \cal LT_1 \otimes \cal LT_2 \rightarrow \mathbb G_{m, \Ran},
\end{equation}
uniquely characterized by the property that its restriction along $\lambda_1, \lambda_2$, viewed as homomorphisms from $\cal L\mathbb G_m$ to $\cal LT_1$ (resp.~$\cal LT_2$), equals $b(\lambda_1, \lambda_2)\langle\cdot, \cdot\rangle$.

Pairings $\cal LT_1\otimes\cal LT_2\rightarrow \mathbb G_{m, \Ran}$ of the form $\langle\cdot,\cdot\rangle_b$ are called \emph{tame}. Given morphisms $T_1' \rightarrow T_1$, $T_2' \rightarrow T_2$, a tame pairing $\cal LT_1 \otimes \cal LT_2\rightarrow \mathbb G_{m, \Ran}$ induces a tame pairing $\cal LT_1' \otimes \cal LT_2' \rightarrow \mathbb G_{m, \Ran}$. The converse also holds for surjections of tori.
\end{void}

\begin{lem}
\label{lem-tame-pairing-descent}
Let $\langle\cdot,\cdot\rangle : \cal LT_1\otimes\cal LT_2 \rightarrow \mathbb G_{m, \Ran}$ be a factorization pairing. Given surjections of $X$-tori $T_1' \rightarrow T_1$, $T_2' \rightarrow T_2$, if the induced pairing $\cal LT'_1\otimes\cal LT'_2 \rightarrow \mathbb G_{m, \Ran}$ is tame, then so is $\langle\cdot, \cdot\rangle$.
\end{lem}

\begin{void}
\label{void-bilinear-pairing-induces-skew-pairing}
Before proving Lemma \ref{lem-tame-pairing-descent}, we shall make an observation.

Suppose that we are given $X$-tori $T_1$ and $T_2$. \emph{Claim}: all factorization morphisms of group presheaves over $\Ran$ below are trivial:
\begin{align}
\label{eq-arc-to-grassmannian}
	\cal L^+T_1 &\rightarrow \Gr_{T_2}; \\
\label{eq-grassmannian-to-arc}
	\Gr_{T_1} &\rightarrow \cal L^+T_2.
\end{align}

For \eqref{eq-arc-to-grassmannian}, this is because $\cal L^+T_1\rightarrow \Ran$ is pro-smooth with connected geometric fibers, whereas $\Gr_{T_2}\rightarrow \Ran$ has formal geometric fibers. For \eqref{eq-grassmannian-to-arc}, this is because $\Gr_{T_1}\rightarrow \Ran$ is ind-proper, whereas $\cal L^+T_2\rightarrow \Ran$ is pro-affine. The combination of these two facts shows that any factorization bilinear pairing $\cal LT_1\otimes\cal LT_2\rightarrow \mathbb G_m$ induces, and is uniquely determined by a pairing $\cal L^+T_1 \otimes \Gr_{T_2} \rightarrow \mathbb G_m$.
\end{void}

\begin{proof}[Proof of Lemma \ref{lem-tame-pairing-descent}]
Let $\check T_2$ (resp.~$\check T_2'$) denote the $X$-torus dual to $T_2$ (resp.~$T_2'$). By perfectness of the Contou-Carr\`ere symbol, $\langle\cdot,\cdot\rangle$ is equivalent to a factorization morphism of group presheaves over $\Ran$:
\begin{equation}
\label{eq-arbitrary-pairing-adjoint}
\cal LT_1 \rightarrow \cal L\check T_2.
\end{equation}

We need to prove that for any pair of cocharacters $\lambda_1$, $\lambda_2$ of $T_1$, $T_2$, the endomorphism $\varphi$ of $\cal L\mathbb G_m$ defined by the composition:
$$
\begin{tikzcd}[column sep = 1em]
\cal L\mathbb G_m \ar[r, "\lambda_1"] & \cal LT_1 \ar[r, "\eqref{eq-arbitrary-pairing-adjoint}"] & \cal L\check T_2 \ar[r, "\lambda_2"] & \cal L\mathbb G_m
\end{tikzcd}
$$
is given by $a\mapsto a^N$ for some integer $N$.

The hypothesis implies that this statement holds after composing with an endomorphism of $\cal L\mathbb G_m$ defined by $n$th power map $a\mapsto a^n$ for some integer $n \ge 1$.

By the observation of \S\ref{void-bilinear-pairing-induces-skew-pairing}, $\varphi$ is uniquely determined by its restriction $\varphi^+$ to $\cal L^+\mathbb G_m$, whose image is also contained in $\cal L^+\mathbb G_m$. Note furthermore that for each $S$-point $x^I$ of $X^I$, the restriction $\cal L^+_{x^I}\mathbb G_m$ of $\cal L^+\mathbb G_m$ is the extension of a group $S$-scheme of multiplicative type by an iterated extension of vector group $S$-schemes (\S\ref{void-arc-group-structure}).

In particular, $\varphi^+$ induces a homomorphism of short exact sequences:
$$
\begin{tikzcd}[column sep = 1em]
	0 \ar[r] & \cal L_{x^I}^{\ge 1}\mathbb G_m \ar[r]\ar[d, "\varphi^+_a"] & \cal L_{x^I}^+\mathbb G_m \ar[r]\ar[d, "\varphi^+"] & R_{\Gamma}\mathbb G_m \ar[r]\ar[d, "\varphi^+_m"] & 1 \\
	0 \ar[r] & \cal L_{x^I}^{\ge 1}\mathbb G_m \ar[r] & \cal L_{x^I}^+\mathbb G_m \ar[r] & R_{\Gamma}\mathbb G_m \ar[r] & 1
\end{tikzcd}
$$

The fact that $\varphi_m^+$ is the $N'$th power map after composing with the $n$th power map shows that $n\mid N'$ and $\varphi_m^+$ is the $N$th power map, for $N := N'/n$.

Since $\cal L^+_{x^I}\mathbb G_m \rightarrow S$ is pro-smooth, it suffices to prove that $\varphi^+$ is the $N$th power map on $\bar k$-points. In other words, given $f \in \bar k\arc{t}^{\times}$ satisfying the equality:
$$
\varphi^+(f)^n = (f^N)^n \text{ in } \bar k\arc{t}^{\times},
$$
we need to deduce the equality $\varphi^+(f) = f^N$. Setting $g := \varphi^+(f)/f^N$, we may write:
$$
g = 1 + \sum_{i\ge 1} a_i t^i \text{ in } \bar k\arc{t}^{\times}.
$$

\emph{Claim}: $g^n = 1$ implies $g = 1$. For $\tn{char}(\bar k) \nmid n$, this holds because all $n$th roots of unity of $\bar k\arc{t}$ are contained in $\bar k$. For $\tn{char}(\bar k) \mid n$, this holds because the Frobenius is injective.
\end{proof}

\begin{eg}
\label{eg-central-extension-wild}
Suppose that $k$ has characterstic $p > 0$. Let us define a factorization central extension whose commutator is \emph{not} tame:
\begin{equation}
\label{eq-central-extension-wild}
1 \rightarrow \mathbb G_{m, \Ran} \rightarrow \cal G \rightarrow \cal L\mathbb G_m \rightarrow 1.
\end{equation}

Given a morphism $Y \rightarrow S$ of $k$-presheaves, we write $\Frob_{Y/S} : Y \rightarrow Y^{(1)}_{/S}$ for the $p$th power Frobenius of $Y$ relative to $S$. Its formation is compatible with base change along $S$. Note that the presheaf $\cal L\mathbb G_m{}^{(1)}_{/\Ran}$ is canonically isomorphic to $\cal L\mathbb G_m$: an $S$-point of $\cal L\mathbb G_m{}^{(1)}_{/\Ran}$ is a pair $(x^I, a)$ where $x^I$ is an $S$-point of $\Ran$ and $a$ is map $\mathring D_{\Frob_S^*(x^I)}\rightarrow\mathbb G_m$. However, $\mathring D_{\Frob_S^*(s^I)}$ is isomorphic to $\mathring D_{x^I}$ since its formation depends only on the subset $|\Gamma_{x^I}|$ of $|S\times X|$. In particular, we may view $\Frob_{\cal L\mathbb G_m/\Ran}$ as an endomorphism of $\cal L\mathbb G_m$ over $\Ran$.

The central extension \eqref{eq-central-extension-wild} is defined to be the presheaf of sets $\mathbb G_{m, \Ran} \times_{\Ran} \cal L\mathbb G_m$ whose group structure is defined by the cocycle:
$$
\cal L\mathbb G_m \otimes \cal L\mathbb G_m \rightarrow \mathbb G_{m, \Ran},\quad (a, b)\mapsto \langle \Frob_{\cal L\mathbb G_m/\Ran}(a), b\rangle,
$$
where $\langle\cdot, \cdot\rangle$ denotes the Contou-Carr\`ere symbol. Since $\langle\cdot, \cdot\rangle$ is anti-symmetric, the commutator of \eqref{eq-central-extension-wild} is the pairing:
\begin{equation}
\label{eq-central-extension-wild-commutator}
\cal L\mathbb G_m\otimes\cal L\mathbb G_m \rightarrow \mathbb G_{m, \Ran},\quad (a, b)\mapsto\langle\Frob_{\cal L\mathbb G_m/\Ran}(a), b\rangle\langle a, \Frob_{\cal L\mathbb G_m/\Ran}(b)\rangle.
\end{equation}

Let us argue that this pairing is not tame over any geometric point $x : \Spec(\bar k) \rightarrow X$. The choice of a uniformizer allows us to identify $\cal L_x\mathbb G_m$ with $\mathbb G_m\loo{t}$. The morphism $\Frob_{\mathbb G_m\loo{t}/\bar k}$ evaluates to the following map on $R$-points for any $\bar k$-algebra $R$:
$$
R\loo{t}^{\times} \rightarrow R\loo{t}^{\times},\quad \sum_n a_nt^n \mapsto \sum_n (a_n)^pt^n.
$$
The commutator \eqref{eq-central-extension-wild-commutator} is indeed $\langle\cdot,\cdot\rangle^{p+1}$ on $\bar k$-points. For a more general $\bar k$-algebra $R$, the Contou-Carr\`ere pairing $\langle 1 - a_1t, 1 - b_{-1}t^{-1}\rangle$ equals $1 - a_1b_{-1}$ for nilpotents $a_1, b_{-1}\in R$ \cite{MR2036223}. Taking $R := \bar k[\epsilon]/\epsilon^3$ with $a_1 := \epsilon x$, $b_{-1} := \epsilon y$ for $x, y\in\bar k^{\times}$ and equating the commutator $(1 - a_1^pb_{-1})(1 - a_1b_{-1}^p)$ with $(1 - a_1b_{-1})^{p+1}$, we find $xy = 0$, which is impossible.
\end{eg}

\begin{void}
We now show that tameness is a positive characteristic phenomenon.

The assertion below relies on \cite{tao2021mathrmgrg} which uses the hypothesis $\tn{char}(k) = 0$.
\end{void}

\begin{prop}
\label{prop-characteristic-zero-implies-tame}
Let $T_1$, $T_2$ be a pair of $X$-tori. If $\tn{char}(k) = 0$, then any factorization pairing $\cal LT_1\otimes\cal LT_2\rightarrow \mathbb G_{m, \Ran}$ is tame.
\end{prop}
\begin{proof}
Using the observations in \S\ref{void-bilinear-pairing-induces-skew-pairing}, it suffices to prove that any factorization pairing $\cal L^+T_1 \otimes \Gr_{T_2} \rightarrow \mathbb G_{m, \Ran}$ is necessarily of the form $\langle\cdot, \cdot\rangle_b$ for some bilinear form $b : \Lambda_1 \otimes \Lambda_2 \rightarrow \mathbb Z$ (see \S\ref{void-definition-tame}).

Using the duality between $\Gr_{T_2}$ and $\cal L^+\check T_2$ under the Contou-Carr\`ere symbol, we reduce the statement to the special case $T_1 = T_2 = \mathbb G_m$.

For each $I$-tuple $\lambda^I = (\lambda^i)$ of integers, there is a closed immersion $\iota_{\lambda^I} : X^I \rightarrow \Gr_{\mathbb G_m, X^I}$ sending an $S$-point $x^I = (x^i)$ of $X^I$ to the line bundle $\cal O(\sum_{i\in I}\Gamma_{x^i})$ over $S\times X$ equipped with its canonical trivialization off $\Gamma_{x^I}$. Consider the category of pairs $(I, \lambda^I)$ where $I$ is a nonempty finite set and $\lambda^I$ is as above, where morphisms $(I, \lambda^I) \rightarrow (J, \lambda^J)$ are defined by surjections $\varphi : I\twoheadrightarrow J$ with $\lambda^j = \sum_{i \in \varphi^{-1}(j)}\lambda^i$ for each $j\in J$. The closed immersions $\iota_{\lambda^I}$ assemble into a morphism of presheaves over $\Ran$:
$$
\Gr_{\mathbb G_m}^{\tn{comb}} := \colim_{(I, \lambda^I)} X^I \rightarrow \Gr_{\mathbb G_m},
$$
which induces a bijection on field-valued points.

\emph{Claim}: factorization pairings $\cal L^+\mathbb G_m \otimes \Gr_{\mathbb G_m}^{\tn{comb}} \rightarrow \mathbb G_{m, \Ran}$ are in bijection with sections of $\underline{\mathbb Z}$ over $X$. More precisely, locally on $X$, a generator is the colimit over $(I, \lambda^I)$ of maps:
\begin{equation}
\label{eq-factorization-homomorphism-generator}
f_{(I, \lambda^I)} : \cal L^+_{X^I}\mathbb G_m \rightarrow \mathbb G_m,\quad (x^I, a) \mapsto \prod_{i\in I} (a|_{\Gamma_{x^i}})^{\lambda^i},
\end{equation}
where $a|_{\Gamma_{x^i}}$ is the restriction of $a$ to the closed subscheme $\Gamma_{x^i}\subset D_{x^I}$.

To prove the claim, we use the presentation \eqref{eq-weil-restriction-sequence} of the group $X^I$-scheme $\cal L^+_{X^I}\mathbb G_m$ as an extension of $R_{\Gamma}\mathbb G_m$ by $\cal L_{X^I}^{\ge 1}\mathbb G_m$. Every character $\cal L^+_{X^I}\mathbb G_m \rightarrow \mathbb G_m$ must factor through $R_{\Gamma}\mathbb G_m$ and is uniquely determined by its restriction to the pairwise disjoint locus of $X^I$.

Given a factorization pairing $f' : \cal L^+\mathbb G_m\otimes \Gr_{\mathbb G_m}^{\tn{comb}} \rightarrow \mathbb G_{m, \Ran}$, we obtain a system of maps indexed by $(I, \lambda^I)$:
\begin{equation}
\label{eq-factorization-homomorphism-arbitrary}
f'_{(I, \lambda^I)} : \cal L_{X^I}^+\mathbb G_m \rightarrow \mathbb G_m.
\end{equation}
By the observation above, \eqref{eq-factorization-homomorphism-arbitrary} is uniquely determined by the case $I = \{1\}$. Moreover, $f'_{(\{1\}, \lambda)}$ is a character of $R_{\Gamma}\mathbb G_m\cong \mathbb G_{m, X}$, hence a section of $\underline{\mathbb Z}$ over $X$. Looking at $I = \{1, 2\}$, we see that the association $\lambda\mapsto f'_{(\{1\}, \lambda)}$ defines a group homomorphism $\underline{\mathbb Z} \rightarrow \underline{\mathbb Z}$. The group homomorphism corresponding to multiplication by $n$ yields the $n$th power of \eqref{eq-factorization-homomorphism-generator}.

Under the hypothesis $\tn{char}(k) = 0$, \cite[Proposition 5.1.5]{tao2021mathrmgrg} shows that any $S$-point of $\Gr_{\mathbb G_m}$ admits a factorization $S \rightarrow S_0 \rightarrow \Gr_{\mathbb G_m}$, where $S_0$ is reduced. Therefore, any pairing $\cal L^+\mathbb G_m \otimes \Gr_{\mathbb G_m} \rightarrow \mathbb G_{m, \Ran}$ is uniquely determined by its values on reduced test schemes, hence on field-valued points. Consequently, any such pairing is uniquely determined by its restriction to $\cal L^+\mathbb G_m\otimes\Gr_{\mathbb G_m}^{\tn{comb}}$. 
\end{proof}

\begin{cor}
\label{cor-contou-carrere-universal}
If $\tn{char}(k) = 0$, then any factorization pairing $\cal L\mathbb G_m \otimes \cal L\mathbb G_m \rightarrow \mathbb G_{m, \Ran}$ is an integral power of the Contou-Carr\`ere symbol.
\end{cor}
\begin{proof}
This is a restatement of Proposition \ref{prop-characteristic-zero-implies-tame} in the special case $T_1 = T_2 = \mathbb G_m$.
\end{proof}

\begin{rem}
For any field $k$, we may view the proof of Proposition \ref{prop-characteristic-zero-implies-tame} as establishing the implication (1) $\Rightarrow$ (2) between the following statements:
\begin{enumerate}
	\item any $S$-point of $\Gr_{\mathbb G_m}$ admits a factorization $S\rightarrow S_0\rightarrow\Gr_{\mathbb G_m}$ where $S_0$ is reduced;
	\item any factorization pairing $\cal L\mathbb G_m\otimes\cal L\mathbb G_m \rightarrow \mathbb G_{m, \Ran}$ is an integral power of the Contou-Carr\`ere symbol.
\end{enumerate}

Since (2) fails when $\tn{char}(k) > 0$ (Example \ref{eg-central-extension-wild}), (1) must also fail when $\tn{char}(k) > 0$. In other words, the hypothesis $\tn{char}(k) = 0$ in \cite[Theorem 1.2.1]{tao2021mathrmgrg} is necessary.
\end{rem}

\begin{rem}
It is known that the Contou-Carr\`ere symbol is the universal Steinberg symbol over a point of $X$, \emph{cf.}~\cite{MR3438594}. The universal property established in Corollary \ref{cor-contou-carrere-universal} is of a different kind: instead of imposing the Steinberg relation, we impose compatibility with factorization.
\end{rem}

\subsection{Classification}
\label{sec-factorization-classification}

\begin{void}
Let $G$ be a reductive group $X$-scheme. Denote by $\Rad(G)$ the radical of $G$, \emph{i.e.}~the maximal torus of the center $Z_G$ of $G$ \cite[XXII, D\'efinition 4.3.6]{SGA3}.

Our principal goal is to study factorization super central extensions of $\cal LG$ by $\mathbb G_{m, \Ran}$ subject to the followng property: the commutator of the induced factorization super central extension of $\cal L\Rad(G)$ by $\mathbb G_{m, \Ran}$ is tame in the sense of \S\ref{void-definition-tame}. Such a factorization super central extension of $\cal LG$ by $\mathbb G_{m, \Ran}$ is said to have \emph{tame commutator}.

By Example \ref{eg-central-extension-wild} and Proposition \ref{prop-characteristic-zero-implies-tame}, the condition of having tame commutator is vacuous when $\tn{char}(k) = 0$, but not so when $\tn{char}(k) > 0$.
\end{void}

\begin{void}
\label{void-omega-twist}
Suppose that $G$ has a maximal torus $T$ with sheaf of cocharacters $\Lambda$.

Write $G_{\mathrm{sc}}$ for the simply connected form of $G$ with induced maximal torus $T_{\mathrm{sc}}$ and sheaf of cocharacters $\Lambda_{\mathrm{sc}}$.

Recall that any integral Weyl-invariant quadratic form $Q_{\mathrm{sc}}$ on $\Lambda_{\mathrm{sc}}$ defines a section of $\vartheta(\Lambda_{\mathrm{sc}})$ via \eqref{eq-canonical-theta-data-brylinski-deligne}. The corresponding central extension $\widetilde{\Lambda}_{\mathrm{sc}}$ of $\Lambda_{\mathrm{sc}}$ by $\mathbb G_m$ can be viewed as a monoidal morphism:
\begin{equation}
\label{eq-simply-connected-canonical-extension-as-morphism}
\nu_{Q_{\mathrm{sc}}} : \Lambda_{\mathrm{sc}} \rightarrow \deloop \mathbb G_m \cong \Pic.
\end{equation}

We define a morphism of pointed $X$-stacks by the formula:
\begin{equation}
\label{eq-simply-connected-canonical-map-omega-shift}
\nu_{Q_{\mathrm{sc}}, +} : \Lambda_{\mathrm{sc}} \rightarrow \Pic,\quad \lambda \mapsto \nu_{Q_{\mathrm{sc}}}(\lambda) \otimes \omega_X^{Q_{\mathrm{sc}}(\lambda)}.
\end{equation}
\end{void}

\begin{void}
To specify the additional structure of the map $\nu_{Q_{\mathrm{sc}}, +}$ inherited from the monoidal structure of $\nu_{Q_{\mathrm{sc}}}$, we introduce a piece of terminology.

Let $\Lambda$ denote, temporarily, any \'etale sheaf of finite free $\mathbb Z$-modules over $X$ and $b : \Lambda\otimes\Lambda \rightarrow \mathbb Z$ be any symmetric bilinear form. A morphism of $X$-stacks $\nu_+ : \Lambda \rightarrow \Pic^{\super}$ is said to be \emph{$\omega$-monoidal with respect to $b$} if it is equipped with ismorphisms:
\begin{align*}
\cal O_X &\xrightarrow{\simeq} \nu_+(0); \\
\nu_+(\lambda_1) \otimes \nu_+(\lambda_2) \otimes \omega_X^{b(\lambda_1, \lambda_2)} &\xrightarrow{\simeq} \nu_+(\lambda_1 + \lambda_2),
\end{align*}
for each $\lambda_1, \lambda_2\in\Lambda$, satisfying unitality and associativity. (This notion does not refer to the commutativity constraint of the Picard groupoid $\Pic^{\super}$.)

In this terminology, the map $\nu_{Q_{\mathrm{sc}}, +}$ is $\omega$-monoidal with respect to the symmetric form $b_{\mathrm{sc}}$ associated to $Q_{\mathrm{sc}}$.
\end{void}

\begin{void}
\label{void-super-theta-data-shifted}
Denote by $\vartheta^{\super}_{G, +}(\Lambda)$ the Picard groupoid of triples $(b, \nu_+, \varphi)$, where:
\begin{enumerate}
	\item $b$ is a Weyl-invariant integral symmetric bilinear form on $\Lambda$, such that $b(\lambda, \lambda) \in 2\mathbb Z$ if $\lambda \in \Lambda_{\mathrm{sc}}$---we write $Q_{\mathrm{sc}}$ for the corresponding quadratic form on $\Lambda_{\mathrm{sc}}$;
	\item $\nu_+$ is a morphism $\Lambda \rightarrow \Pic^{\super}$ which is $\omega$-monoidal with respect to $b$ and commutes with the commutativity constraint up to the bilinear form $(-1)^{b}$;
	\item $\varphi$ is an isomorphism between the restriction of $\nu_+$ to $\Lambda_{\mathrm{sc}}$ and $\nu_{Q_{\mathrm{sc}}, +}$ as $\omega$-monoidal morphisms.
\end{enumerate}

The relation between $\vartheta_{G, +}^{\super}(\Lambda)$ and the Picard groupoid $\vartheta_G^{\super}(\Lambda)$ defined in \S\ref{void-theta-data-super} is as follows. Given a \emph{$\vartheta$-characteristic} $\omega^{1/2}$, \emph{i.e.}~a line bundle over $X$ equipped with an isomorphism $(\omega^{1/2})^{\otimes 2}\cong \omega_{X}$, we obtain an isomorphism (called the \emph{$\omega^{1/2}$-shift}):
\begin{equation}
\label{eq-super-theta-data-omega-shift}
\vartheta_G^{\super}(\Lambda) \xrightarrow{\simeq} \vartheta_{G, +}^{\super}(\Lambda),\quad (b, \nu, \varphi) \mapsto (b, \nu_+, \varphi),
\end{equation}
where $\nu_+$ is the $\omega$-monoidal morphism $\lambda \mapsto \nu(\lambda)\otimes (\omega^{1/2})^{b(\lambda, \lambda)}$ and $\nu$ is the monoidal morphism $\Lambda \rightarrow \Pic^{\super}$ corresponding to $\widetilde{\Lambda}$, \emph{i.e.}~\eqref{eq-central-extension-as-monoidal-morphism-super}.
\end{void}

\begin{thm}
\label{thm-factorization-super-central-extension-classification}
Let $G$ be a reductive group $X$-scheme. The following Picard groupoids are canonically equivalent:
\begin{enumerate}
	\item factorization super central extensions of $\cal LG$ by $\mathbb G_{m, \Ran}$ with tame commutator;
	\item factorization super line bundles over $\Gr_G$;
	\item $\vartheta_{G, +}^{\super}(\Lambda)$---if $G$ is equipped with a maximal torus $T$ with sheaf of cocharacters $\Lambda$.
	\item rigidified sections of $\Ktheory{}_{[1, 2]}^{\super}$ over the Zariski classifying stack of $G$---if $X$ is equipped with a $\vartheta$-characteristic.
\end{enumerate}
\end{thm}

\begin{void}
The equivalences of Theorem \ref{thm-factorization-super-central-extension-classification} are constructed in several stages. We briefly indicate the steps and explain where prior works are used.

First, we prove that factorization super central extensions of $\cal L^+G$ by $\mathbb G_{m, \Ran}$ are canonically trivial (Proposition \ref{prop-arc-group-triviality}). By descent, we obtain the functor (1) $\rightarrow$ (2).

Next, the equivalence (2) $\cong$ (3) is essentially known when $G$ is a torus or a semisimple, simply connected group scheme. The torus case is treated in \cite{MR4322626} (using a substantial theorem of \cite{MR4198527}). The simply connected case reduces to Falting \cite{MR1961134}.

One may then define the functor (2) $\rightarrow$ (3) for any reductive group scheme $G$, by appealing to functoriality with respect to the commutative diagram:
$$
\begin{tikzcd}
	T_{\mathrm{sc}} \ar[r]\ar[d] & G_{\mathrm{sc}} \ar[d] \\
	T \ar[r] & G
\end{tikzcd}
$$

We then proceed as follows. We first prove ``by hand'' that (1) $\rightarrow$ (2) is an equivalence for $G$ a torus or a semisimple, simply connected group scheme. Then we prove that, for any reductive group scheme $G$, the composition (1) $\rightarrow$ (2) $\rightarrow$ (3) is an equivalence, whereas the second functor is fully faithful. Here, we use an idea of Finkelberg--Lysenko \cite{MR2684259}, an idea of Gaitsgory \cite{MR4117995}, and an argument from \cite{MR4322626}.

The equivalences $(1) \cong (2) \cong (3)$ are completed in Proposition \ref{prop-super-central-extension-classification}.

Finally, the equivalence (3) $\cong$ (4) is defined by combining Theorem \ref{thm-classification-super} with the $\omega^{1/2}$-shift \eqref{eq-super-theta-data-omega-shift}. The composed functor (1) $\rightarrow $(2) $\rightarrow$ (3) $\rightarrow$ (4) thus \emph{a priori} depends on a maximal torus $T$, but it shall follow from the construction that this is not the case. By \'etale descent, we obtain the equivalence (1) $\cong$ (4) without assuming the existence of a maximal torus.
\end{void}

\begin{cor}
\label{cor-factorization-central-extension-classification}
Let $G$ be a reductive group $X$-scheme. The following Picard groupoids are canonically equivalent:
\begin{enumerate}
	\item factorization central extensions of $\cal LG$ by $\mathbb G_{m, \Ran}$ with tame commutator;
	\item factorization line bundles over $\Gr_G$;
	\item $\vartheta_{G, +}(\Lambda)$, if $G$ is equipped with a maximal torus $T$ with sheaf of cocharacters $\Lambda$;
	\item central extensions of $G$ by $\Ktheory{}_2$ on the big Zariski site of $X$.
\end{enumerate}
\end{cor}
\begin{proof}
Each Picard groupoid in Theorem \ref{thm-factorization-super-central-extension-classification} admits a canonical functor to $\Hom(\pi_1G, \mathbb Z/2)$: for the Picard groupoids (1), (2), these are the functors remembering the grading; for (3), this is the second functor in \eqref{eq-classification-super-fiber-sequence}; for (4), this is the lower horizontal functor in \eqref{eq-mod-2-compatibility-with-grading}.

The equivalences of Theorem \ref{thm-factorization-super-central-extension-classification} commute with these functors to $\Hom(\pi_1G, \mathbb Z/2)$. By restricting them to the fibers, we obtain the corollary. (Note that the restriction of \eqref{eq-super-theta-data-omega-shift} is defined without the choice of $\omega^{1/2}$.)
\end{proof}

\section{Proofs}
\label{sect-factorization-proofs}

The goal of this section is to prove Theorem \ref{thm-factorization-super-central-extension-classification}. We begin in \S\ref{sec-triviality-arc} by constructing a functor from factorization super central extensions of $\cal LG$ to factorization super line bundles over $\Gr_G$. The subsections \S\ref{sec-classification-tori}, \S\ref{sec-classification-simply-connected}, \S\ref{sec-classification-reductive} prove the equivalences (1) $\cong$ (2) $\cong$ (3) of Theorem \ref{thm-factorization-super-central-extension-classification} for tori, simply connected groups, respectively all reductive groups. In \S\ref{sec-poor-transgression}, we show that the equivalence (1) $\cong$ (4) obtained by combining the equivalence (1) $\cong$ (3) and Theorem \ref{thm-classification-super} is independent of the choice of a maximal torus.

We remain in the context of the previous section: we fix a ground field $k$ and a smooth curve $X$.

\subsection{Triviality over arc groups}
\label{sec-triviality-arc}

\begin{void}
In this subsection, we let $G$ denote a smooth affine group $X$-scheme with connected geometric fibers. Our goal is to prove the following result.
\end{void}

\begin{prop}
\label{prop-arc-group-triviality}
Any factorization super central extension of $\cal L^+G$ by $\mathbb G_{m, \Ran}$ is canonically trivial.
\end{prop}

\begin{void}
Consider an $S$-point $x^I$ of $\Ran$. Recall that the restriction $\cal L_{x^I}^+G$ is an extension of $R_{\Gamma}G$ by $\cal L_{x^I}^{\ge 1}G$ as a group $S$-scheme \eqref{eq-weil-restriction-sequence}.

We begin with a Lemma which allows us to reduce the problem to $R_{\Gamma}G$.
\end{void}

\begin{lem}
\label{lem-normality-descent-to-weil-restriction}
If $S$ is locally Noetherian and normal, then pullback along \eqref{eq-weil-restriction-sequence} defines an equivalence between the groupoid of central extensions of $R_{\Gamma}G$ by $\mathbb G_{m, S}$ and that of $\cal L_{x^I}^+G$.
\end{lem}
\begin{proof}
Given a central extension:
\begin{equation}
\label{eq-central-extension-arc-group-schematic}
1 \rightarrow \mathbb G_{m, S} \rightarrow \cal G \rightarrow \cal L_{x^I}^+G \rightarrow 1,
\end{equation}
we claim that \eqref{eq-central-extension-arc-group-schematic} admits a unique splitting over $\cal L_{x^I}^{\ge 1}G$, and its image is normal in $\cal G$. Then the association $\cal G \mapsto \cal G/\cal L_{x^I}^{\ge 1}G$ supplies the desired inverse functor.

Since $\cal L_{x^I}^{\ge 1}G$ is an iterated extension of vector group $S$-schemes, to show the existence and uniqueness of the splitting, it suffices to show that any central extension of $\mathbb G_{a, S}$ by $\mathbb G_{m, S}$ is canonically split.

This assertion follows from the observations below:
\begin{enumerate}
	\item any $S$-morphism $\mathbb G_{a, S} \rightarrow \mathbb G_{m, S}$ is trivial---this uses the hypothesis that $S$ is reduced;
	\item the pullback $\Pic(S) \rightarrow \Pic(\mathbb G_{a, S})$ is an equivalence---this uses the hypothesis that $S$ is locally Noetherian and normal \cite[Corollaire 21.4.13, p.361]{MR238860}.
\end{enumerate}

To prove that the resulting splitting $\cal L_{x^I}^{\ge 1}G \rightarrow \cal G$ has normal image, it suffices to show that it commutes with the conjugation action of $\cal L_{x^I}^+G$. However, this follows from the uniqueness of the splitting.
\end{proof}

\begin{proof}[Proof of Proposition \ref{prop-arc-group-triviality}]
For an integer $n\ge 1$, we denote by $\Ran^{\le n} \subset \Ran$ the subfunctor whose $S$-points are finite subsets of $\Maps(S, X)$ of cardinality $\le 2$.

Given a factorization super central extension:
\begin{equation}
\label{eq-factorization-super-extension-arc-group}
1 \rightarrow \mathbb G_{m, \Ran} \rightarrow \cal G \rightarrow \cal L^+G \rightarrow 1,
\end{equation}
we first observe that ``super'' is redundant since the group scheme $\cal L_{x^I}^+G \rightarrow S$ for any $S$-point $x^I$ of $\Ran$ has connected geometric fibers. Next, we claim that a trivialization of \eqref{eq-factorization-super-extension-arc-group} over $\Ran^{\le 2}$ compatible with the factorization morphism \eqref{eq-central-extension-factorization} for $|I| = |J| = 1$ uniquely extends to a trivialization of \eqref{eq-factorization-super-extension-arc-group} over $\Ran$.

For a nonempty finite set $I$, consider the tautological $X^I$-point $x^I$ of $\Ran$ (Remark \ref{rem-ran-space-colimit-presentation}). Denote by $U\subset X^I$ the open subset where $\Gamma_{x^{i_1}} \cap \Gamma_{x^{i_2}} \neq \emptyset$ for at most one pair of distinct elements $i_1, i_2 \in I$. Let $Z \subset X^I$ be its complement, a closed subset of codimension $\ge 2$. (It is empty if $|I|\le 2$.) The induced $U$-point $x^I|_U$ is a disjoint union of $x^i$ (for $i\neq i_1, i_2$) with $x^{\{i_1, i_2\}}$. The factorization morphism for $\cal G$ and its trivialization over $\Ran^{\le 2}$ define a trivialization of $\cal G_{x^I}$ over $U\subset X^I$.

Since $X^I$ is smooth, $\cal G_{x^I}$ is pulled back from a central extension of $R_{\Gamma}G$ by $\mathbb G_{m, X^I}$ (Lemma \ref{lem-normality-descent-to-weil-restriction}). Since $R_{\Gamma}G$ is also smooth, the trivialization of $\cal G_{x^I}$ extends uniquely along $U\subset X^I$ \cite[031T]{stacks-project}. Given a surjection of nonempty finite sets $I \twoheadrightarrow J$, we need to argue that the trivialization of $\cal G_{x^I}$ restricts to the trivialization of $\cal G_{x^J}$. This statement reduces to the case $|I| = |J| + 1$, where the diagonal $X^J \subset X^I$ intersects nontrivially with $U$. Since the two trivializations agree over $X^J\cap U$, they must agree over $X^J$.

Finally, we construct the trivialization of \eqref{eq-factorization-super-extension-arc-group} over $\Ran^{\le 2}$ compatible with factorization. Consider the tautological $X^{\{1, 2\}}$-point $x^{\{1, 2\}}$ of $\Ran$, with $U\subset X^{\{1, 2\}}$ the complement of the diagonal. The closed immersions $\Gamma_{x^i}\subset \Gamma_{x^{\{1, 2\}}}$ (for $i = 1, 2$) induce projection maps from $R_{\Gamma}G$ to $G$. Moreoever, the base change of $R_{\Gamma}G$ along the diagonal of $X^{\{1, 2\}}$ is an extension $\widetilde G$ of $G$ by a vector group scheme. These morphisms are summarized in the following diagram:
$$
\begin{tikzcd}
	\widetilde G \ar[r, "\Delta"]\ar[d] & R_{\Gamma}G \ar[r, shift left = 0.5ex, "\pr_1"]\ar[r, shift right = 0.5ex, swap, "\pr_2"]\ar[d] & G \ar[d] \\
	X \ar[r, "\Delta"] & X^{\{1, 2\}} \ar[r, shift left = 0.5ex, "\pr_1"]\ar[r, shift right = 0.5ex, swap, "\pr_2"] & X
\end{tikzcd}
$$
where both compositions in the upper row are the canonical surjection $\widetilde G \rightarrow G$.

The restriction of $\cal G$ along the tautological $X$-point of $\Ran$ is pulled back from a central extension of $G$ by $\mathbb G_{m, X}$, to be denoted by $\cal G_1$. Similarly, the restriction $\cal G_{x^{\{1, 2\}}}$ is pulled back from a central extension of $R_{\Gamma}G$, to be denoted by $\cal G_2$. By factorization, the restriction of $\cal G_2$ to $U\subset X^{\{1, 2\}}$ is identified with $\pr_1^*\cal G_1 \otimes \pr_2^*\cal G_1$. This identification extends uniquely to an isomorphism between $\cal G_2$ with $\pr_1^*\cal G_1 \otimes \pr_2^*\cal G_1 \otimes \cal O(n\Delta)$ as line bundles over $R_{\Gamma}G$, for some $n \in \mathbb Z$. Restriction to the unit section $e : X^{\{1, 2\}} \rightarrow R_{\Gamma}G$ tells us that $n = 0$, so we obtain an isomorphism of central extensions of $R_{\Gamma}G$ by $\mathbb G_{m, X^{\{1, 2\}}}$:
\begin{equation}
\label{eq-arc-group-two-copies-of-curve}
\cal G_2 \xrightarrow{\simeq} \pr_1^*\cal G_1 \otimes \pr_2^*\cal G_1.
\end{equation}
Restriction of \eqref{eq-arc-group-two-copies-of-curve} along the diagonal then yields an isomorphism $\cal G_1 \xrightarrow{\simeq} \cal G_1^{\otimes 2}$, \emph{i.e.}~a trivialization of $\cal G_1$. The trivialization of $\cal G_2$ is deduced from \eqref{eq-arc-group-two-copies-of-curve}, so it is automatically compatible with factorization.
\end{proof}

\begin{void}
\label{void-descent-from-loop-group}
Using Proposition \ref{prop-arc-group-triviality}, we construct a functor of Picard groupoids:
\begin{equation}
\label{eq-loop-group-descent-to-grassmannian}
	\Hom_{\fact}(\cal LG, \Pic^{\super}) \rightarrow \Gamma_{\fact}(\Gr_G, \Pic^{\super}),
\end{equation}
where the target consists of factorization super line bundles over $\Gr_G$.

Indeed, any factorization monoidal morphism $\cal LG \rightarrow \Pic^{\super}$ is trivial over $\cal L^+G$, so the monoidal structure yields its descent data to $\cal LG/\cal L^+G\cong \Gr_G$.

One can make a stronger statement: this monoidal structure yields descent data to the local Hecke stack $\Hec_G := \cal L^+G\backslash \cal LG/\cal L^+G$, and the resulting factorization super line bundle over $\Hec_G$ is compatible with the ``convolution structure'' of $\Hec_G$. To make this precise, we need the prestack:
$$
\Hec_G^{[2]} := \cal L^+G\backslash \cal LG\times^{\cal L^+G}\cal LG/\cal L^+G,
$$
equipped with three maps $p_1$, $m$, $p_2$ to $\Hec_G$: projection onto the first factor, multiplication, and projection onto the second factor. There is also a unit section $e : B(\cal L^+G) \rightarrow \Hec_G$, where $B$ stands for delooping relative to $\Ran$. A factorization super line bundle $\cal L$ over $\Hec_G$ is \emph{compatible with convolution} if there are additional isomorphisms:
\begin{align}
\label{eq-compatibility-with-convolution-product}
	\cal O_{B(\cal L^+G)} &\xrightarrow{\simeq} e^*\cal L; \\
\label{eq-compatibility-with-convolution-unit}
	(p_1)^*\cal L \otimes (p_2)^*\cal L &\xrightarrow{\simeq} m^*\cal L,
\end{align}
which satisfy the conditions of an associative algebra and commute with factorization.

Let $\Hom_{\fact}(\Hec_G, \Pic^{\super})$ denote the Picard groupoid of factorization super line bundles over $\Hec_G$ compatible with convolution. The descent procedure then yields an equivalence of Picard groupoids:
\begin{equation}
\label{eq-loop-group-descent-to-hecke}
	\Hom_{\fact}(\cal LG, \Pic^{\super}) \xrightarrow{\simeq} \Hom_{\fact}(\Hec_G, \Pic^{\super}).
\end{equation}
This assertion follows at once from two observations: the convolution structure on $\Hec_G$ is defined by the \v{C}ech nerve of $B(\cal L^+G) \rightarrow B(\cal LG)$ and any monoidal morphism $B(\cal L^+G) \rightarrow \Pic^{\super}$ compatible with factorization is canonically trivial (Proposition \ref{prop-arc-group-triviality}).
\end{void}

\subsection{Tori}
\label{sec-classification-tori}

\begin{void}
Let $\Lambda$ be an \'etale sheaf of finite free $\mathbb Z$-modules over $X$. Denote by $\vartheta^{\super}_+(\Lambda)$ the Picard groupoid of pairs $(b, F_+)$, where:
\begin{enumerate}
	\item $b$ is an integral symmetric bilinear form on $\Lambda$;
	\item $F_+ : \Lambda \rightarrow \Pic^{\super}$ is an $\omega$-monoidal morphism with respect to $b$ and commutes with the commutativity constraint up to the bilinear form $(-1)^b$.
\end{enumerate}

This is the special case of $\vartheta^{\super}_{G, +}(\Lambda)$ defined in \S\ref{void-super-theta-data-shifted}, for $G$ the $X$-torus $\Lambda\otimes\mathbb G_m$.
\end{void}

\begin{void}
\label{void-torus-grassmannian-subschemes}
Let $T$ be an $X$-torus with sheaf of cocharacters $\Lambda$. We shall define a functor from the Picard groupoid of factorization super line bundles over $\Gr_T$ to $\vartheta_+^{\super}(\Lambda)$:
\begin{equation}
\label{eq-torus-grassmannian-classification}
	\Gamma_{\fact}(\Gr_T, \Pic^{\super}) \rightarrow \vartheta_+^{\super}(\Lambda).
\end{equation}
This is a variant of \cite[\S3.10.7]{MR2058353}, where objects of $\vartheta_+^{\super}(\Lambda)$ are called \emph{$\vartheta$-data}.

For each $I$-tuple $\lambda^I = (\lambda_i)$ of elements of $\Lambda$, there is a closed immersion $\iota_{\lambda^I} : X^I \rightarrow \Gr_{T, X^I}$ sending an $S$-point $x^I = (x^i)$ of $X^I$ to the $T$-torsor $\cal O(\lambda_i \Gamma_{x^i})$ equipped with the canonical trivialization off $\Gamma_{x^I}$.

To define \eqref{eq-torus-grassmannian-classification}, we take a factorization super line bundle $\cal L$ over $\Gr_T$ and construct a pair $(b, F_+)$. Given $\lambda_1, \lambda_2 \in \Lambda$, the line bundle $(\iota_{\lambda_1, \lambda_2})^*\cal L$ is identified with $(\iota_{\lambda_1})^*\cal L \boxtimes (\iota_{\lambda_2})^*\cal L$ off the diagonal of $X^2$ by factorization---this identification extends to an isomorphism:
\begin{equation}
\label{eq-factorization-isomorphism-induced-map}
(\iota_{\lambda_1, \lambda_2})^*\cal L \cong (\iota_{\lambda_1})^*\cal L \otimes (\iota_{\lambda_2})^*\cal L \otimes \cal O_{X^2}(b(\lambda_1, \lambda_2)\Delta),
\end{equation}
for a uniquely defined integer $b(\lambda_1, \lambda_2)$. The associativity and unitality of the factorization isomorphism implies that $\lambda_1, \lambda_2\mapsto b(\lambda_1, \lambda_2)$ is a bilinear form. Commutativity of the factorization isomorphism implies that $b$ is symmetric and that $b(\lambda, \lambda)\tn{ mod }2$ agrees with the grading on $(\iota_{\lambda})^*\cal L$. The definition of $b$ is complete.

The second datum $F_+$ is set to be $F_+(\lambda) := (\iota_{\lambda})^*\cal L$, with $\omega$-monoidal structure given by restricting \eqref{eq-factorization-isomorphism-induced-map} along the diagonal. The fact that $F_+$ commutes with the brading up to the factor $(-1)^{b(\lambda_1, \lambda_2)}$ follows from the fact that the isomorphism $\cal O_{X^2}(\Delta)|_{\Delta} \cong \omega_X$ is equivariant against the exchange map $X^2\rightarrow X^2$, $(x^1, x^2)\mapsto (x^2, x^1)$ up to the factor $(-1)$.
\end{void}

\begin{prop}
\label{prop-tori-grassmannian-classification}
The functor \eqref{eq-torus-grassmannian-classification} is an equivalence of Picard groupoids.
\end{prop}
\begin{proof}
The proof is identical to that of \cite[Proposition 1.4]{MR4322626}.
\end{proof}

\begin{void}
\label{void-pairing-induced-factorization-line-bundle}
Let $T$, $T'$ be $X$-tori whose sheaves of cocharacters are $\Lambda$, respectively $\Lambda'$. Suppose that we are given a factorization bilinear pairing:
\begin{equation}
\label{eq-arbitrary-pairing-arc-against-grassmannian}
\langle \cdot, \cdot\rangle : \cal L^+T \otimes \Gr_{T'} \rightarrow \mathbb G_m.
\end{equation}

Delooping in the first variable and pulling back along the projection map $\Gr_{T} \rightarrow B(\cal L^+T)$, we obtain a morphism compatible with factorization:
\begin{equation}
\label{eq-contou-carrere-torus-line}
\langle\cdot, \cdot\rangle : \Gr_{T} \times_{\Ran} \Gr_{T'} \rightarrow B\mathbb G_m,
\end{equation}
or equivalently a factorization line bundle over $\Gr_{T\times T'}$.

Consider the factorization pairing $\langle\cdot,\cdot\rangle_b$ defined by a bilinear form $b : \Lambda\otimes\Lambda' \rightarrow \mathbb Z$ as in \S\ref{void-definition-tame}. The property of Contou-Carr\`ere symbol shows that $\langle\cdot,\cdot\rangle_b$ induces a pairing $\cal L^+T\otimes \Gr_{T'} \rightarrow \mathbb G_m$, hence a factorization line bundle $\cal O(b)$ over $\Gr_{T\times T'}$. We may calculate its image under the equivalence of Proposition \ref{prop-tori-grassmannian-classification}.
\end{void}

\begin{lem}
\label{lem-contou-carrere-line-calculation}
The factorization line bundle $\cal O(b)$ corresponds under the equivalence of Proposition \ref{prop-tori-grassmannian-classification} to the pair consisting of:
\begin{enumerate}
	\item the quadratic form $\Lambda\oplus\Lambda' \rightarrow \mathbb Z$, $(\lambda, \lambda') \mapsto b(\lambda, \lambda')$;
	\item the $\omega$-monoidal morphism $\Lambda \oplus\Lambda' \rightarrow \Pic$ with $(\lambda, \lambda') \mapsto \omega_X^{b(\lambda, \lambda')}$, whose $\omega$-monoidal structure is the isomorphism:
	$$
	(-1)^{b(\lambda_2, \lambda_1')}\cdot\id : \omega_X^{b(\lambda_1 + \lambda_2, \lambda_1' + \lambda'_2)} \xrightarrow{\simeq} \omega_X^{b(\lambda_1, \lambda_1')} \otimes \omega_X^{b(\lambda_2, \lambda'_2)} \otimes \omega_X^{b(\lambda_1, \lambda'_2) + b(\lambda_2, \lambda_1')},
	$$
	for any pair of elements $(\lambda_1, \lambda_1'), (\lambda_2, \lambda'_2) \in \Lambda \oplus\Lambda'$.
\end{enumerate}
\end{lem}
\begin{proof}
Let $(\lambda')^I = (\lambda_i')$ be an $I$-tuple of elements of $\Lambda'$. The restriction of $\cal O(b)$ along:
$$
(\id, \iota_{(\lambda')^I}) : \Gr_{T, X^I} \rightarrow \Gr_{T, X^I} \times_{X^I} \Gr_{T', X^I}
$$
is the line bundle over $\Gr_{T, X^I}$ whose fiber at an $S$-point $(x^I, P_{T}, \alpha)$ of $\Gr_{T, X^I}$ is given by $\bigotimes_{i\in I} (P_{T} |_{\Gamma_{x^i}})^{b(-, \lambda'_i)}$, where the superscript indicates inducing along the character $T\rightarrow \mathbb G_m$ defined by $b(-, \lambda_i') : \Lambda \rightarrow \mathbb Z$.

In particular, for an $I$-tuple $\lambda^I = (\lambda_i)$ of elements of $\Lambda$, further restricting $(\id, \iota_{(\lambda')^I})^*\cal O(b)$ along $\iota_{\lambda^I}$ yields the following line bundle over $X^I$:
\begin{equation}
\label{eq-contou-carrere-line-bundle}
\bigotimes_{i\in I} (p_i^*\omega_X^{b(\lambda_i, \lambda_i')} \otimes \bigotimes_{\substack{j\in I \\ j\neq i}} p_{ij}^*\cal O_{X^2}(b(\lambda_j, \lambda_i')\Delta)),
\end{equation}
where $p_i : X^I \rightarrow X$ (resp.~$p_{ij} : X^I \rightarrow X^2$) denotes the projection onto the factor labeled by $i$ (resp.~factors labeled by $(i, j)$).

Statement (1) follows by inspecting \eqref{eq-contou-carrere-line-bundle} for $I = \{1, 2\}$, seeing that the quadratic form $(\lambda, \lambda')\mapsto b(\lambda, \lambda')$ has symmetric form $(\lambda_1, \lambda_1'), (\lambda_2, \lambda_2') \mapsto b(\lambda_2, \lambda'_1) + b(\lambda_1, \lambda'_2)$. The first part of statement (2) follows by inspecting \eqref{eq-contou-carrere-line-bundle} for $I = \{1\}$, the second part for $I = \{1, 2\}$, taking into account the fact that the isomorphism $\cal O_{X^2}(\Delta)|_{\Delta} \cong \omega_X$ is equivariant for the exchange map $X^2 \rightarrow X^2$, $(x^1, x^2) \mapsto (x^2, x^1)$ up to the factor $(-1)$.
\end{proof}

\begin{void}
\label{void-torus-classification-complete}
Let $T$ be an $X$-torus with sheaf of cocharacters $\Lambda$. We now complete the classification of factorization super central extensions of $\cal LT$ by $\mathbb G_{m, \Ran}$ with tame commutator.

In light of the equivalence between factorization super line bundles over $\Gr_T$ and $\vartheta_+^{\super}(\Lambda)$ (Proposition \ref{prop-tori-grassmannian-classification}), it remains to prove the following assertion.
\end{void}

\begin{prop}
\label{prop-torus-loop-group-to-grassmannian}
The functor \eqref{eq-loop-group-descent-to-grassmannian} induces an equivalence between:
\begin{enumerate}
	\item factorization super central extensions of $\cal LT$ with tame commutator; and
	\item factorization super line bundles over $\Gr_T$.
\end{enumerate}
\end{prop}
\begin{proof}
Passing through the equivalence \eqref{eq-loop-group-descent-to-hecke}, we replace \eqref{eq-loop-group-descent-to-grassmannian} by the forgetful functor:
\begin{equation}
\label{eq-hecke-to-grassmannian-forgetful-torus}
	\Hom_{\fact}(\Hec_T, \Pic^{\super}) \rightarrow \Gamma_{\fact}(\Gr_T, \Pic^{\super}),
\end{equation}
defined via pullback along $\Gr_T \rightarrow \Hec_T$.

The desired equivalence amounts to showing that every factorization super line bundle $\cal L$ over $\Gr_T$ admits a unique collection of the following pieces of structure:
\begin{enumerate}
	\item an $\cal L^+T$-equivariance structure;
	\item compatibility data with convolution, \emph{i.e.}~\eqref{eq-compatibility-with-convolution-product} and \eqref{eq-compatibility-with-convolution-unit}, on the factorization super line bundle over $\Hec_T$ induced from structure (1),
\end{enumerate}
subject to the \emph{tameness condition}: the induced factorization super central extension of $\cal LT$ has tame commutator.

Since the $\cal L^+T$-action on $\Gr_T$ is trivial, we may view an $\cal L^+T$-equivariance structure on $\cal L$ as a morphism:
\begin{equation}
\label{eq-arc-equivariance-as-morphism}
\cal L^+T \times_{\Ran} \Gr_T \rightarrow \mathbb G_{m, \Ran},
\end{equation}
linear in the first variable. On the other hand, given a factorization super central extension of $\cal LT$ with commutator $\langle\cdot,\cdot\rangle : \cal LT\otimes\cal LT\rightarrow \mathbb G_{m, \Ran}$, the $\cal L^+T$-equivariance structure on the induced factorization super line bundle over $\Gr_T$ is precisely the map $\cal L^+T\times\Gr_T \rightarrow\mathbb G_{m, \Ran}$ associated to $\langle\cdot, \cdot\rangle$ by restriction (\emph{cf.}~\S\ref{void-bilinear-pairing-induces-skew-pairing}). By the tameness condition, the morphism \eqref{eq-arc-equivariance-as-morphism} must then be of the form $\langle\cdot,\cdot\rangle_{b_1}$ for some bilinear form:
\begin{equation}
\label{eq-arc-equivariance-bilinear-form}
b_1 : \Lambda \otimes \Lambda \rightarrow \mathbb Z.
\end{equation}

Let us now analyze the compatibility data with convolution. The existence and uniqueness of the unital structure \eqref{eq-compatibility-with-convolution-unit} follows from the canonical triviality of factorization super line bundles over $\Ran$ and the bilinearity of \eqref{eq-arc-equivariance-bilinear-form}. The multiplicative structure \eqref{eq-compatibility-with-convolution-product} amounts to an isomorphism of factorization super line bundles over $\cal LT \times^{\cal L^+T} \Gr_T$ equivariant against the leftmost $\cal L^+T$-action. Triviality of the $\cal L^+T$-action on $\Gr_T$ yields an isomorphism:
$$
\cal LT\times^{\cal L^+T}\Gr_T \cong \Gr_T\times_{\Ran}\Gr_T,
$$
Since the $\cal L^+T$-equivariance is defined by $b_1$, the multiplicative structure \eqref{eq-compatibility-with-convolution-product} amounts to an isomorphism of factorization line bundles over $\Gr_T\times\Gr_T$:
\begin{equation}
\label{eq-torus-multiplicativity}
m^*(\cal L) \cong (\cal L\boxtimes\cal L) \otimes \cal O(b_1),
\end{equation}
where $\cal O(b_1)$ is the factorization line bundle associated to the form $b_1$ in \S\ref{void-pairing-induced-factorization-line-bundle}.

Under the equivalence of Proposition \ref{prop-tori-grassmannian-classification}, $\cal L$ corresponds to a pair $(b, F_+)$. Applying Proposition \ref{prop-tori-grassmannian-classification} to $T\times T$, we see that \eqref{eq-torus-multiplicativity} exists if and only if $b_1 = b$. Indeed, the quadratic form equates $Q(\lambda_1 + \lambda_2)$ with $Q(\lambda_1) + Q(\lambda_2) + b_1(\lambda_1, \lambda_2)$ for each $\lambda_1, \lambda_2\in\Lambda$, using Lemma \ref{lem-contou-carrere-line-calculation}(1). When $b_1 = b$, Lemma \ref{lem-contou-carrere-line-calculation}(2) yields a canonical isomorphism of $\omega$-monoidal morphisms associated to the two sides of \eqref{eq-torus-multiplicativity}. The isomorphism \eqref{eq-torus-multiplicativity} thus defined is the unique one satisfying the cocycle condition.
\end{proof}

\begin{void}
\label{void-tori-commutator}
Consider any factorization super central extension of $\cal LT$ with tame commutator:
\begin{equation}
\label{eq-torus-factorization-super-central-extension}
1 \rightarrow \mathbb G_{m, \Ran} \rightarrow \cal T \rightarrow \cal LT \rightarrow 1.
\end{equation}
The equivalences of Proposition \ref{prop-tori-grassmannian-classification}, Proposition \ref{prop-torus-loop-group-to-grassmannian} show that \eqref{eq-torus-factorization-super-central-extension} is classified by a pair $(b, F_+)$ in $\vartheta^{\super}_+(\Lambda)$.

In the course of the proof of Proposition \ref{prop-torus-loop-group-to-grassmannian}, we have also established the fact that the commutator pairing of \eqref{eq-torus-factorization-super-central-extension} equals $\langle\cdot, \cdot\rangle_b$.
\end{void}

\subsection{Simply connected groups}
\label{sec-classification-simply-connected}

\begin{void}
Let $G$ denote a semisimple and simply connected group $X$-scheme.

Since $\Gr_G \rightarrow \Ran$ (resp.~$\cal LG \rightarrow \Ran$) has connected geometric fibers, every factorization super line bundle over $\Gr_G$ (resp.~factorization super central extension of $\cal LG$ by $\mathbb G_{m, \Ran}$) is pure of even grading.
\end{void}

\begin{prop}
\label{prop-simply-connected-descent}
If $G$ is semisimple and simply connected, then the functor \eqref{eq-loop-group-descent-to-grassmannian} is an equivalence between:
\begin{enumerate}
	\item factorization central extensions of $\cal LG$; and
	\item factorization line bundles over $\Gr_G$.
\end{enumerate}
\end{prop}

\begin{void}
Before proving Proposition \ref{prop-simply-connected-descent}, we define Schubert varieties in $\Gr_G$ as flat schematic morphisms to $\Ran$. We give a detailed presentation because Lemma \ref{lem-schubert-variety-properties} below was also used in the proof of \cite[Lemma 3.6]{MR4322626} but the justification there is inadequate.

Let us assume that $G$ contains a Borel subgroup and a maximal torus $T \subset B\subset G$. Denote by $\Lambda^+\subset\Lambda$ the subsheaf of dominant cocharacters of $T$. For an $I$-tuple $\lambda^I = (\lambda^i)$ of elements of $\Lambda^+$, we may view $\iota_{\lambda^I}$ of \S\ref{void-torus-grassmannian-subschemes} as a closed immersion $X^I \rightarrow \Gr_{G, X^I}$. Denote by $\Gr_G^{\le \lambda^I}\subset\Gr_{G, X^I}$ the schematic image of the map $\cal L^+_{X^I}G \rightarrow \Gr_{G, X^I}$ defined by acting on $\iota_{\lambda^I}(X^I)$.

Since $G$ is semisimple and simply connected, $\Gr_{G, X^I}$ is reduced and $\Gr_{G, X^I}\rightarrow X^I$ has connected geometric fibers. In particular, the above closed subschemes define an isomorphism of indschemes:
\begin{equation}
\label{eq-simply-connected-schubert-presentation}
	\colim_{\lambda^I} \Gr_G^{\le \lambda^I} \xrightarrow{\simeq} \Gr_{G, X^I}.
\end{equation}
\end{void}

\begin{lem}
\label{lem-schubert-variety-properties}
For each $I$-tuple $\lambda^I$ of elements of $\Lambda^+$, the projection $p : \Gr_G^{\le\lambda^I} \rightarrow X^I$ is flat and the canonical map below is an isomorphism:
\begin{equation}
\label{eq-schubert-variety-o-connected}
\cal O_{X^I} \rightarrow Rp_*\cal O_{\Gr_G^{\le\lambda^I}}.
\end{equation}
\end{lem}
\begin{proof}
For $|I| \le 2$, flatness of $p$ is established in \cite[\S1.2]{MR2508931}. The argument below which applies to general $I$ is explained to me by Jo\~{a}o Louren\c{c}o. We call a morphism $f : Y_1 \rightarrow Y_2$ of schemes \emph{derived $\cal O$-connected} if the induced map $\cal O_{Y_2} \rightarrow Rf_*\cal O_{Y_1}$ is an equivalence.

We begin by recalling some classical facts taking place over geometric points of $X$. Let $x$ be a $\bar k$-point of $X$ and $\Gr_{G, x}$ the fiber of $\Gr_G$. Let $W$ denote the Weyl group of $(G, T)$ and $W_{\aff} := \Lambda \rtimes W$ its affinization. Write $I_x\subset\cal L_x^+G$ for the Iwahori group scheme associated to the Borel $B$. The affine flag variety $\Fl_{G, x} := \cal L_xG/I_x$ has $I_x$-orbits parametrized by $W_{\aff}$. For each $\lambda \in \Lambda^+$, the preimage of $\Gr_{G, x}^{\le\lambda}$ along $\Fl_{G, x} \rightarrow \Gr_{G, x}$ concides with $\Fl_{G, x}^{\le w(\lambda)}$, the closure of the $I_x$-orbit corresponding to the longest element $w(\lambda)$ in $W\lambda W\subset W_{\aff}$.

The scheme $\Fl_{G, x}^{\le w(\lambda)}$ admits a Demazure resolution $D^{w(\lambda)}$ associated to any reduced expression of $w(\lambda)$:
\begin{equation}
\label{eq-demazure-resolution}
D^{w(\lambda)} \xrightarrow{\pi} \Fl_{G, x}^{\le w(\lambda)} \rightarrow \Gr_{G, x}^{\le\lambda}.
\end{equation}

By \cite[Theorem 8]{MR1961134}, the morphism $\pi$ is derived $\cal O$-connected. Since $G/B \rightarrow \Spec(\bar k)$ is derived $\cal O$-connected \cite{MR409474}, the same holds for the composition \eqref{eq-demazure-resolution}. We collect two consequences of this fact:
\begin{enumerate}
	\item $\Gr_{G, x}^{\le\lambda} \rightarrow \Spec(\bar k)$ is derived $\cal O$-connected; indeed, this is because $D^{w(\lambda)}\rightarrow\Spec(\bar k)$ is derived $\cal O$-connected, being an iterated $\mathbb P^1$-bundle.
	\item the convolution map $\Gr_{G, x}^{\le\lambda_1} \, \widetilde{\times} \, \Gr_{G, x}^{\le\lambda_2} \rightarrow \Gr_{G, x}^{\le\lambda_1 + \lambda_2}$ is derived $\cal O$-connected; indeed, the source and target both admit rational resolutions in the sense of \cite[Definition 9.1]{kovacs2022rational}, so this claim follows from \cite[Theorem 9.12(i)]{kovacs2022rational}.
\end{enumerate}

To prove that $p : \Gr_G^{\le\lambda^I} \rightarrow X^I$ is flat, we consider the global convolution map $m : \widetilde{\Gr}{}_G^{\le\lambda^I} \rightarrow \Gr_G^{\le\lambda^I}$ over $X^I$, where the composition $\widetilde{\Gr}{}_G^{\le\lambda^I} \rightarrow X^I$ is evidently flat. Statement (2) implies that for each $\bar k$-point $x^I$ of $X^I$, the base change $m_{x^I}$ of $m$ satisfies $R^i(m_{x^I})_*\cal O \cong 0$ for $i\ge 1$. By \cite[Proposition 3.13]{MR1978342}, $m_*\cal O$ is $X^I$-flat and its formation is compatible with base change along $X^I$. Since $\Gr_G^{\le\lambda^I}$ is reduced and $m$ is surjective on $\bar k$-points, we see that $m_*\cal O$ coincides with the structure sheaf of $\Gr_G^{\le\lambda^I}$. This implies that $p$ is flat and its geometric fibers are identified with the corresponding Schubert varieties. The derived $\cal O$-connectedness \eqref{eq-schubert-variety-o-connected} then follows from its pointwise version, \emph{i.e.}~statement (1) above.
\end{proof}

\begin{void}
\label{void-simply-connected-fiberwise}
Lemma \ref{lem-schubert-variety-properties} has the following consequence: given any $S$-point $x^I : S\rightarrow X^I$, a line bundle $\cal L$ over $\Gr_{G, x^I}$ descends to $S$ if and only if it is trivial over all geometric fibers. Indeed, if $\cal L$ is trivial over all geometric fibers, then by Lemma \ref{lem-schubert-variety-properties} and cohomology and base change, the derived pushforward of $\cal L$ to $S$ yields the desired descent.

On the other hand, \cite[Theorem 7]{MR1961134} proves that the Picard group of $\Gr_{G, x}$, for any $\bar k$-point $x$ of $X$, is isomorphic to $\mathbb Z^r$, where $r$ denotes the number of simple factors of $G$, with $(1,\cdots, 1)\in\mathbb Z^r$ corresponding to an ample line bundle over $\Gr_{G, x}$.
\end{void}

\begin{proof}[Proof of Proposition \ref{prop-simply-connected-descent}]
The problem is of \'etale locally nature on $X$, so we may assume that $G$ contains a Borel subgroup and a maximal torus $T\subset B\subset G$.

Let $\cal L$ be a factorization line bundle over $\Gr_G$. According to \S\ref{void-descent-from-loop-group}, it suffices to prove that $\cal L$ admits a unique $\cal L^+G$-equivariance structure and the induced factorization line bundle over $\Hec_G$ admits unique compatibility data with respect to convolution.

Consider an $S$-point $(x^I, g)$ of $\cal L_{X^I}^+G$. The action by $g$ defines an automorphism $\act_g$ of $\Gr_{G, x^I}$. There is a unique isomorphism:
\begin{equation}
\label{eq-simply-connected-arc-equivariance}
(\act_g)^*\cal L \xrightarrow{\simeq} \cal L
\end{equation}
extending the identity over the unit section $e : S\rightarrow \Gr_{G, x^I}$. Indeed, this is because the difference $(\act_g)^*\cal L\otimes\cal L^{-1}$ is trivial along geometric fibers, so we may apply the observation of \S\ref{void-simply-connected-fiberwise}. The uniqueness of \eqref{eq-simply-connected-arc-equivariance} implies that it satisfies the cocycle condition.

The compatibility data with respect to convolution consist of isomorphisms \eqref{eq-compatibility-with-convolution-product} and \eqref{eq-compatibility-with-convolution-unit}. The second isomorphism is clear. The first one amounts to an isomorphism of line bundles over:
$$
\Gr_G\,\widetilde{\times}\,\Gr_G := \cal LG \times^{\cal L^+G} \Gr_G,
$$
compatible with the left $\cal L^+G$-equivariance. This isomorphism is constructed in the same way as \eqref{eq-simply-connected-arc-equivariance}, by reducing to geometric fibers over $\Ran$.
\end{proof}

\begin{void}
\label{void-simply-connected-summary}
Suppose that $G$ contains a maximal torus $T$ with sheaf of cocharacters $\Lambda$. The Weyl group $W$ acts naturally on $\Lambda$. Restricting along $T\subset G$ and applying Proposition \ref{prop-tori-grassmannian-classification}, each factorization line bundle over $\Gr_G$ defines a pair $(b, F_+)$ with $b(\lambda, \lambda)\in 2\mathbb Z$, hence quadratic form $Q : \lambda\mapsto b(\lambda, \lambda)/2$ on $\Lambda$.

By \cite[Proposition 2.5]{MR4322626}, the quadratic form $Q$ is Weyl-invariant and this procedure defines an equivalence of Picard groupoids between factorization line bundles over $\Gr_G$ and Weyl-invariant quadratic forms $\Quad(\Lambda, \mathbb Z)^W$ on $\Lambda$. (Evaluation on short coroots belonging to each simple factor of $G$ defines an isomorphism $\Quad(\Lambda, \mathbb Z)^W\cong \mathbb Z^{\oplus r}$.)

In summary, all of the Picard groupoids below are canonically equivalent when $G$ is semisimple and simply connected:
\begin{equation}
\label{eq-simply-connected-diagram}
\begin{tikzcd}[column sep = 1em]
	\Hom_{\fact}(\cal LG, \Pic) \ar[r, "\simeq"]\ar[d, "\cong"] & \Hom_{\fact}(\cal LG, \Pic^{\super}) \ar[d, "\cong"] \\
	\Gamma_{\fact}(\Gr_G, \Pic) \ar[r, "\simeq"]\ar[d, "\cong"] & \Gamma_{\fact}(\Gr_G, \Pic^{\super}) \\
	\Quad(\Lambda, \mathbb Z)^W
\end{tikzcd}
\end{equation}

In particular, composing these equivalences with the restriction along $T\subset G$ and the functor \eqref{eq-torus-grassmannian-classification}, we obtain a functor:
\begin{equation}
\label{eq-canonical-theta-data}
	\Quad(\Lambda, \mathbb Z)^W \rightarrow \vartheta_+^{\super}(\Lambda).
\end{equation}

In \cite[\S2.4.7]{MR4322626}, we have verified that the image of $Q$ is the pair $(b, F_+)$, where $F_+$ is the $\omega$-twist of the monoidal morphism $F_Q$ defined in \S\ref{void-omega-twist}.
\end{void}

\begin{rem}
\label{rem-simply-connected-pointwise-classification}
The equivalences in \eqref{eq-simply-connected-diagram} remain valid when factorization central extensions of $\cal LG$ (resp.~factorization line bundles over $\Gr_G$) are replaced by central extensions of $\cal L_xG$ (resp.~line bundles over $\Gr_{G, x}$) for any geometric point $x$ of $X$.
\end{rem}

\begin{void}
\label{void-adjoint-equivariance}
Let us relax the hypothesis and let $G$ be any reductive group $X$-scheme. We shall now use our knowledge about the simply connected case to perform a commutator calculation.

Denote by $G_{\mathrm{sc}}$ the simply connected form of $G$ and $G_{\ad}$ the adjoint form of $G$. The $G_{\ad}$-action on $G$ by conjugation extends to a $G_{\ad}$-action on $G_{\mathrm{sc}}$, which we still refer to as the \emph{conjugation} action.

Consider any factorization central extension:
\begin{equation}
\label{eq-factorization-central-extension-simply-connected}
	1 \rightarrow \mathbb G_{m, \Ran} \rightarrow \cal G_{\mathrm{sc}} \rightarrow \cal LG_{\mathrm{sc}} \rightarrow 1.
\end{equation}
\emph{Claim}: the conjugation $\cal LG_{\ad}$-action on $\cal LG_{\mathrm{sc}}$ extends uniquely to $\cal G_{\mathrm{sc}}$.

Indeed, let $(x^I, g)$ be an $S$-point of $\cal LG_{\ad}$. Action by $g$ defines an automorphism $\act_g$ of $\cal L_{x^I}G_{\mathrm{sc}}$. Viewing $\cal G_{\mathrm{sc}}$ as a multiplicative line bundle over $\cal LG_{\mathrm{sc}}$, we shall argue that there is a unique isomorphism:
\begin{equation}
\label{eq-adjoint-action-lift}
(\act_g)^*\cal G_{\mathrm{sc}, x^I} \xrightarrow{\simeq} \cal G_{\mathrm{sc}, x^I},
\end{equation}
compatible with the multiplicative structure of $\cal G_{\mathrm{sc}, x^I}$. According to \S\ref{void-simply-connected-fiberwise}, it suffices to show that the two sides of \eqref{eq-adjoint-action-lift} are isomorphic on geometric fibers over $S$. This statement holds by Remark \ref{rem-simply-connected-pointwise-classification}, seeing that $(\act_g)^*$ induces the identity map on $\Quad(\Lambda, \mathbb Z)^W$.
\end{void}

\begin{void}
Suppose that $G$ contains a maximal torus $T$ with sheaf of cocharacters $\Lambda$. Write $T_{\mathrm{sc}}$, $T_{\ad}$ for the induced maximal tori in $G_{\mathrm{sc}}$, $G_{\mathrm{ad}}$, with sheaves of cocharacters $\Lambda_{\mathrm{sc}}$, $\Lambda_{\ad}$.

The $\cal LG_{\ad}$-action on $\cal G_{\mathrm{sc}}$ constructed in \S\ref{void-adjoint-equivariance} restricts to an $\cal LT_{\ad}$-action, and we obtain a factorization pairing:
\begin{equation}
\label{eq-simply-connected-pairing}
\cal LT_{\ad}\otimes \cal LT_{\mathrm{sc}} \rightarrow \mathbb G_{m, \Ran},\quad (t_{\ad}, t_{\mathrm{sc}}) \mapsto (t_{\ad} \tilde t_{\mathrm{sc}}t_{\ad}^{-1})\cdot \tilde t_{\mathrm{sc}}^{-1},
\end{equation}
where $\tilde t_{\mathrm{sc}}$ is an arbitrary lift of $t_{\mathrm{sc}}$ to $\cal G_{\mathrm{sc}}$, which exists locally.

To compute this pairing, we recall that $\cal G_{\mathrm{sc}}$ is classified by a Weyl-invariant quadratic form $Q_{\mathrm{sc}}$ on $\Lambda_{\mathrm{sc}}$. Since $\Lambda_{\ad}$ is canonically dual to the root lattice, the formula:
$$
(\lambda, \alpha) \mapsto Q_{\mathrm{sc}}(\alpha)\langle\lambda, \check{\alpha}\rangle
$$
for each $\lambda \in \Lambda_{\ad}$ and coroot $\alpha \in \Lambda_{\mathrm{sc}}$ yields a pairing $b_1 : \Lambda_{\ad} \otimes \Lambda_{\mathrm{sc}} \rightarrow \mathbb Z$.
\end{void}

\begin{lem}
\label{lem-adjoint-equivariance-pairing}
The factorization pairing \eqref{eq-simply-connected-pairing} equals $\langle\cdot,\cdot\rangle_{b_1}$ (in the notation of \S\ref{void-definition-tame}).
\end{lem}
\begin{proof}
The problem is of \'etale local nature on $X$, so we may assume that $G$ is split and $T$ is a split maximal torus. Each coroot $\alpha$ induces a morphism $f_{\alpha} : \SL_2 \rightarrow G_{\mathrm{sc}}$, sending the upper-triangular unipotent matrices to the root subgroup $U_{\alpha}\subset G_{\mathrm{sc}}$ and restricts to $\alpha$ on the diagonal torus $\mathbb G_m \subset \SL_2$.

The central extension \eqref{eq-simply-connected-pairing} restricts along $f_{\alpha}$ to the $Q_{\mathrm{sc}}(\alpha)$-multiple of the factorization central extension:
\begin{equation}
\label{eq-canonical-central-extension-special-linear}
1 \rightarrow \mathbb G_{m, \Ran} \rightarrow \widetilde{\SL}_2 \rightarrow \cal L\SL_2 \rightarrow 1,
\end{equation}
defined by restricting the Tate central extension $\widetilde{\GL}_2$ along $\cal L\SL_2\subset \cal L\GL_2$. Indeed, factorization central extensions of $\cal L\SL_2$ are uniquely determined by their quadratic forms (\S\ref{void-simply-connected-summary}), hence by the commutator of their restrictions to $\cal L\mathbb G_m$, so we conclude by \S\ref{void-tori-commutator} and our definition of the Contou-Carr\`ere symbol.

By functoriality of the construction, it suffices to show that \eqref{eq-simply-connected-pairing} equals the Contou-Carr\`ere pairing for $G = \SL_2$ and $\cal G_{\mathrm{sc}}$ being the central extension \eqref{eq-canonical-central-extension-special-linear}. Note that the $T_{\ad}$ ($=\mathbb G_m$)-action on $G$ ($=\SL_2$) extends to the inner action of $\mathbb G_m$ on $\GL_2$ as the subgroup:
\begin{equation}
\label{eq-general-linear-group-submatrix}
\mathbb G_m\subset\GL_2,\quad
a\mapsto
\begin{pmatrix}
	a & 0 \\
	0 & 1
\end{pmatrix}.
\end{equation}
Using the group structure on $\widetilde{\GL}_2$, we extend the induced inner $\cal L\mathbb G_m$-action on $\cal L\GL_2$ to an action on $\widetilde{\GL}_2$. This action must restrict to the $T_{\ad}$-action on the subgroup $\widetilde{\SL}_2\subset\widetilde{\GL}_2$, by the uniqueness of the latter (\S\ref{void-adjoint-equivariance}). Therefore, it remains to prove that the commutator pairing in $\widetilde{\GL}_2$ between the subtorus \eqref{eq-general-linear-group-submatrix} and the subtorus:
$$
\mathbb G_m\subset \GL_2,\quad
a\mapsto
\begin{pmatrix}
	a & 0 \\
	0 & a^{-1}
\end{pmatrix}
$$
is the Contou-Carr\`ere symbol. This assertion holds because $\widetilde{\GL}_2$ restricts to the Tate central extension $\widetilde{\mathbb G}_m$ of $\cal L\mathbb G_m$ along \eqref{eq-general-linear-group-submatrix}.
\end{proof}

\subsection{Descent}
\label{sec-classification-reductive}

\begin{void}
\label{void-semi-direct}
Let $\cal K$ and $\cal H$ be factorization group presheaves whose base changes along any $S$-point of $\Ran$ are fppf sheaves.

An action of $\cal H$ on $\cal K$ as group presheaves is \emph{compatible with factorization} if for any disjoint $S$-points $x^I, x^J$ of $\Ran$, the $\cal H_{x^I\sqcup x^J}$-action on $\cal K_{x^I\sqcup x^J}$ coincides with the $\cal H_{x^I}\times\cal H_{x^J}$-action on $\cal K_{x^I}\times\cal K_{x^J}$ under the factorization isomorphisms of $\cal H$ and $\cal K$. When this happens, the group presheaf $\cal K\rtimes\cal H$ over $\Ran$ inherits a factorization structure.

Consider a factorization super central extension:
$$
1 \rightarrow \mathbb G_m \rightarrow \widetilde{\cal K} \rightarrow \cal K \rightarrow 1.
$$
Suppose that $\widetilde{\cal K}$ is equipped with an $\cal H$-action which is trivial on $A$. Then we say that the $\cal H$-action on $\widetilde{\cal K}$ is \emph{compatible with factorization} if for disjoint $S$-points $x^I$, $x^J$ of $\Ran$, the $\cal H_{x^I\sqcup x^J}$-action on $\widetilde{\cal K}_{x^I\sqcup x^J}$ coincides with the induced $\cal H_{x^I} \times\cal H_{x^J}$-action on the quotient:
$$
\widetilde{\cal K}_{x^I} \times \widetilde{\cal K}_{x^J} \twoheadrightarrow \widetilde{\cal K}_{x^I\sqcup x^J}.
$$

The following lemma is a variant of \cite[Construction 1.7]{MR1896177}.
\end{void}

\begin{lem}
\label{lem-semidirect-central-extension}
Let $\cal K$, $\cal H$ be as in \S\ref{void-semi-direct} with $\cal H$ acting on $\cal K$ compatibly with factorization. The following categories are equivalent:
\begin{enumerate}
	\item factorization super central extensions of $\cal K\rtimes\cal H$ by $\mathbb G_{m, \Ran}$;
	\item triples $(\widetilde{\cal K}, \widetilde{\cal H}, \alpha)$, where $\widetilde{\cal K}$ (resp.~$\widetilde{\cal H}$) is a factorization super central extension of $\cal K$ (resp.~$\cal H$) by $\mathbb G_{m, \Ran}$, and $\alpha$ is an $\cal H$-action on $\widetilde{\cal K}$ which is trivial on $\mathbb G_{m, \Ran}$, compatible with factorization, and induces the given $\cal H$-action on $\cal K$.
\end{enumerate}
\end{lem}
\begin{proof}
The functor (1) $\Rightarrow$ (2) is given by restricting a factorization central extension of $\cal K\rtimes\cal H$ along the group sub-presheaves $\cal K \subset \cal K\rtimes\cal H$, $\cal H\subset\cal K\rtimes\cal H$ to obtain $\widetilde{\cal K}$, $\widetilde{\cal H}$, and observing that the $\widetilde{\cal H}$-action on $\widetilde{\cal K}$ factors through $\cal H$.

The functor (2) $\Rightarrow$ (1) is given by forming the central extension $\widetilde{\cal K} \rtimes \widetilde{\cal H}$ of $\cal K\rtimes\cal H$ by $\mathbb G_{m, \Ran}\times \mathbb G_{m, \Ran}$ using the action $\alpha$, and pushing out along the product map on $\mathbb G_{m, \Ran}$.
\end{proof}

\begin{prop}
\label{prop-semidirect-central-extension}
Let $G$ be a reductive group $X$-scheme with a maximal torus $T$. Denote by $G_{\mathrm{sc}}$ the simply connected form of $G$, equipped with the conjugation $T$-action. The following categories are canonically equivalent:
\begin{enumerate}
	\item factorization super central extensions of $\cal LG_{\mathrm{sc}} \rtimes \cal LT$ by $\mathbb G_{m, \Ran}$;
	\item pairs $(\cal G_{\mathrm{sc}}, \cal T)$, where $\cal G_{\mathrm{sc}}$ (resp.~$\cal T$) is a factorization super central extension of $\cal LG_{\mathrm{sc}}$ (resp.~$\cal LT$) by $\mathbb G_{m, \Ran}$.
\end{enumerate}
\end{prop}
\begin{proof}
We appeal to the equivalence of Lemma \ref{lem-semidirect-central-extension}. It suffices to prove that given any pair $(\cal G_{\mathrm{sc}}, \cal T)$ in (2), there is a unique $\cal LT$-action on $\cal G_{\mathrm{sc}}$ which is trivial on $\mathbb G_{m, \Ran}$, compatible with factorization, and induces the conjugation action on $\cal LG_{\mathrm{sc}}$. This is established in \S\ref{void-adjoint-equivariance} when $T$ is replaced by $G_{\ad}$, but the argument for uniqueness carries over.
\end{proof}

\begin{void}
Under the equivalence of Proposition \ref{prop-semidirect-central-extension}, we shall write the factorization super central extension of $\cal LG_{\mathrm{sc}} \rtimes \cal LT$ induced from $(\cal G_{\mathrm{sc}}, \cal T)$ as follows:
$$
1 \rightarrow \mathbb G_{m, \Ran} \rightarrow \cal G_{\mathrm{sc}} \widetilde{\rtimes} \cal T \rightarrow \cal LG_{\mathrm{sc}} \rtimes \cal LT \rightarrow 1.
$$
It is by construction the pushout of $\cal G_{\mathrm{sc}} \rtimes\cal T$ along the product map on $\mathbb G_{m, \Ran}$.
\end{void}

\begin{lem}
Let $\widetilde G \rightarrow G$ be a surjection of reductive group $X$-schemes whose kernel is a torus. The induced morphism $\cal L\widetilde G \rightarrow \cal LG$ is surjective in the topology generated by fpqc and proper covers.
\end{lem}
\begin{proof}
We shall deduce this from two statements:
\begin{enumerate}
	\item $\cal L^+\widetilde G \rightarrow \cal L^+G$ is surjective in the fpqc topology;
	\item $\Gr_{\widetilde G} \rightarrow \Gr_G$ is surjective in the topology generated by proper covers.
\end{enumerate}

Let us prove the lemma assuming both statements. Indeed, the morphism $\cal L\widetilde G \rightarrow \cal LG$ factors as:
$$
\cal L\widetilde G \rightarrow \cal LG\times_{\Gr_G}\Gr_{\widetilde G} \rightarrow \cal LG.
$$
By statement (2), the second morphism is surjective in the proper topology. We claim that the first morphism is surjective in the fpqc topology. Indeed, consider an $S$-point $(g, x)$ of $\cal LG\times_{\Gr_G}\Gr_{\widetilde G}$. Since $\cal L\widetilde G \rightarrow\Gr_{\widetilde G}$ is surjective in the \'etale topology, we may lift $x$ to an $S_1$-point $\tilde g$ of $\cal L\widetilde G$ over some \'etale cover $S_1\rightarrow S$. The image of $\tilde g$ in $\cal LG$ differs from $g$ by an $S_1$-point $h$ of $\cal L^+G$. Statement (1) allows us to lift $h$ to an $S_2$-point $\tilde h$ in $\cal L^+\widetilde G$ over some fpqc cover $S_2\rightarrow S_1$, which we may then use to modify $\tilde g$ to obtain a lift of $(g, x)$.

To prove statement (1), we consider an $S$-point $x^I$ of $X^I$ with graph $\Gamma_{x^I} \subset S\times X$. The base change $\cal L_{x^I}^+\widetilde G \rightarrow \cal L_{x^I}^+G$ is the inverse limit of morphisms:
\begin{equation}
\label{eq-congruence-subgroup-maps}
R_{\Gamma^{(n)}}\widetilde G \rightarrow R_{\Gamma^{(n)}}G,
\end{equation}
which are the Weil restrictions of $\widetilde G\rightarrow G$ along the finite locally free morphisms $\Gamma^{(n)}_{x^I} \rightarrow S$, see \S\ref{void-arc-group-structure}. Since $\widetilde G\rightarrow G$ is affine, smooth, and surjective, the same holds for \eqref{eq-congruence-subgroup-maps}. Hence $\cal L_{x^I}^+\widetilde G \rightarrow \cal L_{x^I}^+G$ is affine, \emph{flat}, and surjective.

We now turn to statement (2). Since the formation of the affine Grassmannian is compatible with \'etale base change, we may assume that the kernel of $\widetilde G \rightarrow G$ is a \emph{split} torus $T$. This implies that $\cal L\widetilde G \rightarrow \cal LG$ is surjective on field-valued points. Indeed, given any $k$-field $F$, the graph of an $F$-point of $X^I$ is the disjoint union of schemes isomorphic to $\Spec(F\loo{t})$, but the map $\widetilde G(F\loo{t}) \rightarrow G(F\loo{t})$ is surjective because $H^1(F\loo{t}, T) = 0$. In particular, any $F$-point of $\Gr_G$ lifts to $\Gr_{\widetilde G}$ after a finite extension $F\subset F_1$.

On the other hand, the morphism $\Gr_{\widetilde G} \rightarrow \Gr_G$ is ind-proper because $\widetilde G$ is reductive. Since $\Gr_G$ is of finite presentation, taking schematic points of $\Gr_G$ puts us in the following situation: an affine $k$-scheme $S$ of finite type, an ind-proper $S$-indscheme $Y$, such that $Y \rightarrow S$ is surjective on field-valued points up to finite extension. We claim that \emph{some closed subscheme $Y_i\subset Y$ surjects onto $S$.} Then $Y_i \rightarrow S$ is a proper cover which lifts to $Y$.

Let us prove the claim. Since $S$ is of finite type over $k$, we reduce to the case where $S$ is irreducible. Then its generic point lifts to some $Y_i$ after a finite extension. Since $Y_i \rightarrow S$ is proper, its image contains the closure of the generic point which is all of $S$.
\end{proof}

\begin{void}
Let $G$ denote a reductive group $X$-scheme with a maximal torus $T$. Denote by $G_{\mathrm{sc}}$ the simply connected form of $G$ and $T_{\mathrm{sc}} \subset G_{\mathrm{sc}}$ the preimage of $T$. The $T$-action on $G$ by conjugation extends to $G_{\mathrm{sc}}$. There is a short exact sequence:
\begin{equation}
\label{eq-semi-direct-covering}
1 \rightarrow T_{\mathrm{sc}} \rightarrow G_{\mathrm{sc}} \rtimes T \rightarrow G \rightarrow 1,
\end{equation}
where the first map is the anti-diagonal embedding $t\mapsto (t, t^{-1})$. Furthermore, its image is central in $G_{\mathrm{sc}} \rtimes T$.

The exact sequence \eqref{eq-semi-direct-covering} induces an exact sequence of factorization group presheaves:
\begin{equation}
\label{eq-semi-direct-covering-loop}
	1 \rightarrow \cal LT_{\mathrm{sc}} \rightarrow \cal LG_{\mathrm{sc}} \rtimes \cal LT \rightarrow \cal LG,
\end{equation}
where the last map is surjective in the topology generated by fpqc and proper covers. Since perfect complexes satisfy derived proper descent \cite[Theorem 1.8]{chough2022spectral} and loop groups are classical \cite[Theorem 9.3.5]{MR3220628}, central extensions of $\cal LG$ by $\mathbb G_{m, \Ran}$ are equivalent to those of $\cal LG_{\mathrm{sc}} \rtimes \cal LT$ by $\mathbb G_{m, \Ran}$ equipped with a splitting over $\cal LT_{\mathrm{sc}}$ whose image is normal. (This idea is due to Gaitsgory, \emph{cf.}~\cite[Corollary 5.2.7]{MR4117995}.)
\end{void}

\begin{void}
\label{void-semi-direct-covering}
Let $(G, T)$ be as above. Appealing to Proposition \ref{prop-semidirect-central-extension}, we obtain an equivalence of Picard groupoids between:
\begin{enumerate}
	\item factorization super central extensions of $\cal LG$ by $\mathbb G_{m, \Ran}$; and
	\item triples $(\cal T, \cal G_{\mathrm{sc}}, \varphi)$, where $\cal T$ and $\cal G_{\mathrm{sc}}$ are factorization super central extensions:
\begin{align}
\label{eq-torus-central-extension}
	1 \rightarrow \mathbb G_{m, \Ran} \rightarrow &\cal T \rightarrow \cal LT \rightarrow 1, \\
\label{eq-simply-connected-central-extension}
	1 \rightarrow \mathbb G_{m, \Ran} \rightarrow &\cal G_{\mathrm{sc}} \rightarrow \cal LG_{\mathrm{sc}} \rightarrow 1,
\end{align}
and $\varphi$ is an isomorphism of their pullbacks to $\cal LT_{\mathrm{sc}}$, subject to the \emph{normality condition} that the section $\cal LT_{\mathrm{sc}} \rightarrow \cal G_{\mathrm{sc}} \widetilde{\rtimes} \cal T$ induced from $\varphi$ has normal image.
\end{enumerate}
\end{void}

\begin{lem}
\label{lem-tame-commutator-center-vs-torus}
A factorization super central extension of $\cal LG$ by $\mathbb G_{m, \Ran}$ has tame commutator if and only if its restriction to $\cal LT$ does.
\end{lem}
\begin{proof}
Note that any factorization (super) central extension of $\cal LG_{\mathrm{sc}}$ by $\mathbb G_{m, \Ran}$ has tame commutator (Lemma \ref{lem-adjoint-equivariance-pairing}). The claim now follows from Lemma \ref{void-bilinear-pairing-induces-skew-pairing}, seeing that $\Rad(G) \times T_{\mathrm{sc}} \rightarrow T$ is an isogeny of $X$-tori.
\end{proof}

\begin{void}
Consider a triple $(\cal T, \cal G_{\mathrm{sc}}, \varphi)$ as in \S\ref{void-semi-direct-covering} and assume that $\cal T$ has tame commutator.

In particular, this implies that its commutator is $\langle\cdot, \cdot\rangle_b$ where $b$ is the symmetric bilinear form appearing in the classifying data of $\cal T$, see \S\ref{void-tori-commutator}.

Under this assumption, we shall make the normality condition of \S\ref{void-semi-direct-covering}(2) explicit.
\end{void}

\begin{lem}
\label{lem-normality-explicit}
If $\cal T$ has commutator $\langle\cdot,\cdot\rangle_b$, then the normality condition holds if and only if $b$ is Weyl-invariant.
\end{lem}
\begin{proof}
Let $\Lambda$ (resp.~$\Lambda_{\ad}$, $\Lambda_{\mathrm{sc}}$) denote the sheaf of cocharacters of $T$ (resp.~$T_{\ad}$, $T_{\mathrm{sc}}$). Note that $\cal T$ has commutator $\langle\cdot,\cdot\rangle_b$ while $\cal G_{\mathrm{sc}}$ defines the pairing $\langle\cdot,\cdot\rangle_{b_1}$ via its $\cal LT_{\ad}$-action (see Lemma \ref{lem-adjoint-equivariance-pairing}). The existence of $\varphi$ implies that $b$ and $b_1$ agree on $\Lambda_{\mathrm{sc}}\otimes\Lambda_{\mathrm{sc}}$. In particular, the restriction of $b$ to $\Lambda_{\mathrm{sc}}$ comes from a Weyl-invariant quadratic form $Q$.

We shall prove that the normality condition holds if and only if $b$ and $b_1$ coincide over $\Lambda \otimes \Lambda_{\mathrm{sc}}$. This latter condition means that for each $\lambda\in\Lambda$ and root $\alpha\in\Lambda_{\tn{sc}}$, there holds:
\begin{equation}
\label{eq-strict-weyl-invariance}
b(\lambda, \alpha) = Q(\alpha)\langle\check{\alpha},\lambda\rangle.
\end{equation}
The equality \eqref{eq-strict-weyl-invariance} is equivalent to the Weyl-invariance of $b$.

Let us now analyze the normality condition. The section induced from $\varphi$ has the following description on $S$-points:
\begin{equation}
\label{eq-isomorphism-induced-section}
\cal LT_{\mathrm{sc}} \rightarrow \cal G_{\mathrm{sc}} \widetilde{\rtimes} \cal T,\quad t_{\mathrm{sc}} \mapsto (\tilde t_{\mathrm{sc}}, \varphi(\tilde t_{\mathrm{sc}})^{-1}),
\end{equation}
where $\tilde t_{\mathrm{sc}}$ is any lift of $t_{\mathrm{sc}}$ to $\cal G_{\mathrm{sc}}$ and $\varphi(\tilde t_{\mathrm{sc}})$ is its image in $\cal T$ under $\varphi$. Since $\cal LT_{\mathrm{sc}}$ is a central subgroup of $\cal LG_{\mathrm{sc}} \rtimes \cal LT$, the image of \eqref{eq-isomorphism-induced-section} is normal if and only if it is central. This condition translates to the following equality in $\cal G_{\mathrm{sc}} \widetilde{\rtimes} \cal T$:
\begin{equation}
\label{eq-conjugation-action-fixed}
(g_{\mathrm{sc}}, t) \cdot (\tilde t_{\mathrm{sc}}, \varphi(\tilde t_{\mathrm{sc}})^{-1}) \cdot (g_{\mathrm{sc}}, t)^{-1} = (\tilde t_{\mathrm{sc}}, \varphi(\tilde t_{\mathrm{sc}})^{-1}),
\end{equation}
for all $S$-points $(g_{\mathrm{sc}}, t)$ of $\cal LG_{\mathrm{sc}} \rtimes \cal LT$ and $t_{\mathrm{sc}}$ of $\cal LT_{\mathrm{sc}}$.

The left-hand-side of \eqref{eq-conjugation-action-fixed} computes to $(\langle(t, \tilde t_{\mathrm{sc}}\rangle_{b_1}\tilde t_{\mathrm{sc}}, \langle t, \varphi(\tilde t_{\mathrm{sc}})^{-1}\rangle_b \varphi(\tilde t_{\mathrm{sc}})^{-1})$. Its equality with the right-hand-side amounts to the equality:
$$
\langle t, \tilde t_{\mathrm{sc}}\rangle_{b_1} = \langle t, \varphi(\tilde t_{\mathrm{sc}})\rangle_b,
$$
for all $S$-points $t$ of $\cal LT$ and $t_{\mathrm{sc}}$ of $\cal LT_{\mathrm{sc}}$, \emph{i.e.}~the agreement of $b$ and $b_1$ over $\Lambda\otimes\Lambda_{\mathrm{sc}}$.
\end{proof}

\begin{void}
Define the Picard groupoid $\widetilde{\vartheta}_{G, +}^{\super}(\Lambda)$ by the Cartesian diagram:
$$
\begin{tikzcd}[column sep = 1em]
	\widetilde{\vartheta}_{G, +}^{\super}(\Lambda) \ar[r]\ar[d] & \Quad(\Lambda_{\mathrm{sc}}, \mathbb Z)^W \ar[d, "\eqref{eq-canonical-theta-data}"] \\
	\theta_+^{\super}(\Lambda) \ar[r] & \theta_+^{\super}(\Lambda_{\tn{sc}})
\end{tikzcd}
$$
Then $\vartheta_{G, +}^{\super}(\Lambda)$ can be viewed as the full subgroupoid of $\tilde{\vartheta}_{G, +}^{\super}(\Lambda)$ consisting of objects whose images in $\theta_+^{\super}(\Lambda)$ have a \emph{Weyl-invariant} form $b$.

Pulling back along $T\subset G$, $G_{\mathrm{sc}}\subset G$, and using the compatibility over $T_{\mathrm{sc}}$, we obtain a functor:
\begin{equation}
\label{eq-factorization-line-bundle-classification-morphism}
	\Gamma_{\fact}(\Gr_G, \Pic^{\super}) \rightarrow \widetilde{\vartheta}_{G, +}^{\super}(\Lambda).
\end{equation}
\end{void}

\begin{void}
We are now ready to establish the equivalence (1) $\cong$ (2) $\cong$ (3) in Theorem \ref{thm-factorization-super-central-extension-classification}. We shall do so using the equivalence of \S\ref{void-semi-direct-covering}, together with an argument from \cite{MR4322626} showing that factorization super line bundles over $\Gr_G$ embed fully faithfully into $\widetilde{\vartheta}_{G, +}^{\super}(\Lambda)$.

Denote by:
$$
\Hom_{\fact}^{\tame}(\cal LG, \Pic^{\super}) \subset \Hom_{\fact}(\cal LG, \Pic^{\super})
$$
the full subgroupoid of factorization super central extensions of $\cal LG$ by $\mathbb G_{m, \Ran}$, characterized by the property of having tame commutator.
\end{void}

\begin{prop}
\label{prop-super-central-extension-classification}
Let $G$ be a reductive group $X$-scheme equipped with a maximal torus $T$ with sheaf of cocharacters $\Lambda$. The functors \eqref{eq-loop-group-descent-to-grassmannian} and \eqref{eq-factorization-line-bundle-classification-morphism} induce equivalences among the following Picard groupoids:
\begin{equation}
\label{eq-super-central-extension-classification}
	\Hom_{\fact}^{\tame}(\cal LG, \Pic^{\super}) \xrightarrow{\simeq} \Gamma_{\fact}(\Gr_G, \Pic^{\super}) \xrightarrow{\simeq} \vartheta_{G, +}^{\super}(\Lambda).
\end{equation}
\end{prop}
\begin{proof}
The functors \eqref{eq-loop-group-descent-to-grassmannian} and \eqref{eq-factorization-line-bundle-classification-morphism} \emph{a priori} define:
$$
\Hom_{\fact}^{\tame}(\cal LG, \Pic^{\super}) \rightarrow  \Gamma_{\fact}(\Gr_G, \Pic^{\super}) \rightarrow \widetilde{\vartheta}_{G, +}^{\super}(\Lambda),
$$

The composition is an equivalence onto $\vartheta_{G, +}^{\super}(\Lambda)$---this is the equivalence of \S\ref{void-semi-direct-covering} restricted to the subgroupoid characterized by the tameness condition, as we see from Lemma \ref{lem-tame-commutator-center-vs-torus} and Lemma \ref{lem-normality-explicit}. Hence it suffices to prove that \eqref{eq-factorization-line-bundle-classification-morphism} is fully faithful and its essential image is contained in $\vartheta_{G, +}^{\super}(\Lambda)$.

The fully faithfulness is the content of \cite[\S3.2]{MR4322626}. The assertion that its image lies in $\vartheta_{G, +}^{\super}(\Lambda)$ amounts to establishing the Weyl-invariance of the bilinear form $b$ associated to any factorization super line bundle over $\Gr_G$. Since $\Rad(G) \times T_{\mathrm{sc}} \rightarrow T$ is a Weyl-equivariant isogeny of $X$-tori, the statement can be proved when $G$ is replaced by $\Rad(G) \times G_{\mathrm{sc}}$. In this case, we claim that the external product:
$$
\boxtimes : \Gamma_{\fact}(\Gr_{\Rad(G)}, \Pic^{\super}) \times \Gamma_{\fact}(\Gr_{G_{\mathrm{sc}}}, \Pic) \rightarrow \Gamma_{\fact}(\Gr_{\Rad(G)\times G_{\mathrm{sc}}}, \Pic^{\super})
$$
is an equivalence of Picard groupoids. Indeed, this follows from the fiberwise characterization of line bundles over $\Gr_{\Rad(G)\times G_{\mathrm{sc}}}$ which descend to $\Gr_{\Rad(G)}$ (\S\ref{void-simply-connected-fiberwise}).
\end{proof}

\subsection{Poor man's transgression}
\label{sec-poor-transgression}

\begin{void}
Let $G$ be a reductive group $X$-scheme. Denote by $\deloop G$ the $X$-stack classifying \emph{Zariski} locally trivial $G$-torsors.

The following result completes the equivalence (1) $\cong$ (4) in Theorem \ref{thm-factorization-super-central-extension-classification}.
\end{void}

\begin{prop}
\label{prop-trangression-torus-independence}
Fix a $\vartheta$-characteristic $\omega^{1/2}$ over $X$. There is a canonical equivalence of Picard groupoids:
\begin{equation}
\label{eq-poor-man-transgression}
\int_{(\mathring D, \omega^{1/2})} : \Gamma_e(\deloop G, \Ktheory{}_{[1, 2]}^{\super}) \xrightarrow{\simeq} \Hom_{\fact}^{\tame}(\cal LG, \Pic^{\super}).
\end{equation}

Furthermore, if $G$ is equipped with a maximal torus $T$ with sheaf of cocharacters $\Lambda$. Then the following diagram is canonically commutative:
\begin{equation}
\label{eq-poor-man-transgression-compatibility}
\begin{tikzcd}
	\Gamma_e(\deloop G, \Ktheory{}_{[1, 2]}^{\super}) \ar[r, "\int_{(\mathring D, \omega^{1/2})}"]\ar[d, "\eqref{eq-classification-super}"] & \Hom_{\fact}^{\tame}(\cal LG, \Pic^{\super}) \ar[d, "\eqref{eq-super-central-extension-classification}"] \\
	\vartheta_{G}^{\super}(\Lambda) \ar[r, "\eqref{eq-super-theta-data-omega-shift}"] & \vartheta_{G, +}^{\super}(\Lambda)
\end{tikzcd}
\end{equation}
\end{prop}

\begin{void}
In an ideal world, we would define \eqref{eq-poor-man-transgression} by a ``transgression'' on $\tn K$-theory and verify the commutativity of \eqref{eq-poor-man-transgression-compatibility}, but we do not know how to do so.

In what follows, we offer a poor man's substitute: we first fix a maximal torus contained in Borel subgroup $T\subset B\subset G$ (which exists \'etale locally on $X$) and \emph{define} \eqref{eq-poor-man-transgression} as the composition of the three equivalences in \eqref{eq-poor-man-transgression-compatibility}. We denote this functor by $\Phi_{(T, B)}$ to emphasize its \emph{a priori} dependence on $(T, B)$.

Then we prove that $\Phi_{(T, B)}$ is canonically independent of $(T, B)$, \emph{i.e.}~given two pairs $T_i\subset B_i\subset G$ (for $i = 1, 2$) of maximal tori contained in Borel subgroups, there is a canonical isomorphism of functors:
\begin{equation}
\label{eq-killing-pair-comparison}
\alpha_{(T_1, B_1), (T_2, B_2)} : \Phi_{(T_1, B_1)} \xrightarrow{\simeq} \Phi_{(T_2, B_2)}.
\end{equation}
satisfiying the cocycle condition for three such pairs $(T_i, B_i)$ ($i = 1, 2, 3$).

These canonical isomorphisms allow us to glue the functors $\Phi_{(T, B)}$ over an \'etale cover of $X$, which yields the equivalence \eqref{eq-poor-man-transgression}.
\end{void}

\begin{rem}
The definition of $\Phi_{(T, B)}$ uses only $T$. However, the isomorphism \eqref{eq-killing-pair-comparison} depends on $B_1$ and $B_2$, so we prefer keep the Borel subgroups in the notation.
\end{rem}

\begin{proof}[Proof of Proposition \ref{prop-trangression-torus-independence}]
It suffices to construct \eqref{eq-killing-pair-comparison} satisfying the cocycle condition.

Suppose that $t$ is an $X$-point of $T$. The inner automorphism $\inner_t : G\rightarrow G$, $g\mapsto tgt^{-1}$ preserves $T\subset G$ and induces a commutative diagram:
\begin{equation}
\label{eq-interior-automorphism-k-theory}
\begin{tikzcd}[column sep = 1.5em]
	\Gamma_e(\deloop G, \Ktheory{}_{[1, 2]}^{\super}) \ar[d, "\inner_t^*"]\ar[r, "\eqref{eq-classification-super}"] & \vartheta^{\super}_G(\Lambda)  \ar[d, "\inner_t^*"] \\
	\Gamma_e(\deloop G, \Ktheory{}_{[1, 2]}^{\super}) \ar[r, "\eqref{eq-classification-super}"] & \vartheta^{\super}_G(\Lambda)
\end{tikzcd}
\end{equation}
Rigidified sections of $\Ktheory{}_{[1, 2]}^{\super}$ over $BG$ are equivalent to monoidal functors $G \rightarrow \Omega(\Ktheory{}_{[1, 2]}^{\super})$. Since the target has a \emph{symmetric} monoidal structure, $\inner_t^*$ acts as the identity on the groupoid of such monoidal functors.

On the other hand, the right vertical functor $\inner_t^*$ in \eqref{eq-interior-automorphism-k-theory} carries a triple $(b, \widetilde{\Lambda}, \varphi)$ to the triple $(b, \widetilde{\Lambda}, \varphi\cdot t^b)$, where $t^b : \Lambda \rightarrow \mathbb G_m$ denotes the character sending $\lambda \in \Lambda$ to the character $b(\lambda, -) : T \rightarrow \mathbb G_m$ evaluated at $t$, and $\varphi\cdot t^b$ is the sum of $\varphi$ with the restriction of $t^b$ to $\Lambda_{\mathrm{sc}}$. The $2$-isomorphism rendering \eqref{eq-interior-automorphism-k-theory} commutative evalutes to the isomorphism:
$$
(b, \widetilde{\Lambda}, \varphi) \xrightarrow{\simeq} \inner_t^*(b, \widetilde{\Lambda}, \varphi) \xrightarrow{\simeq} (b, \widetilde{\Lambda}, \varphi\cdot t^b),
$$
induced by the automorphism of the central extension $\widetilde{\Lambda}$ defined by $t^b$. These calculations are performed using the description of the commutator in \cite[Proposition 3.13]{MR1896177}.

The situation is parallel for the functor \eqref{eq-super-central-extension-classification}: an $X$-point $t$ of $T$ induces a commutative diagram:
\begin{equation}
\label{eq-interior-automorphism-super-central-extension}
\begin{tikzcd}[column sep = 1.5em]
	\Hom_{\fact}^{\tame}(\cal LG, \Pic^{\super}) \ar[d, "\inner_t^*"]\ar[r, "\eqref{eq-super-central-extension-classification}"] & \vartheta_{G, +}^{\super}(\Lambda)  \ar[d, "\inner_t^*"] \\
	\Hom_{\fact}^{\tame}(\cal LG, \Pic^{\super}) \ar[r, "\eqref{eq-super-central-extension-classification}"] & \vartheta_{G, +}^{\super}(\Lambda)
\end{tikzcd}
\end{equation}
where the left vertical functor is isomorphic to the identity, the right vertical functor sends $(b, \widetilde{\Lambda}_+, \varphi)$ to $(b, \widetilde{\Lambda}_+, \varphi\cdot t^b)$, and the $2$-isomorphism rendering \eqref{eq-interior-automorphism-super-central-extension} commutative is given by the automorphism $t^b$ of $\widetilde{\Lambda}_+$. These calculations are performed using the description of the commutator in \S\ref{void-tori-commutator}.

Suppose now that $T_i \subset B_i \subset G$ (for $i = 1, 2$) are a pair of maximal tori contained in Borel subgroups and $g$ is an $X$-point of $G$ with $\inner_g(T_2) = T_1$, $\inner_g(B_2) = B_1$. Denote by $\Lambda_i$ the sheaf of cocharacters of $T_i$.

The inner automorphism $\inner_g$ gives rise to a commutative diagram:
\begin{equation}
\label{eq-interior-automorphism-altogether}
\begin{tikzcd}[column sep = 1.5em]
	\Gamma_e(BG, \Ktheory{}_{[1, 2]}^{\super}) \ar[d, "\inner_g^*"]\ar[r, "\eqref{eq-classification-super}"] & \vartheta^{\super}_G(\Lambda_1)  \ar[d, "\inner_g^*"] \ar[r, phantom, "\simeq"] & \vartheta_{G, +}^{\super}(\Lambda_1)\ar[d, "\inner_g^*"] & \Hom_{\fact}^{\tame}(\cal LG, \Pic^{\super}) \ar[l, swap, "\eqref{eq-super-central-extension-classification}"]\ar[d, "\inner_g^*"] \\
	\Gamma_e(BG, \Ktheory{}_{[1, 2]}^{\super}) \ar[r, "\eqref{eq-classification-super}"] & \vartheta^{\super}_G(\Lambda_2) \ar[r, phantom, "\simeq"] & \vartheta_{G, +}^{\super}(\Lambda_1) & \Hom_{\fact}^{\tame}(\cal LG, \Pic^{\super}) \ar[l, swap, "\eqref{eq-super-central-extension-classification}"]
\end{tikzcd}
\end{equation}
where the middle equivalences are defined by $\omega^{1/2}$-shift \eqref{eq-super-theta-data-omega-shift}. Since the outer vertical functors are equivalent to the identity, the $2$-isomorphism rendering \eqref{eq-interior-automorphism-altogether} commutative defines an isomorphism of functors:
\begin{equation}
\label{eq-change-of-torus-isomorphism}
\alpha_g : \Phi_{(T_1, B_1)} \xrightarrow{\simeq} \Phi_{(T_2, B_2)}.
\end{equation}

\emph{Claim}: $\alpha_g$ depends only on the pairs $(T_1, B_1)$ and $(T_2, B_2)$ (as opposed to $g$). Indeed, any other choice of an $X$-point of $G$ conjugating $(T_2, B_2)$ into $(T_1, B_1)$ differs from $g$ by an $X$-point of $T_1$, so the claim is equivalent to the following assertion: for $(T_1, B_1) = (T_2, B_2) = (T, B)$, the isomorphism $\alpha_t$ is the identity on $\Phi_{(T, B)}$. However, this follows from the description of the $2$-isomorphisms in \eqref{eq-interior-automorphism-k-theory} and \eqref{eq-interior-automorphism-super-central-extension}.

Finally, we set \eqref{eq-killing-pair-comparison} to be the isomorphism \eqref{eq-change-of-torus-isomorphism} for any $X$-point $g$ conjugating $(T_2, B_2)$ into $(T_1, B_1)$, which exists \'etale locally over $X$---these choices glue thanks to the independence of $\alpha_g$ on $g$. The cocycle condition follows from the equality $\alpha_{g_1g_2} = \alpha_{g_1}\cdot\alpha_{g_2}$.
\end{proof}

\bibliographystyle{amsalpha}
\bibliography{../biblio_mathscinet.bib}

\end{document}